\documentclass[10pt]{amsart}
\usepackage{graphicx,enumitem,dsfont}
\usepackage[colorlinks]{hyperref}
\usepackage{subfigure}
\usepackage{pdfsync}
\usepackage{upref}
\usepackage{a4wide}

\usepackage{amscd}
\usepackage{amssymb}
\usepackage{amsmath,amssymb,amsthm}

\newcommand{\R}{\mathbb{R}}
\newcommand{\Z}{\mathbb{Z}}

\newcommand{\seq}[1]{\left\{#1\right\}}

\newcommand{\Dx}{{\Delta x}}
\newcommand{\Dt}{{\Delta t}}
\newcommand{\norm}[1]{\left\|#1\right\|}
\newcommand{\abs}[1]{\left|#1\right|}
\newcommand{\sgn}[1]{\mathrm{sign}\left(#1\right)}

\newcommand{\Dp}{D_{+}}

\newcommand{\Dpt}{D_+^t}

\newcommand{\eps}{\varepsilon}
\newcommand{\CL}{\mathcal{L}}
\newcommand{\N}{\mathbb{N}}
\newcommand{\CMloc}{\mathcal{M}_{\mathrm{loc}}}
\newcommand{\Dm}{D_-}

\newcommand{\loc}{{\mathrm{loc}}}

\newcommand{\Div}{\mathrm{div}}
\newcommand{\Curl}{\mathrm{curl}}
\newcommand{\st}{\;\bigm|\;}

\newcommand{\fhp}{\hat{f}^n_{j+\frac12}}
\newcommand{\fhm}{\hat{f}^n_{j-\frac12}}

\newcommand{\kp}{k_{j+\frac12}}
\newcommand{\km}{k_{j-\frac12}}

\newcommand{\hp}{h_{j+\frac12}}
\newcommand{\hm}{h_{j-\frac12}}

\newcommand{\xp}{x_{j+\frac12}}
\newcommand{\xm}{x_{j-\frac12}}

\newcommand{\bigO}[1]{\ensuremath{\mathop{}\mathopen{}\mathcal{O}\mathopen{}\left(#1\right)}}

\newtheorem{definition}{Definition}[section]

\newtheorem{theorem}{Theorem}[section]
\newtheorem{lemma}{Lemma}[section]

\newtheorem{remark}{Remark}[section]
\theoremstyle{definition} 

\newtheorem*{maintheorem*}{Main Theorem}

\allowdisplaybreaks

\numberwithin{equation}{section}
\numberwithin{figure}{section}
\numberwithin{table}{section}

\newcounter{asnr}

\newenvironment{Assumptions} %
{\ifnum\value{asnr}=0 \stepcounter{asnr} 
  \begin{enumerate}[label=\textbf{A}.\arabic{enumi}]
    \else
    \begin{enumerate}[label=\textbf{B}.\arabic{enumi},resume] \fi}
{\end{enumerate}}

\newcounter{defnr}

{\ifnum\value{defnr}=0 \stepcounter{defnr} 
  \begin{enumerate}[label=\textbf{D}.\arabic{enumi}]
    \else
    \begin{enumerate}[label=\textbf{D}.\arabic{enumi},resume] \fi}
{\end{enumerate}}

\title[Diffusive-Dispersive Approximation]{Convergence of Fully discrete schemes for diffusive dispersive conservation laws with \\discontinuous coefficient}

\author[R. Dutta]{Rajib Dutta} \address[Rajib
Dutta]{\newline 
    Institut f\"{u}r Mathematik,  \newline Julius-Maximilians-Universit\"{a}t W\"{u}rzburg,  
 \newline Campus Hubland Nord, Emil-Fischer-Strasse 30, \newline 97074, W\"{u}rzburg, Germany.}
\email[]{rajib.ami@gmail.com}

\author[U. Koley]{Ujjwal Koley} \address[Ujjwal Koley] {\newline  
   Centre For Applicable Mathematics (CAM)  \newline
   Tata Institute of Fundamental Research  
\newline P.O. Box 6503, GKVK post office \newline
Bangalore-560065, India.} 
\email[]{ujjwal@math.tifrbng.res.in}

\author[D. Ray]{Deep Ray} \address[Deep Ray] {\newline  
   Centre For Applicable Mathematics (CAM)  \newline
   Tata Institute of Fundamental Research  
\newline P.O. Box 6503, GKVK post office \newline
Bangalore-560065, India.} 
\email[]{deep@math.tifrbng.res.in}

\keywords{conservation laws; discontinuous flux; diffusive-dispersive approximation; finite difference scheme; convergence; entropy condition; nonclassical shock}

\date{\today}

\begin{document}

\begin{abstract}
We are concerned with fully-discrete schemes
for the numerical approximation of diffusive-dispersive hyperbolic conservation
laws with a discontinuous flux function in one-space dimension. 
More precisely, we show the convergence of approximate solutions,
generated by the scheme corresponding to vanishing diffusive-dispersive scalar conservation laws with a
discontinuous coefficient, to the corresponding scalar conservation law with
discontinuous coefficient.  Finally, the convergence is illustrated by
several examples. In particular, it is delineated that the limiting solutions generated
by the scheme need not coincide, depending on the relation 
between diffusion and the dispersion coefficients, with the classical Kru\v{z}kov-Ole\u{i}nik
entropy solutions, but contain nonclassical undercompressive shock waves.
\end{abstract}

\maketitle

%\tableofcontents

\section {Introduction}
In this paper, we consider a finite difference method 
for the vanishing diffusive-dispersive approximations of scalar 
conservation laws with a discontinuous flux 
\begin{equation}
\label{eq:system}
\begin{cases}
u^{\eps}_t + f\left(k(x),u^{\eps}\right)_x = \mathcal{R}\,[\eps, \mu(\eps); u^{\eps}], &\ \ x \in  \R \times (0,T),\\
u^{\eps}(x,0)=u_0(x), &\ \ x \in \R,
  \end{cases}
\end{equation}
when $\eps >0$ tends to zero with $0<\mu(\eps) \mapsto 0$ as $\eps \mapsto 0$.
Here $\mathcal{R}\,[\eps, \mu(\eps); u^{\eps}]$ is a regularization term,
depends upon two parameters
$\eps$ and $\mu(\eps)$ referred to as the diffusion and
the dispersion coefficients, motivated by 
the equations of two-phase flow in porous media,
$T>0$ is fixed, $u^{\eps}: \R \times
[0,T) \mapsto \R$ is the unknown scalar map, 
$u_0$ the initial data, $k: \R \mapsto \R$ is a spatially varying (discontinuous) coefficient, 
and the flux function $f : \R^2 \mapsto \R$
is a sufficiently smooth scalar function (see Section ~\ref{sec:math} for the complete list of assumptions).

Motivated by the dynamic capillary pressure \cite{hass}, we consider in this paper the simplified model
\begin{align}
\label{eq:imp1}
\mathcal{R}\,[\eps, \mu(\eps); u^{\eps}] = \eps \beta \,u^{\eps}_{xx} + \mu(\eps) \gamma \,u^{\eps}_{xxt}
\end{align}
with a third-order mixed derivatives term including one time derivative. Here $\beta, \gamma >0$ are fixed parameters. The equation \eqref{eq:system} along with \eqref{eq:imp1} serves as a concrete model of two phase flows in a heterogeneous porous medium.

Moreover, drawing preliminary motivation from phase transition dynamics, we also consider the following specific form 
of the regularization term
\begin{align}
\label{eq:imp2}
\mathcal{R}\,[\eps, \mu(\eps); u^{\eps}] = \eps \beta \,u^{\eps}_{xx} + \mu(\eps) \gamma \,u^{\eps}_{xxx},
\end{align}
with $\beta, \gamma >0$ fixed. 

Furthermore, for the simplicity in the exposition, we assume that the flux function has the following particular form 
\begin{align*}
f(k(x), u) = k(x) f(u).
\end{align*}
Note that the flux $k(x) f(u)$ has a possibly discontinuous spatial dependence through the
coefficient $k$, which is allowed to have jump discontinuities.

The scalar conservation laws
with a discontinuous flux function
\begin{align}
\label{eq:discont}
u_t + f(k(x), u)_x =0
\end{align}
is a special example of this type of problems,
corresponds to the case $\beta = \gamma=0$. A simple physical model corresponding to
\eqref{eq:discont} is the Witham model of car traffic flow on a highway 
(consult the monograph by Leveque \cite{lev}), where the
spatially varying coefficient $k$ corresponds to changing road conditions.
Several other models such as two phase
flow in a heterogeneous porous medium that arise in petroleum industry, 
the modeling of the clarifier thickener unit used 
in waste water treatment plants are also corresponding to \eqref{eq:discont}.

Independently of the smoothness of the initial data $u_0$ and $k$, 
solutions to \eqref{eq:discont} are not necessarily smooth due to the presence of nonlinear flux term in the equation \eqref{eq:discont}. Thus, weak solutions must be sought. 
\begin{definition}[Weak solution]
A weak solution of the initial value problem \eqref{eq:discont} is a bounded measurable function $u: \R \times [0,T) \rightarrow \R$ satisfying
\begin{equation}
\begin{aligned}
\label{eq:weak}
\int_{\R} \int_0^T \left(  \varphi_t u + \varphi_x k(x) f(u) \right) \,dx\,dt + \int_{\R} \varphi(x,0) u_0(x) \,dx =0,
\end{aligned}
\end{equation}
\end{definition}
for all $\varphi \in C_0^{\infty} (\R \times [0,T))$.

It is well known that (weak) solutions may be discontinuous and they
are not uniquely determined by their initial data. Consequently, an entropy condition must be
imposed to single out the physically correct solution. If 
 $k(x)$ is ``smooth", a weak solution $u$ satisfies 
the entropy condition if for all convex $C^2$ functions $\eta: \R \rightarrow \R$
\begin{align*}
\eta(u)_t + \left( k(x) Q(u) \right)_x + k'(x) \left(  \eta'(u) f(u) - Q(u) \right) \le 0, \,\, \text{in}\,\, \mathcal{D}(\R \times [0,T]),
\end{align*}
where $Q:\R \rightarrow \R$ is defined by $Q'(u) = \eta'(u) f'(u)$.

By standard limiting argument, this implies the Kru\v{z}kov-type entropy condition
\begin{align*}
\abs{u-c}_t + \sgn{u-c} \left( k(x) (f(u) -f(c)) \right)_x
+ \sgn{u-c} f(c) k'(x)  \le 0, \,\, \text{in}\,\, \mathcal{D}(\R \times [0,T]),
\end{align*}
holds for all $c \in \R$.

However, the notion of entropy solution described above breaks down when 
$k(x)$ is discontinuous. In view of \cite{kenneth1}, we use the following notion of entropy solution 
for (one-dimensional) conservation laws with discontinuous flux
equations with coefficients that are only spatially dependent. We assume that
the spatially varying coefficient $k(x)$ is piecewise $C^1$ with
finitely many jumps (in $k$ and $k'$), located at $\xi_1, \xi_2, \cdots, \xi_M$.
\begin{definition}[Entropy solution]
A weak solution $u$ of the initial value problem \eqref{eq:discont} is called
an entropy solution, if the following Kru\v{z}kov-type entropy inequality holds for all
$c \in \R$ and all test functions $0 \le \psi \in \mathcal{D}(\R \times [0,T]) $. 
\begin{equation}
\label{eq:ent}
\begin{aligned}
\int_{\R}& \int_0^T  \left( \abs{u-c} \psi_t + \sgn{ u-c}  k(x) \left(f(u) -f(c) \right) \psi_x \right) \,dx\,dt 
+ \int_{\R} \abs{u_0 -c} \psi(x,0)\,dx \\
& \quad+  \int_{\R \setminus \seq{\xi_m}_{m=1}^M } \int_0^T \sgn{u-c} k'(x) \,f(c)\,\psi \,dx\,dt 
+  \sum_{m=1}^{M} \int_0^T \abs{ f(c) (k_m^{+} - k_m^{-})} \psi(\xi_m, t) \,dt \ge 0.
\end{aligned}
\end{equation}
\end{definition}
The last couple of decades have witnessed remarkable advances in the studies
of conservation laws with discontinuous flux function. However, we will not be able
to discuss the whole literature here, but only refer to the parts that are pertinent to
the current paper. In case of ``smooth" $k(x)$, the notion of entropy solution was 
introduced independently by Kru\v{z}kov \cite{kruzkov} and Vol'pert \cite{Volpert}
(the latter author considered the smaller BV class). These authors also proved general existence,
uniqueness, and stability results for the entropy solution, see also Ole\u{i}nik \cite{Oleinik}
for similar results in the convex case $f_{uu} \ge 0$.

\subsection{Diffusive Dispersive Approximation }
It is well known that the conservation law \eqref{eq:discont} is derived by 
neglecting underlying small
scale effects such as diffusion, dispersion, capillarity etc., and
may admit physically relevant discontinuous
solutions containing shock waves 
(non-classical shock) that may depend on underlying small-scale mechanisms.
It has been successively recognized that a standard entropy 
inequality (due to Kru\v{z}kov, Ole\u{i}nik, and others) does
not suffice to single out such a physically relevant solution, and it is 
important to incorporate these small-scale
effects in the entropy condition. In other words, 
additional admissibility criteria (a kinetic relation) are required in order to characterize
these small-scale dependent non classical shock waves uniquely. 
In \cite{lefloch1}, the authors developed a framework for the existence and uniqueness of the
non-classical shock waves that arise as limits of diffusive-dispersive approximations.

Noting that the solutions to \eqref{eq:discont} can be different due to their explicit dependence on the underlying small scale effects, we focus on a concrete model of two phase flow in porous medium (for a brief derivation of this model consult \cite{sid}).
The relevant small scale effect is a dynamic capillary 
pressure term, that was introduced in \cite{hass}. Compared to the 
standard capillary pressure models \cite{aziz}, the addition of the 
new term resulted in a model that contain
higher-order mixed spatio-temporal derivatives (cf. \eqref{eq:imp1}).
 
The diffusive-dispersive model has a long tradition, starting with the analysis of 
linear diffusion-dispersion model \eqref{eq:system}. 
A pioneering study of the effect of vanishing diffusion and dispersion terms in
scalar conservation laws, with $x$-independent flux function, can be found in Schonbek \cite{schonbek}. The technique of compensated compactness was used to prove convergence toward weak solutions. 
Kondo and LeFloch \cite{kondo} studied zero diffusion-dispersion limits for $x$-independent fluxes under an optimal balance between the sizes
of the diffusion and dispersion parameters. LeFloch and Natalini \cite{lefloch2} used the concept of 
measure-valued solution and established convergence
results assuming that the diffusion dominates the dispersion. Subsequently, the approach of kinetic decomposition 
and velocity averaging \cite{perthame} was introduced
by Hwang and Tzavaras \cite{hwang} to analyze singular limits including nonclassical shock waves.
Furthermore, we also mention related works by Wu \cite{wu} and Jacobs, McKinney,
and Shearer \cite{jacobs} which provides the first
existence result of undercompressive shocks for the modified Korteweg-de Vries-Burgers equation.

It is well known that the relative scaling between $\eps$ and $\mu(\eps)$ 
determines the limiting behavior of solutions, and we can distinguish 
between three cases:
\begin{itemize}
\item Diffusion-dominant regime $\mu(\eps) < < \eps^2$: The 
qualitative behavior of solution is same as the solution of the 
conservation laws. 
\item Dispersion-dominant regime $\mu(\eps) > > \eps^2$: In this case, high
oscillations develop (as $\eps \downarrow 0$) especially in regions 
of steep gradients of the solutions and only weak convergence is observed.
\item Balanced diffusion-dispersion regime: This typically corresponds to the scenario where $\mu(\eps) = \mathcal{O} (\eps^2)$. Only mild oscillations are observed near
shocks, and the limit solution is a weak solution to conservation laws.
Most importantly, in this case, the solution exhibit non-classical 
behavior, as they contain undercompressive shocks. However, for the regime $ \mu(\eps)= \scalebox{1.1}{$\scriptstyle\mathcal{O}$}({\eps^2})$, the limit
solution coincides with the entropy solution determined by Kru\v{z}kov theory \cite{kruzkov}.
\end{itemize}

\subsection{Numerical Schemes}

It is well known that standard finite difference, finite volume and finite element methods
have been very successful in computing solutions to hyperbolic conservation
laws with discontinuous coefficients, including those containing shock waves. 
However, we mention that most of these well-established
methods are proven to be not good enough to capture nonclassical shock wave solutions 
numerically. This well known phenomena has been explained by many 
authers (Hou and LeFloch \cite{hou}, Hayes and LeFloch \cite{hayes}, and others) 
in terms of the equivalent equation associated with discrete schemes through a formal Taylor
expansion. The key idea behind capturing nonclassical shocks is 
to design finite difference schemes whose equivalent equation
matches, both, the diffusive and the dispersive terms (cf. \eqref{eq:imp2}) in the underlying model. 
However, these
schemes fail to approximate nonclassical solutions with large amplitude, especially strong
shocks due to lack of control on higher order error terms present in equivalent equation. A recent
work by Ernest et al. \cite{ernest} has overcome such problems 
by dominating higher order error terms in amplitude by the leading order terms 
of the equivalent equation.

In another paper by Chalons and Lefloch \cite{chalons}, the authors introduced a fully -discrete scheme for the numerical approximation of diffusive-dispersive hyperbolic conservation
laws (cf. \eqref{eq:system}-\eqref{eq:imp2}) in one-space dimension. 
An important feature of their scheme is that it satisfies a cell entropy inequality and, as a consequence,
the space integral of the entropy is a decreasing function of time. Moreover, they showed that the limiting solutions generated by the scheme contains nonclassical undercompressive shock waves.

On the other hand, there is a sparsity
of efficient numerical schemes for \eqref{eq:system}-\eqref{eq:imp1} available 
in the literature. In fact, to the best of our
knowledge, this is the first systematic attempt to construct a provably convergent 
numerical scheme for \eqref{eq:system}-\eqref{eq:imp1}.
Having said this, there are some numerical experiments available in the final section of
the recent paper by Coclite et al. \cite{sid} without rigorous analysis of the scheme.

\subsection{Scope and Outline of the Paper}
In view of the above discussion, it is fair to claim that there are no robust
and provably stable numerical schemes currently available to simulate the
vanishing capillarity approximations
of scalar conservation laws equation \eqref{eq:system}-\eqref{eq:imp1}. 
In this context, we consider a fully-discrete (in both space and time)
finite difference scheme for \eqref{eq:system}-\eqref{eq:imp1} which
is provably convergent and able to capture non classical shocks quite well.
Since diffusion-dispersion model for the conservation laws with discontinuous
flux has not been studied in detail, we analyze a fully-discrete scheme for \eqref{eq:system}-\eqref{eq:imp2} as well.
While there are several numerical
methods which perform well in practice, perhaps better than the one presented here, (see \cite{sid1} for a recent comparison of diferent numerical methods) we emphasize that
we prove the convergence of the schemes proposed in this paper. Here, we mention that
a detailed analysis of the scheme introduced by Ernest et al. \cite{ernest} 
is beyond the scope of this paper, and 
will be the topic of an upcoming paper.

To sum up, the schemes in the present paper have the following properties:
\begin{itemize}
\item [(a)] Approximate solutions for \eqref{eq:system}-\eqref{eq:imp1}, generated by the
scheme \eqref{eq:scheme}, converge to the unique entropy solution of \eqref{eq:discont}
as long as $ \mu(\Dx)= \scalebox{1.1}{$\scriptstyle\mathcal{O}$}({\Dx})$. 
A scheme (cf. \eqref{eq:scheme_a})
has been formulated for  \eqref{eq:system}-\eqref{eq:imp2} and the same techniques can be 
applied, mutatis mutandis,to prove convergence of approximate solutions
to the unique entropy solution of \eqref{eq:discont}.
\item [(b)] Approximate solutions for \eqref{eq:system}-\eqref{eq:imp1} have been  shown to converge to weak
solutions of \eqref{eq:discont}, when $\mu(\Dx) = \mathcal{O} (\Dx^2)$.
Moreover, we show numerically that the limiting solutions generated
by the schemes \eqref{eq:scheme} and \eqref{eq:scheme_a} 
contain nonclassical undercompressive shock waves.
\end{itemize}

The rest of the paper is organized as follows: In section~\ref{sec:math}, 
we present the mathematical framework 
used in this paper. In particular, we have used a compensated
compactness result in the spirit of Tartar \cite{Tartar} but the proof is based on div-curl lemma
and does not rely on the Young measure. Section~\ref{sec:semi} introduces 
the fully-discrete finite difference scheme for \eqref{eq:system}-\eqref{eq:imp1} . In section~\ref{sec:energy}, we derive a priori
estimates for the approximate solutions and a detailed convergence analysis towards weak solutions
of \eqref{eq:discont} has been discussed in section~\ref{sec:convergence}. Convergence 
towards the unique entropy solution has been considered in section~\ref{sec:entropy},
while a brief discussion on the results for diffusive-dispersive approximation 
\eqref{eq:system}-\eqref{eq:imp2} has been addressed in section~\ref{ap:dif_dis}.
Finally, numerical results are presented in section~\ref{sec:numerical} to illustrate the performance of the designed schemes.

\section{Mathematical Framework}
\label{sec:math}
In this section, we list all the assumptions on the data for the problem \eqref{eq:system}, and present relevant mathematical tools to be used in the subsequent analysis. 
Throughout this paper we use the letters $C,\,K$ etc. to denote various generic constants independent of approximation parameters, which may change line to line, but the notation is kept unchanged so long as it does not impact the central idea.

The basic assumptions on the data of the problem \eqref{eq:system} are as follows:

\begin{Assumptions}
\item \label{def:w1} For the initial function $u_0: \R \mapsto \R$, we assume that
\begin{align*}
u_0 \in L^2(\R) \cap L^{\infty}(\R), \quad a \le u_0(x) \le b, \,\, \text{for a.e}\,\,x \in \R;
\end{align*}
\item \label{def:w2} For the discontinuous coefficient $k:\R \mapsto \R$, we assume that
  \begin{align*}
  k \in L^{\infty}(\R) \cap BV_{\loc}(\R), \quad \alpha \le k(x) \le \beta, \,\, \text{for a.e}\,\,x \in \R;
  \end{align*}
\item \label{def:w3} Regarding the flux function $f:[\alpha, \beta] \times [a,b] \mapsto \R$, we assume that
\begin{align*}
u \mapsto f(k,u) &\in C^2([a,b]), \,\, \text{for all} \,\,k \in [\alpha, \beta], \\
k\mapsto f(k,u) &\in C^1([\alpha,\beta]),\,\, \text{for all} \,\,u \in [a,b];
\end{align*}
\item \label{def:w4} Furthermore, we assume that $u \mapsto f(k,u)$ is genuinely nonlinear a.e. in $\R \times [0,T]$, i.e., 
$
f_{uu}(k(x),u) \neq 0$, for a.e. $u\in [a,b]$.
\end{Assumptions}
\begin{remark}
It is worth mentioning that the Assumption ~\ref{def:w4} is typically required in the \emph{compensated compactness} framework. 
This condition also imposes a condition on the coefficient 
$k(x)$. In fact, it implies that $f(u)$ is genuinely nonlinear (i.e., $f'' \neq 0$) and 
$\abs{k(x)} \neq 0$, for a.e. $x\in \R$.
\end{remark}

Next, we recapitulate the results required from the \emph{compensated
compactness} method due to Murat and Tartar \cite{Murat,Tartar}. For a
nice overview of applications of the compensated compactness method to
hyperbolic conservation laws, we refer to Chen \cite{Chen}.  Let
$\mathcal{M}(\R)$ denote the space of bounded Radon measures on $\R$
and
\begin{align*}
  C_0(\R) = \seq{ \psi \in C(\R) \st \lim_{\abs{x} \rightarrow \infty}
  \psi(x) =0 }.
\end{align*}
If $\mu \in \mathcal{M}(\R)$, then
\begin{align*}
  \langle \mu, \psi \rangle = \int_{\R} \psi \,d \mu, \quad \text{for
    all} \quad \psi \in C_0(\R).
\end{align*}
Recall that $\mu \in \mathcal{M}(\R)$ if and only if $ \abs{\langle
  \mu, \psi \rangle} \le C \norm{\psi}_{L^{\infty}(\R)} $ for all
$\psi \in C_0(\R)$. We define the norm
\begin{align*}
  \norm{\mu}_{\mathcal{M}(\R)} := \sup \Big\{\abs{\langle \mu, \psi
      \rangle}: \psi \in C_0(\R), \norm{\psi}_{L^{\infty}(\R)} \le 1
    \Big\}.
\end{align*}
The space $\left( \mathcal{M}(\R), \norm{\cdot}_{\mathcal{M}(\R)}
\right)$ is a Banach space and it is isometrically isomorphic to the
dual space of $\left(C_0(\R), \norm{\cdot}_{L^{\infty}(\R)} \right)$. Furthermore, we define the space of probablity measures
\begin{align*}
  \text{Prob}(\R) := \Big\{ \mu \in \mathcal{M}(\R): \mu \, \text{is
    nonnegative and} \, \norm{\mu}_{\mathcal{M}(\R)}=1 \Big\}.
\end{align*}
Before we state the compensated compactness theorem, we shall recall the
celebrated div-curl lemma \cite{Murat}.
\begin{lemma}[div-curl lemma]
  \label{lem:div}
Let $\Omega$ be a bounded open subset of $\R^2$. Suppose 
\begin{align*}
u^1_{\Dx} \rightharpoonup \overline{u}^1, \quad u^2_{\Dx} \rightharpoonup \overline{u}^2,  \quad v^1_{\Dx} \rightharpoonup \overline{v}^1, \,\, \text{and} \,\, v^2_{\Dx} \rightharpoonup \overline{v}^2,
\end{align*}
in $L^2(\Omega)$ as $\Dx \downarrow 0$.  Furthermore, assume that the two sequences 
$\Big\{\Div\left(u^1_{\Dx}, u^2_{\Dx}\right) \Big\}_{\Dx>0}$ and $\Big\{\Curl\left(v^1_{\Dx}, v^2_{\Dx}\right) \Big\}_{\Dx>0}$ lie 
in a (common) compact subset of $H^{-1}_{\mathrm{loc}}(\Omega)$, where 
$\Div\left(u^1_{\Dx}, u^2_{\Dx}\right) = \partial_{x_1} u^1_{\Dx} +  \partial_{x_2} u^2_{\Dx}$
and $\Curl\left(v^1_{\Dx}, v^2_{\Dx}\right) = \partial_{x_1} v^2_{\Dx} -  \partial_{x_2} v^1_{\Dx}$. Then along a 
subsequence \vspace{0.1cm}
\begin{align*}
\left(u^1_{\Dx}, u^2_{\Dx}\right) \cdot \left(v^1_{\Dx}, v^2_{\Dx}\right) \mapsto \left(\overline{u}^1, \overline{u}^2\right) \cdot \left(\overline{v}^1, \overline{v}^2\right),
\,\, \text{in} \,\, \mathcal{D}'(\Omega), \,\,\text{as}\,\, \Dx \downarrow 0.
\end{align*}
\end{lemma}
Suitably modified for our purpose, we shall use the following compensated compactness result. 
For a proof, we refer to the paper by Kenneth and Towers \cite[Lemma 3.2]{towers}. 
\begin{theorem}
  \label{thm:compcomp} Assume that \ref{def:w2}, \ref{def:w3} and \ref{def:w4} hold. Let $\Omega\subset \R\times [0,T]$ be a bounded
  open set, and assume that $\seq{u_\Dx}$ is a sequence of uniformly
  bounded functions such that $\abs{u_\Dx}\le M$, for all $\Dx$. Set
  \begin{equation}
    \begin{aligned}
      \left(\eta_1(s),q_1(k,s)\right) &= \left(s-c,f(k,s)-f(k,c)\right),\\
      \left(\eta_2(k,s),q_2(k,s)\right) &= \left(f(k,s)-f(k,c), \int_c^s
        (f_s(k,\theta))^2\,d\theta \right),
    \end{aligned}\label{eq:entropies}
  \end{equation}
  where $c$ is an arbitrary constant. If the two sequences 
  \begin{equation*}
  \seq{\eta_1(u_{\Dx})_t + q_1(k(x), u_{\Dx})_x}_{\Dx>0}, \quad \text{and} \quad  \seq{\eta_2(k(x), u_{\Dx})_t + q_2(k(x), u_{\Dx})_x}_{\Dx>0}
  \end{equation*}
belong to a compact subset of $H^{-1}_{\mathrm{loc}}(\Omega)$, then
 there exists a subsequence of $\seq{u_{\Dx}}_{\Dx>0}$ that converges a.e. to a function $u \in L^{\infty}(\Omega)$.
\end{theorem}
We remark that, a feature of the compensated compactness result above is that it avoids
the use of the Young measure by following an approach developed by
Chen and Lu \cite{Chen, Lu} for the standard scalar conservation
law. This is preferable as the fundamental theorem of Young measures
applies most easily to functions that are continuous in all variables.

The following compactness interpolation result (known as Murat's lemma
\cite{Murat}) is useful in obtaining the $H^{-1}_{\loc}$ compactness
needed in Theorem ~\ref{thm:compcomp}.
\begin{lemma}
  \label{lem:Murat}
  Let $\Omega$ be a bounded open subset of $\R^2$.  Suppose that the
  sequence $\seq{\CL_\Dx}_{\Dx>0}$ of distributions is bounded in
  $W^{-1,\infty}(\Omega)$.  Suppose also that
  $$
  \CL_\Dx=\CL_{1,\Dx} + \CL_{2,\Dx},
  $$
  where $\seq{\CL_{1,\Dx}}_{\Dx>0}$ is in a compact subset of
  $H^{-1}_{\mathrm{loc}}(\Omega)$ and $\seq{\CL_{2,\Dx}}_{\Dx>0}$ is in a bounded
  subset of $\CMloc(\Omega)$.  Then $\seq{\CL_\Dx}_{\Dx>0}$ is in a
  compact subset of $H^{-1}_{\mathrm{loc}}(\Omega)$.
\end{lemma}

\section{A Fully-Discrete Finite Difference Scheme}
\label{sec:semi}
We begin by introducing some notation needed to define the
fully-discrete finite difference scheme. Throughout this paper, we
reserve $\Dx, \Dt$ to denote small positive numbers that represent the
spatial and temporal discretizations parameter of the numerical scheme. Given
$\Dx>0$, we set $x_j=j\Dx$ for $j\in \Z$, to denote the spatial mesh points. Similarly,  we set $t^n = n \Dt$ for $n= 0,1,\cdots,N$, where $N\Dt=T$
for some fixed time horizon $T>0$. Moreover, for any function $u =
u(x,t)$ admitting point values, we write $u^n_j = u(x_j, t^n)$. Furthermore, let
us introduce the spatial and spatial-temporal grid cells
\begin{align*}
  I_j = [x_{j-1/2}, x_{j+1/2}), \qquad I^n_j = [x_{j-1/2}, x_{j+1/2}) \times [t^n, t^{n+1}).
\end{align*}
where $x_{j\pm1/2} = x_j \pm \Dx/2$. Let $D_{\pm}$ denote the discrete
forward and backward differences in space, i.e.,
\begin{equation*}
  D_{\pm}u_j = \pm \frac{u_{j\pm 1} - u_j}{\Dx},
\end{equation*}
The discrete Leibnitz rule is given by
\begin{align*}
  D_{\pm} (u_j v_j) = u_j D_{\pm} v_j + v_{j \pm 1} D_{\pm} u_j
\end{align*}
while the summation-by-parts formula is given by
\begin{align*}
\sum \limits_{j\in \Z} u_j D_{\pm} v_j = -  \sum \limits_{j\in \Z} v_j D_{\mp} u_j.
\end{align*}
Furthermore, for any $C^2$ function $f$, using the Taylor expansion on
the sequence $f(u_j)$ we obtain
\begin{align*}
  D_{\pm} f(u_j) = f'(u_j) D_{\pm} u_j \pm \frac{\Dx}{2} f''(\xi_{j
    \pm \frac{1}{2}}) (D_{\pm} u_j)^2,
\end{align*}
for some $\xi_{j\pm \frac{1}{2}}$ between $u_{j \pm 1}$ and $u_j$. In other words, the discrete chain
rule is accurate up to an error term of order $\Dx (D_{\pm} u_j)^2$.  

Finally, let $D^t_{\pm}$ denote the discrete forward  and backward difference operator in the time, i.e.,
\begin{equation*}
  D^t_{\pm}u^n_j = \mp \frac{u^{n\pm1}_j - u^n_j}{\Dt}.
\end{equation*}
The following identity is readily verified:
\begin{equation}
\label{eq:iden}
\begin{aligned}
u^n_j D^{t}_{+} u^n_j = \frac{1}{2} D^{t}_{+} (u^n_j)^2 - \frac{\Dt}{2} (D^{t}_{+} u^n_j)^2.
\end{aligned}
\end{equation}
We propose the following fully-discrete (in space and time) finite difference scheme 
approximating the limiting solutions generated by the equation \eqref{eq:system}-\eqref{eq:imp1}

\begin{align}
D^t_{+} u^n_j + \Dm \hp^n &= \beta \Dx \Dp\Dm u^n_j  + \gamma \mu(\Dx) D^t_{+} \Dp \Dm u^n_j, \quad j \in \Z, \, n \in \N_0, \label{eq:scheme}\\
u^0_j &=\frac{1}{\Dx}\int_{\xm}^{\xp}u_0(\theta) d\theta,\quad j \in \Z, \label{eq:scheme_initial}
\end{align}
where $\beta, \gamma >0$ are fixed parameters, and $\mu(\Dx) \mapsto 0$ as $\Dx \mapsto 0$. 
More specifically, we will either use $\mu(\Dx)=\mathcal{O}(\Dx^2)$ 
or $ \mu(\Dx)= \scalebox{1.1}{$\scriptstyle\mathcal{O}$}({\Dx^2})$ depending on the
quest for the convergence of approximate solution $u_{\Dx}$ towards a weak solution or the entropy solution, respectively.
\begin{remark}
Here we used the notation $\mathcal{O}(\Dx)$ to denote quantities that depend 
on $\Dx$ and are bounded above by $C \Dx$, where $C$ is a constant independent 
of $\Dx$. Likewise, we used the notation 
$ \mu(\Dx)= \scalebox{1.1}{$\scriptstyle\mathcal{O}$}({\Dx})$ to denote quantities 
that depend on $\Dx$ and are bounded above by $C \Dx^{\alpha}$, where $C$
 is a constant independent of $\Dx$ and $\alpha >1$.
\end{remark}

The numerical flux corresponding to the flux function $k(x) f(u)$ is given by
\begin{align*}
\hp^n=\kp \hat{f}^n_{j+\frac{1}{2}}, \quad \text{with} \,\,k_{j+\frac12} = \frac{1}{\Dx} \int_{x_j}^{x_{j+1}} k(x) \,dx,
\end{align*}
where $\hat{f}^n_{j+\frac{1}{2}}:=\hat{f}^n(u_j, u_{j+1})$ is based on a two-point 
monotone numerical flux, i.e., non-decreasing with respect to the first argument 
and non-increasing with respect to the second argument, and consistent with the actual flux,
i.e., $\hat{f}(u,u) = f(u)$. Moreover, in order to maintain monotonicity of the 
scheme \eqref{eq:scheme} without the higher order terms 
(corresponds to $\beta=\gamma=0$), the arguments of the numerical
flux are transposed when the coefficient $k$ is negative.
More specifically, we choose
\begin{align*}
\hat{f}^n_{j+\frac12}=\begin{cases}\hat{f}(u^n_{j}, u^n_{j+1}), \, & \text{if } \kp\ge0\\
\hat{f}(u^n_{j+1}, u^n_{j}), \, & \text{if } \kp<0.
\end{cases}
\end{align*}
Summing up, the numerical flux $\hp^n$ is given by
\begin{align*}
\hp^n=\begin{cases}\kp \, \hat{f}(u^n_{j}, u^n_{j+1}), \, & \text{if } \kp \ge0\\
\kp \,\hat{f}(u^n_{j+1}, u^n_{j}), \, & \text{if } \kp<0.
\end{cases}
\end{align*}
In particular, we focus on Engquist-Osher (EO) numerical flux given by 
\begin{align}
\hat{f}(u,v) = \frac12( f(u) + f(v)) -\frac12 \int_{u}^{v} \abs{f'(s)}\,ds. \label{eq:EO}
\end{align}
\begin{remark}\label{rem:EO}
We have chosen to analyse the scheme \eqref{eq:scheme}--\eqref{eq:scheme_initial} 
with EO flux because of its apparent simplicity.
One can, however, adopt the method
of proof developed in this paper and obtain similar results for other schemes (e.g., all monotone schemes).
\end{remark}
To this end, observe that the EO flux given by \eqref{eq:EO} is Lipschitz continuous . In fact, for $f \in C^1$, it has 
continuous partial derivatives satisfying 
\begin{align}
f'_{-}(v) = \hat{f}_{v}(v,u)   \le 0 \le  \hat{f}_{u}(v,u) = f'_{+}(u), \label{eq:Lip}
\end{align}
using the conventional notations that $a_{-} = \min{(a,0)}$ and $a_{+} = \max{(a,0)}$.
It is also clear that $\norm{f'}_{\infty}$ serves as a Lipschitz constant for EO flux.

For a given initial data $u_0$, we define the initial grid function $\{u^0_j\}_{j\in\Z}$ by \eqref{eq:scheme_initial}. Moreover, for the sequence $\seq{u^n_j}_{j\in \Z, n\in \N_0}$, we associate the function $u_\Dx$
defined by
\begin{equation*}
  u_{\Dx}(x,t)= \sum_{j \in \Z, n\ge 0} u^n_j \, \mathds{1}_{I^n_j}(x,t),
\end{equation*}
where $\mathds{1}_{A}$ denotes the characteristic function of the set
$A$. Similarly, for the coefficient $k$ approximated at each cell boundary,  we associate the function $k_\Dx$ defined by
\begin{equation*}
k_{\Dx}(x)= \sum_{j \in \Z} k_{j+\frac12} \, \mathds{1}_{I_{j+\frac12}}(x)
\end{equation*}
Note that $k$ and $u$ are discretized on grids
that are staggered with respect to each other. This indeed results in a reduction in complexity,
compared with the approach where two discretizations are aligned.

Throughout this paper, we use the notation $u_{\Dx}$ to
denote the functions associated with the sequence $\seq{u^n_j}_{j\in\Z, n\in \N_0}$. For later use, recall that the discrete
$\ell^{\infty}(\R)$, $\ell^1(\R)$ and $\ell^2(\R)$ norms, and BV semi-norm for a lattice function $u_{\Dx}$ are defined
respectively as
\begin{equation*}
  \begin{aligned}
    & \norm{u_{\Dx}(\cdot,t^n)}_{\ell^{\infty}(\R)} = \sup_{j \in \Z} \abs{u^n_j}, \quad
     \norm{u_\Dx(\cdot, t^n)}_{\ell^1(\R)} = \Dx \sum_{j\in\Z} \abs{u^n_j}, \\
    & \norm{u_\Dx(\cdot, t^n)}_{\ell^2(\R)} = \sqrt{\Dx \sum_{j\in\Z} \abs{u^n_j}^2}, \quad
    \abs{u_\Dx(\cdot, t^n)}_{BV} = \sum_{j\in\Z} \abs{u^n_{j+1}- u^n_j}
  \end{aligned}
\end{equation*}
For the sake of simplified notations, unless specified, we shall use the notation $\norm{\cdot}$ to denote the discrete $\ell^2(\R)$ norm.

% end of this section

%beginning of new section

\section{A Priori Estimates}
\label{sec:energy}
This section is devoted to the derivation of a priori estimates which turns out to be useful to prove 
``strong compactness'' of the approximate solution $u_{\Dx}$.  
To begin with, following Coclite et al. \cite{sid}, we assume 
that the approximate solutions generated by the scheme are uniformly bounded, i.e., $u_{\Dx} \in L^{\infty}(\R \times [0,T])$. In other words, we assume that 
\begin{Assumptions}
\item For almost every $(x,t) \in \R \times [0,T]$,  $a' \le u_{\Dx} \le b'$, for some fixed constants $a', b' \in \R$;\label{def:b1}
\end{Assumptions}
\begin{remark}
It is worth mentioning that the above assumption on the approximate solutions is the manifestation 
of the specific structure of the flux function (depends explicitly on space variable). In fact, to obtain $L^2$ bound on the solution, one requires a priori $L^{\infty}$ bound on the solution (cf. \eqref{test10}).  
This assumption can be toppled by replacing the ``space dependent flux function'' to a flux function
which depends explicitly only on the solution. In such a scenario, one can use $L^p$ framework of the
compensated compactness result \cite{diperna}, to reproduce all the results in this paper.
\end{remark}

To proceed further, we first collect all the available estimates on the approximate solutions in the following lemma.
%For the ease of notations, hereinafter, we derive estimates for the scheme \eqref{eq:scheme}
%with $\beta = \gamma =1$.
\begin{lemma}
\label{lemma2}
Let $u_{\Dx}$ be a sequence of approximations generated by the scheme 
\eqref{eq:scheme}. Moreover, assume that the initial data $u_0$ lies in $L^2(\R)$. Then  
the following estimate holds
\begin{equation}
\label{eq:std}
\begin{aligned}
\frac{1}{2} D^t_{+} \norm{u^n}^2 + \left(\frac{\gamma \mu(\Dx)}{2} + \frac{\beta \Dx^2}{2} \right) &D^t_{+} \norm{\Dm u^n}^2 + \delta \Dx \norm{\Dm u^n}^2 \\
& \qquad  \qquad + \delta \Dx \norm{D^t_{+} u^n}^2 + \delta \gamma  \Dx \mu(\Dx) \norm{D^t_{+} \Dm u^n}^2 \le C,
\end{aligned}
\end{equation}
provided $\Dt$ and $\Dx$ satisfies the following CFL condition 
\begin{align}
\label{eq:cfl}
\max{ \Bigg \{2 \max{\Big \{\norm{f}_{\infty},\norm{f'}_{\infty},\norm{k}_{\infty} \Big\}} + \frac{\lambda}{2},   \,\,
\frac{\lambda}{2} \left( 1+ \frac{\beta \, \Dx^2}{ \gamma \, \mu(\Dx)} \right) \Bigg \}} \le \min{(1 -\delta,\beta-\delta)},
\end{align}
with $\delta \in (0,\min{(1,\beta)})$. Here $\lambda= \Dt/\Dx$ and the constant $C>0$ is independent of $\Dx$. 

In particular, the estimate \eqref{eq:std} guarantees following space-time estimates:
\begin{subequations}
\begin{align}
\forall n \in \N, \quad \Dx \sum_{j} (u^n_j)^2 & \le C,  \label{eq:est_final}\\
 \Dx^2\Dt \sum_{j} \sum_{n} (\Dm u^n_j)^2  &\le C, \label{eq:est_final1}\\
 \Dx^2 \Dt \sum_{j} \sum_{n} (D^t_{+} u^n_j)^2 &\le C, \label{eq:est_final2}\\
 \Dt \Dx^2 \mu(\Dx) \sum_{j} \sum_{n} (D^t_{+}\Dm u^n_j)^2 &\le C.\label{eq:est_final3}
\end{align}
\end{subequations}
\end{lemma}
\begin{remark}\label{rem:mu}
In light of the CFL condition \eqref{eq:cfl}, we want to emphasize that if $\mu(\Dx) = \mathcal{O}(\Dx^2)$(required 
to prove convergence towards a weak solution, cf. Theorem~\ref{thm:theorem1}), then we 
need $\zeta:=\frac{\lambda (\Dx)^2}{\mu(\Dx)}=  \frac{\Dt}{\Dx}$ to 
be kept fixed. On the other hand, if $ \mu(\Dx)= \scalebox{1.1}{$\scriptstyle\mathcal{O}$}({\Dx^2})$ 
(required to prove convergence towards the entropy solution, cf. Theorem~\ref{thm:theorem2}), then 
we need $\zeta = \frac{\Dt}{(\Dx)^{1+\alpha}}$, with $\alpha>0$, to be kept fixed. 
To sum up, we need a stronger CFL condition 
to prove convergence of approximate solutions towards 
the unique entropy solution of \eqref{eq:system}. 
To this end, we mention that in the subsequent analysis the 
CFL condition \eqref{eq:cfl} is always assumed to hold.
\end{remark}
\begin{proof}
To start with, we multiply the scheme \eqref{eq:scheme} by $\Dx \,u^n_j$ 
and subsequently sum over $j \in \Z$. Then, using summation-by-parts formula and the identity \eqref{eq:iden}, we obtain
\begin{align*}
\frac12 D^t_{+} \sum_j \Dx (u^n_j)^2  &- \frac{\Dt}{2} \sum_j \Dx \left( D^t_{+} u^n_j \right)^2\,  + \, \underbrace{\Dx \sum_ju^n_j \Dm\hp^n}_{\mathcal{I}_{\Dx}(f)} \\
& \qquad \qquad \quad  = \, - \beta \Dx \sum_j \Dx \abs{\Dm u_j}^2 \, -  \gamma \mu(\Dx) \Dx \sum_j \Dm u^n_j  D^t_{+} (\Dm u^n_j).
\end{align*}
Note that, the identity \eqref{eq:iden} also implies that
\begin{align*}
 \mu(\Dx) \Dx \sum_j  \Dm u^n_j  D^t_{+} (\Dm u^n_j) &= \frac{ \Dx  \mu(\Dx)}{2} \sum_j D^t_{+} (\Dm u^n_j)^2 
- \frac{\Dx  \mu(\Dx)\Dt}{2} \sum_j \left(  D^t_{+} \Dm u^n_j \right)^2.
\end{align*}
Next, we move on to estimate the term $\mathcal{I}_{\Dx}(f)$. Using summation-by-parts formula, we obtain
\begin{align*}
-\mathcal{I}_{\Dx}(f) &= -\sum_j \Dx \, u^n_j \Dp\hm^n = \sum_j (u^n_j-u^n_{j-1}) \km \hat{f}^n_{j-\frac12} \\
& = \sum_j \km \left[ (u^n_j-u^n_{j-1})\hat{f}^n_{j-\frac12} \,  - \, (F(u^n_j)-F(u^n_{j-1}))\right] 
 +  \left(F(u^n_j)-F(u^n_{j-1})\right) \\
& = \sum_j \underbrace{\km \left[ (u^n_j-u^n_{j-1})\hat{f}^n_{j-\frac12} \,  - \, (F(u^n_j)-F(u^n_{j-1}))\right] }_{\mathcal{E}_{j,n}(f)} 
 -  \sum_j F(u^n_j) \left(\kp -\km \right), 
\end{align*}
where $F$ is the primitive of $f$, i.e., $F'=f$. A first order Taylor's expansion together with the monotonicity of the numerical flux function $\hat{f}(u^n_j, u^n_{j+1})$ gives us an estimate of $\mathcal{E}_{j,n}(f)$. To see this, notice that
\begin{align*}
\mathcal{E}_{j,n}(f)= \km \, (u^n_j-u^n_{j-1}) \big( \hat{f}^n_{j-\frac12}  - f(u^n_{j-\frac12})\big),
\end{align*}
where $u^n_{j-\frac12}$ lies in between $u^n_j$ and $u^n_{j-1}$. To proceed further, we consider the following two cases:\\
{\bf Case 1:} Assume that $\km \ge 0$, then by the definition of numerical flux 
$\hat{f}^n_{j-\frac12}=\hat{f}(u^n_{j-1},u^n_j)$. If $u^n_j\leq u^n_{j-\frac12}\leq u^n_{j-1}$, then 
\begin{align*}
\hat{f}^n_{j-\frac12} \geq \hat{f}(u^n_{j-\frac12}, u^n_j)\geq \hat{f}(u^n_{j-\frac12},u^n_{j-\frac12})=f(u^n_{j-\frac12}). 
\end{align*}
On the other hand, if $u^n_{j-1}\leq u^n_{j-\frac12}\leq u^n_{j}$, then
\begin{align*}
\hat{f}^n_{j-\frac12} \leq \hat{f}(u^n_{j-\frac12}, u^n_j)\leq \hat{f}(u^n_{j-\frac12},u^n_{j-\frac12})=f(u^n_{j-\frac12}). 
\end{align*}
Thus, in any case we have $\mathcal{E}_{j,n} (f) \leq 0$.\\
{\bf Case 2:} Assume that $\km<0$, i.e., $\hat{f}^n_{j-\frac12}= \hat{f}(u^n_j, u^n_{j-1})$. If $u^n_j\leq u^n_{j-\frac12}\leq u^n_{j-1}$, then 
\begin{align*}
\hat{f}^n_{j-\frac12}\leq \hat{f}(u^n_{j-\frac12}, u^n_{j-1})\leq \hat{f}(u^n_{j-\frac12}, u^n_{j-\frac12})=f(u^n_{j-\frac12}).
\end{align*}
On the other hand, if $u^n_{j-1}\leq u^n_{j-\frac12}\leq u^n_{j}$, then
\begin{align*}
\hat{f}^n_{j-\frac12}\geq \hat{f}(u^n_{j-\frac12}, u^n_{j-1})\geq \hat{f}(u^n_{j-\frac12}, u^n_{j-\frac12})=f(u^n_{j-\frac12}).
\end{align*}
Therefore, in this case also, we have $\mathcal{E}_{j,n}(f) \leq 0$.
Having this in mind and making use of the Assumption~\ref{def:b1}, we conclude that
\begin{align}
\label{test10}
-\mathcal{I}_{\Dx}(f)
\leq \abs{-\sum_j \left(  \km-\kp  \right)F(u^n_j)}
\leq\text{BV}(k) \, \max_j \abs{F(u^n_j)}
\leq  C,
\end{align}
where $C$ is a constant independent of $\Dx$. Finally,
combining all the above estimates, we obtain
\begin{equation}
\begin{aligned}
\label{eq:est1}
\frac{1}{2} D^t_{+} \norm{u^n}^2 - \frac{\Dt}{2} \norm{D^t_{+} u^n}^2 &+ \frac{\gamma \mu(\Dx)}{2} D^t_{+} \norm{\Dm u^n}^2  \\
& \qquad \qquad \quad - \frac{\gamma \mu(\Dx) \Dt}{2} \norm{D^t_{+} \Dm u^n}^2 + \beta \Dx \norm{\Dm u^n}^2 \leq C. 
\end{aligned}
\end{equation}
Next, we multiply the scheme \eqref{eq:scheme} by $\Dx^2\, D^t_{+} u^n_j$ and sum over $j \in \Z$ to obtain, 
\begin{align}
\label{11}
\Dx^2 \sum_j (D^t_{+} u^n_j)^2 \, &+ \, \Dx^2 \sum_j\Dm \hp^n D^t_{+} u^n_j \\
& \qquad \qquad = - \beta (\Dx)^3 \sum_j \Dm u^n_j  D^t_{+} (\Dm u^n_j)  - \gamma \mu(\Dx)\Dx^2 \sum_j \left(D^t_{+} \Dm u^n_j \right)^2.
\end{align}
Again, use of the identity \eqref{eq:iden} reveals that
\begin{align*}
\beta (\Dx)^3 \sum_j  \Dm u^n_j  D^t_{+} (\Dm u^n_j) = \frac{\beta \Dx^3}{2} \sum_j D^t_{+} (\Dm u^n_j)^2 - \frac{\beta \Dx^3 \Dt}{2} \sum_j \left(  D^t_{+} \Dm u^n_j \right)^2.
\end{align*}
Next, considering the term which involves flux function, we see that
\begin{align*}
 & \Dx^2 \sum_j \Dm \hp D^t_{+} u^n_j
=\Dx \sum_j \left( \kp \hat{f}^n_{j+\frac12}-\km \hat{f}^n_{j-\frac12} \right) D^t_{+} u^n_j \\
&\qquad \qquad =  \Dx\sum_j \left( \kp-\km\right) \hat{f}^n_{j+\frac12} D^t_{+} u^n_j \, + \,  \Dx\sum_j \km\left(\hat{f}^n_{j+\frac12} -\hat{f}^n_{j-\frac12}\right) D^t_{+} u^n_j.
\end{align*}
Recall that we have chosen to work with a specific monotone flux, i.e., Engquist Osher flux. 
Since EO flux is Lipschitz continuous with Lipschitz constant $\norm{f'}_{\infty}$ (c.f. \eqref{eq:Lip}), we conclude that 
\begin{align*}
\displaystyle{\sup_{j} \sup_{n} \abs{\hat{f}^n_{j+\frac12}}}\leq \norm{f}_{\infty}, \,\, \text{and} \,\,
\abs{\hat{f}^n_{j+\frac12}-\hat{f}^n_{j-\frac12}} \leq \norm{f'}_{\infty} \left( \abs{u^n_j-u^n_{j-1}} + \abs{u^n_{j+1}-u^n_j} \right).
\end{align*}
An application of Young's inequality: For any two real numbers $a$ and $b$ and for every $\eps>0$
\begin{align*}
ab \le \eps a^2 + \frac{b^2}{4 \eps}
\end{align*}
leads to the estimates
\begin{align*}
\Dx\sum_j \left( \kp-\km\right) \hat{f}^n_{j+\frac12} D^t_{+} u^n_j  & \leq \Dx \norm{f}_{\infty} \sum_j \abs{\left( \kp-\km\right)} \abs{ D^t_{+} u^n_j} \\
& \leq \eps \Dx \norm{f}_{\infty} \norm{D^t_{+} u^n}^2 + \frac{\norm{f}_{\infty} \Dx}{4 \eps}  \norm{ \Dp k}^2
\end{align*}
and
\begin{align*}
 \Dx\sum_j \km\left(\hat{f}^n_{j+\frac12} -\hat{f}^n_{j-\frac12}\right) D^t_{+} u^n_j & \leq \Dx^2 \norm{f^\prime}_{\infty} \norm{k}_{\infty} \sum_j \left( \abs{D_- u_j^n} \, + \abs{D_-u_{j+1}^n} \right) \abs{D_+^t u_j^n } \\
& \leq 2 \eps_1 \Dx \norm{f^\prime}_{\infty} \norm{k}_{\infty} \norm{D^t_{+} u^n}^2 + \frac{\norm{f^\prime}_{\infty} \norm{k}_{\infty} \,\Dx }{2 \eps_1}  \norm{ \Dm u^n}^2
\end{align*}
For notational simplification, we define $M :=\max{\Big \{\norm{f}_{\infty},\norm{f'}_{\infty},\norm{k}_{\infty} \Big\}}$. Combining all above estimates, we arrive
\begin{equation}
\begin{aligned}
\label{eq:est2}
\Dx \norm{D^t_{+} u^n}^2 + \gamma \Dx \mu(\Dx) & \norm{D^t_{+} \Dm u^n}^2 + \frac{\beta \Dx^2}{2} D^t_{+}  \norm{\Dm u^n}^2 \\
& - \frac{\beta \Dx^2 \Dt}{2}  \norm{D^t_{+} \Dm u^n}^2 
 \le \eps \Dx M \norm{D^t_{+} u^n}^2 + \frac{M \Dx}{4 \eps}  \norm{ \Dp k}^2 \\
&\qquad \qquad \qquad \qquad +  2 \eps_1 \Dx M^2 \norm{D^t_{+} u^n}^2 + \frac{M^2 \,\Dx }{2 \eps_1}  \norm{ \Dm u^n}^2.
\end{aligned}
\end{equation}
Finally, adding \eqref{eq:est1} and \eqref{eq:est2} yields
\begin{equation}
\label{eq:sum1}
\begin{aligned}
\frac{1}{2} D^t_{+} \norm{u^n}^2  + \left(\frac{\gamma \mu(\Dx)}{2} + \frac{\beta \Dx^2}{2} \right) & D^t_{+} \norm{\Dm u^n}^2 + \Dx \left(  \beta -\frac{M^2}{ 2\eps_1} \right) \norm{\Dm u^n}^2 \\
& \quad + \Dx \left( 1 - M \eps - 2 M^2 \eps_1 - \frac{\lambda}{2}  \right)  \norm{D^t_{+} u^n}^2 \\
& \qquad \quad + \gamma \Dx \mu(\Dx) \left(1-\frac{\lambda}{2} - \frac{\beta \lambda \Dx^2}{2 \gamma \mu(\Dx)}\right) \norm{D^t_{+} \Dm u^n}^2 \le C,
\end{aligned}
\end{equation}
where $\lambda = \frac{\Dt}{\Dx}$.

We must now use the CFL condition \eqref{eq:cfl} to conclude that for some $\delta \in (0,1)$
\begin{align}\label{eq:cfl_conds}
1-\frac{\lambda}{2} - \frac{\beta \lambda \Dx^2}{2 \gamma \mu(\Dx)} > \delta, \,\, \beta -\frac{M^2}{ 2\eps_1} > \delta, \,\,\text{and} \,\,
 1 - M \eps - 2 M^2 \eps_1 - \frac{\lambda}{2}  > \delta.
\end{align}
Note that $\beta > \delta$ must hold. Furthermore, choosing $\eps= 1, \eps_1=0.5$, the third contraint in \eqref{eq:cfl_conds} can be written as
\[
M + M^2 + \frac{\lambda}{2} < 1 - \delta
\]
which would require $M<1$. Assuming $M$ is small enough (this can be done upto rescaling) to ensure the existence of a $\delta \in (0,1)$, such that $2 M + \frac{\lambda}{2}  < \min{(1 - \delta, \beta - \delta)}$. This in turn would imply
\begin{align*}
M^2  + M  +  \frac{\lambda}{2} &< 2M + \frac{\lambda}{2} < 1-\delta,\\
\beta - M^2 & > \beta - 2M >  \delta
\end{align*}
In other words, we get the last two conditions of \eqref{eq:cfl_conds}. Thus, the CFL condition \eqref{eq:cfl} ensures all the conditions of \eqref{eq:cfl_conds} are satisfied. Consequently, we see that estimate \eqref{eq:std} holds. 

In order to prove estimates \eqref{eq:est_final}, \eqref{eq:est_final1}, \eqref{eq:est_final2}, and \eqref{eq:est_final3}, 
we multiply the inequality \eqref{eq:sum1} by $\Dt$ and subsequently sum over all $n=0,1,\cdots, N-1$ to reach 
\begin{equation*}
\label{eq:st}
\begin{aligned}
\frac{1}{2} \norm{u^N}^2 &+ \left(\frac{\gamma \mu(\Dx)}{2} + \frac{\beta \Dx^2}{2} \right) \norm{\Dm u^N}^2 + \delta \Dx \Dt \sum_{n}\norm{\Dm u^n}^2 \\
& \qquad \qquad \qquad \quad + \delta \Dx\Dt \sum_{n} \norm{D^t_{+} u^n}^2 + \delta \gamma  \Dx \mu(\Dx) \Dt \sum_{n}\norm{D^t_{+} \Dm u^n}^2 \le \mathcal{F}(u_0),
\end{aligned}
\end{equation*}
where 
\begin{align*}
\mathcal{F}(u_0)&:= \frac{1}{2} \norm{u_0}^2 + \left(\frac{\gamma \mu(\Dx)}{2} + \frac{\beta \Dx^2}{2} \right) \norm{\Dm u_{0}}^2 \\
& \hspace{4.5cm} \le \frac12 \left( \norm{u_0}^2 +  \left(\frac{\gamma \mu(\Dx)}{2} + \frac{\beta \Dx^2}{2} \right) \frac{2}{\Dx^2} \norm{u_0}^2 \right) \le C.
\end{align*}
This essentially finishes the proof of the lemma.

\end{proof}

%end of this section

%beginning of new section 

\section{Convergence Analysis}
\label{sec:convergence}
Having obtained all the necessary a priori bounds in the previous section, we are ready to prove that the 
approximate solutions generated by the scheme \eqref{eq:scheme}
converge strongly to a weak solution of \eqref{eq:discont}, at least along a subsequence. The general
strategy of the convergence proof is in the spirit of the one used by DiPerna \cite{diperna}, and it has been
used in various contexts by several different authors. However, as we pointed out
earlier, we shall use a simplified compensated compactness result in the spirit of \cite[Lemma 3.2]{towers}.
In what follows, using a priori estimates derived in Section ~\ref{sec:energy}, 
we first demonstrate the desired $H^{-1}_{\loc}$
compactness of $\seq{u_{\Dx}}_{\Dx>0}$. Then, an
application of the compensated compactness Theorem ~\ref{thm:compcomp} gives the desired 
strong convergence in $L_{\mathrm{loc}}^p$ for any $p < \infty$. To achieve our goal, we start with the following crucial lemma:

\begin{lemma}[$H^{-1}_{\mathrm{loc}}$ compactness]
  \label{lem:compact}
  The sequence
  \begin{align*}
    \seq{\eta_i( k(x), u_\Dx)_t + q_i(k(x), u_\Dx)_x}_{\Dx>0} \,\, \text{is compact
      in} \,\, H^{-1}_{\loc}(\R \times [0,T]),
  \end{align*}
  where $\eta_i$ and $q_i$ are given by \eqref{eq:entropies}, for $i
  =1,2$.
\end{lemma}
\begin{proof}
To begin with, let us assume that $\eta = \eta_i$, for $i=1$ or $i=2$, and $\varphi$ be a test function with compact support such that $\varphi(x)=0$, for all
$|x|>|x_{J+\frac12}|$, and for $t \ge N \Dt$, for some $J$ and $N$. With this $\varphi$, we define the following functional
\begin{align*}
\langle \mathcal{L}_{\Dx}(\eta, q), \varphi \rangle & := -\langle \eta( k(x), u_\Dx)_t + q(k(x), u_\Dx)_x , \varphi \rangle \\
&= \int_0^T\int_{\R}      \eta(k(x),u_{\Dx}) \varphi_t      +   q(k(x),u_{\Dx}) \varphi_x \,dx \,dt
:= \langle \mathcal{E}_{\Dx}(\eta,q), \varphi \rangle +  \langle \mathcal{E}'_{\Dx}(\eta,q), \varphi \rangle,
\end{align*}
where
\begin{align*}
\langle \mathcal{E}'_{\Dx}(\eta,q), \varphi \rangle&=\int_0^T\int_{\R}  \left(  \eta(k(x),u_{\Dx})  - \eta(k_{\Dx},u_{\Dx})    \right) \varphi_t   \,dx dt \\
& \qquad \qquad \qquad \qquad \qquad \qquad + \int_0^T\int_{\R} \left(  q(k(x),u_{\Dx})   - q(k_{\Dx},u_{\Dx}) \right) \varphi_x  \,dx dt,
\end{align*}
and
\begin{align*}
\langle \mathcal{E}_{\Dx}(\eta,q), \varphi \rangle=\int_0^T\int_{\R}      \eta(k_{\Dx},u_{\Dx}) \varphi_t  + q(k_{\Dx},u_{\Dx}) \varphi_x \,dx \,dt.
\end{align*}
In what follows,  we let $\Pi_T$ denote an arbitrary but fixed bounded open subset of $\R \times [0,T]$. Let
$r \in (1,2]$ and set $p = \frac{r}{r-1} \in [1, \infty)$. With $\varphi \in W_0^{1,r}(\Pi_T)$, we have by H\"{o}lder's inequality
\begin{align*}
\abs{\langle \mathcal{E}'_{\Dx}(\eta,q) , \varphi \rangle} \le C \norm{k - k_{\Dx} }_{L^p(\Pi_T)} \norm{\varphi}_{W_0^{1,r} (\Pi_T)} \rightarrow 0 \quad \text{as} \, \,\Dx \downarrow 0,
\end{align*}
so that 
\begin{align}
\label{eq:1}
\seq{\mathcal{E}'_{\Dx}(\eta,q)}_{\Dx>0} \, \text{ is compact in} \, \,W^{-1,r}(\Pi_T), \quad r \in (1,2].
\end{align}
Next we focus on the other term 
\begin{align*}
\langle \mathcal{E}_{\Dx}(\eta,q), \varphi \rangle&=\sum_j \sum_{n=0}^{N-1} \int_{t^n}^{t^{n+1}}\int_{\xm}^{\xp}   \eta(k^{\Dx},u_j^n) \varphi_t      +   q(k^{\Dx},u_j^n) \varphi_x \,dx dt\\
& = \sum_j\sum_{n=0}^{N-1} \int_{t^n}^{t^{n+1}}\int_{\xm}^{x_j}   \eta(\km,u_j^n) \varphi_t      +   q(\km,u_j^n) \varphi_x \ dx dt\\
&\qquad \qquad \qquad \quad +\sum_j \sum_{n=0}^{N-1} \int_{t^n}^{t^{n+1}}\int_{x_j}^{\xp}   \eta(\kp,u_j^n) \varphi_t      +   q(\kp,u_j^n) \varphi_x \ dx dt\\
&:= \langle \mathcal{E}^1_{\Dx}(\eta,q), \varphi \rangle + \langle \mathcal{E}^2_{\Dx}(\eta,q), \varphi \rangle
\end{align*}
where 
\begin{align}
\label{eq:2}
&\langle \mathcal{E}^1_{\Dx}(\eta,q) ,\varphi \rangle=\sum_j\sum_{n=0}^{N-1} \int_{t^n}^{t^{n+1}}\int_{\xm}^{\xp}   \eta(\km,u_j^n) \varphi_t      +   q(\km,u_j^n) \varphi_x \ dx dt 
\end{align}
and
\begin{align}
\label{eq:3}
&\langle \mathcal{E}^2_{\Dx}(\eta,q), \varphi \rangle  =  \sum_j \sum_{n=0}^{N-1} \int_{t^n}^{t^{n+1}}\int_{x_j}^{\xp}   
\left( \eta(\kp,u_j^n) -\eta(\km,u_j^n)\right)   \varphi_t   \ dx dt  \\
& \hspace{3cm} + \sum_j \sum_{n=0}^{N-1} \int_{t^n}^{t^{n+1}}\int_{x_j}^{\xp}  \left( q(\kp,u_j^n)- q(\km,u_j^n)\right)\varphi_x \ dx dt . \notag
\end{align}
The proof is essentially complete if we assume that both $\seq{\mathcal{E}^1_{\Dx}(\eta,q)}_{\Dx>0}$ and $\seq{\mathcal{E}^2_{\Dx}(\eta,q)}_{\Dx>0}$ are compact in $H^{-1}_{\loc}(\Pi_T)$. 
\end{proof}

In the proof Lemma~\ref{lem:compact}, we claimed the compactness of  $\seq{\mathcal{E}^1_{\Dx}(\eta,q)}_{\Dx>0}$ and $\seq{\mathcal{E}^2_{\Dx}(\eta,q)}_{\Dx>0}$. We try to justify this claim below. For the rest of this section, we assume that the conditions stated in Lemma~\ref{lem:compact} hold. We will also continue to use $\Pi_T$ introduced in the proof of the above lemma.

\begin{lemma}  \label{lem:compact1}
\begin{align}
\label{eq:4}
\seq{\mathcal{E}^2_{\Dx}(\eta,q)}_{\Dx>0} \, \text{ is compact in} \, \,H^{-1}_{\loc}(\Pi_T).
\end{align}
\end{lemma}
\begin{proof}
Consider the expression given by \eqref{eq:3}, which is split as
\begin{align*}
&\langle \mathcal{E}^{2}_{\Dx}(\eta,q), \varphi \rangle  := \langle \mathcal{E}^{2,1}_{\Dx}(\eta,q), \varphi \rangle  + \langle \mathcal{E}^{2,2}_{\Dx}(\eta,q), \varphi \rangle 
\end{align*} 
where
\begin{align*}
&\langle \mathcal{E}^{2,1}_{\Dx}(\eta,q), \varphi \rangle  =   \sum_j \sum_{n=0}^{N-1} \int_{t^n}^{t^{n+1}}\int_{x_j}^{\xp}   
\left( \eta(\kp,u_j^n) -\eta(\km,u_j^n)\right)   \varphi_t   \ dx dt \\
&\langle \mathcal{E}^{2,2}_{\Dx}(\eta,q), \varphi \rangle = \sum_j \sum_{n=0}^{N-1} \int_{t^n}^{t^{n+1}}\int_{x_j}^{\xp}  \left( q(\kp,u_j^n)- q(\km,u_j^n)\right)\varphi_x \ dx dt.
\end{align*}
Now, using Cauchy-Schwartz inequality repeatedly,  we obtain
\begin{align*}
\abs{\langle \mathcal{E}^{2,1}_{\Dx}(\eta,q), \varphi \rangle} & = \abs{\sum_j \sum_{n=0}^{N-1} \int_{t^n}^{t^{n+1}}\int_{x_j}^{\xp}   \left( \eta(\kp,u^n_j) -\eta(\km,u^n_j)\right)   \varphi_t   \ dx dt} \\
&\leq C \,  \sum_j \int_0^T\int_{x_j}^{\xp}\, \abs{\kp-\km} \, \abs {\varphi_t} \, dx \, dt\\
&\leq C  \,  \sum_j \int_0^T\, \abs{\kp-\km}\, (\Dx)^\frac12 \, \left( \int_{x_j}^{\xp} \, \abs {\varphi_t}^2 \, dx\right)^\frac12 dt\\
&\leq C \, (\Dx)^\frac12 \int_0^T\left( \sum_j \abs{\kp-\km}^2\right)^\frac12 \left( \sum_j\int_{x_j}^{\xp} \, \abs {\varphi_t}^2 \, dx\right)^\frac12 dt \\
&\leq C \, (\Dx)^\frac12 \,  \abs{k}_{BV}\,  T^{1/2} \, \norm{\varphi}_{H^1(\Pi_T)}.
\end{align*}
A similar argument can be used to show almost \emph{verbatim}
\begin{align*}
\abs{\langle \mathcal{E}^{2,2}_{\Dx}(\eta,q), \varphi \rangle} \leq \, C (\Dx)^\frac12 \,  \abs{k}_{BV}\,  T^{1/2} \, \norm{\varphi}_{H^1(\Pi_T)}.
\end{align*}
Therefore, summing up, we have 
\begin{align*}
\abs{\langle \mathcal{E}^{2}_{\Dx}(\eta,q), \varphi \rangle} \leq \, C(\Dx)^\frac12 \, \abs{k}_{BV}\,  T^{1/2} \, \norm{\varphi}_{H^1(\Pi_T)}.
\end{align*}
Thus, 
\begin{align*}
\seq{\mathcal{E}^{2}_{\Dx}(\eta,q)}_{\Dx>0} \, \text{ is compact in} \, \,H^{-1}_{\loc}(\Pi_T).
\end{align*}
\end{proof}
Next, we need to show the compactness of $\seq{\mathcal{E}^{1}_{\Dx}(\eta,q)}_{\Dx>0}$. First note that we can express \eqref{eq:2} as
\begin{align*}
&\langle \mathcal{E}^1_{\Dx}(\eta,q), \varphi \rangle  =- \Dt\sum_j \sum_{n=0}^{N-1} \int_{\xm}^{\xp}  
\Dpt \eta(\km,u^n_j) \ \varphi^{n+1}(x)\ dx \\
&\qquad \qquad \qquad \quad -\sum_{j} \int_{I_j} \eta(\km,u^0_j) \varphi^{0}(x)\ dx    -\Dx \sum_j \sum_{n=0}^{N-1} \int_{t^n}^{t^{n+1}} \Dm q(\km,u^n_j)   \ \varphi_{j-\frac12}(t) \ dt,
\end{align*}
where $\varphi_{j-\frac12}(t):=\varphi(\xm,t)$, and $\varphi^{n+1}(x):=\varphi(x,t^{n+1})$. 
This further implies that
\begin{align*}
-\langle \mathcal{E}^1_{\Dx}(\eta,q), \varphi \rangle & = \sum_j \sum_{n=0}^{N-1} \int_{t^n}^{t^{n+1}}\int_{\xm}^{\xp}   \left(\Dpt\eta(\km,u^n_j) +\Dm q(\km,u^n_j)\right)  \ \varphi  \ dx dt\\
& \qquad \qquad + \sum_j\sum_{n=0}^{N-1} \int_{t^n}^{t^{n+1}} \int_{\xm}^{\xp}\Dm q(\km,u^n_j)   \ \left(\varphi_{j-\frac12}(t)-\varphi(x,t)\right) \ dxdt\\
&\quad \qquad \qquad  + \sum_j\sum_{n=0}^{N-1} \int_{t^n}^{t^{n+1}} \int_{\xm}^{\xp}\Dpt \eta(\km,u^n_j)   \ \left(\varphi^{n+1}(x)-\varphi(x,t)\right) \ dxdt \\
&\qquad \qquad \qquad  \qquad \qquad  + \sum_{j} \int_{I_j} \eta(\km,u^0_j) \varphi^{0}(x)\ dx  \\
& \qquad \qquad   := \langle \mathcal{E}^{1,1}_{\Dx}(\eta,q), \varphi \rangle + \langle \mathcal{E}^{1,2}_{\Dx}(\eta,q), \varphi \rangle
\end{align*}
where
\begin{align}
\label{eq:5}
\langle \mathcal{E}^{1,1}_{\Dx}(\eta,q), \varphi \rangle & = \sum_j \sum_{n=0}^{N-1} \int_{t^n}^{t^{n+1}}\int_{\xm}^{\xp}   \left(\Dpt\eta(\km,u^n_j) +\Dm q(\km,u^n_j)\right)  \ \varphi  \ dx dt
\end{align}
while $\langle \mathcal{E}^{1,2}_{\Dx}(\eta,q), \varphi \rangle$ corresponds to the three remaining summation term. The compactness of $\seq{\mathcal{E}^{1,1}_{\Dx}(\eta,q)}_{\Dx>0}$ and $\seq{\mathcal{E}^{1,2}_{\Dx}(\eta,q)}_{\Dx>0}$ will ensure the compactness of $\seq{\mathcal{E}^{1}_{\Dx}(\eta,q)}_{\Dx>0}$
\begin{lemma} 
\label{lem:compact2}
\begin{align}
\label{eq:6}
\seq{\mathcal{E}^{1,2}_{\Dx}(\eta,q)}_{\Dx>0} \, \text{ is compact in} \, \,H^{-1}_{\loc}(\Pi_T).
\end{align}
\end{lemma}
\begin{proof} Firstly, we split the $\mathcal{E}^{1,2}_{\Dx}(\eta,q)$ as
\begin{align*}
&\langle \mathcal{E}^{1,2}_{\Dx}(\eta,q), \varphi \rangle  := \langle \mathcal{E}^{1,2,1}_{\Dx}(\eta,q), \varphi \rangle  + \langle \mathcal{E}^{1,2,2}_{\Dx}(\eta,q), \varphi \rangle  +  \mathcal{E}^{1,2,3}_{\Dx}(\eta,q), \varphi \rangle
\end{align*}
where
\begin{align*}
&\langle\mathcal{E}^{1,2,1}_{\Dx}(\eta,q),\varphi\rangle =  \sum_j\sum_{n=0}^{N-1} \int_{t^n}^{t^{n+1}} \int_{\xm}^{\xp}\Dm q(\km,u^n_j)   \ \left(\varphi_{j-\frac12}(t)-\varphi(x,t)\right) \ dxdt\\
&\langle\mathcal{E}^{1,2,2}_{\Dx}(\eta,q),\varphi\rangle = \sum_j\sum_{n=0}^{N-1} \int_{t^n}^{t^{n+1}} \int_{\xm}^{\xp}\Dpt \eta(\km,u^n_j)   \ \left(\varphi^{n+1}(x)-\varphi(x,t)\right) \ dxdt \\
&\langle\mathcal{E}^{1,2,3}_{\Dx}(\eta,q),\varphi\rangle = \sum_{j} \int_{I_j} \eta(\km,u^0_j) \varphi^{0}(x)\ dx .
\end{align*}
We first estimate the term $\mathcal{E}^{1,2,1}_{\Dx}(\eta,q)$ as follows:
\begin{align*}
&\langle\mathcal{E}^{1,2,1}_{\Dx}(\eta,q),\varphi \rangle= \sum_{j,n} \int_{t^n}^{t^{n+1}}  \int_{I_j}\frac{q(\km,u^n_j)-q(k_{j-\frac32},u^n_{j-1})}{\Dx} \ 
\left(\varphi_{j-\frac12}-\varphi(x,t)\right) \ dx dt\\
&\quad \qquad \qquad = \, \underbrace{\sum_j \sum_{n=0}^{N-1} \int_{t^n}^{t^{n+1}}   \int_{\xm}^{\xp}\frac{q(\km,u^n_j)-q(\km,u^n_{j-1})}{\Dx} \left(\varphi_{j-\frac12}-\varphi(x,t)\right) \ dx dt}_{\langle I_{\Dx}(\eta,q),\varphi\rangle}\\
&\qquad\qquad \qquad \quad + \underbrace{\sum_j \sum_{n=0}^{N-1} \int_{t^n}^{t^{n+1}}    \int_{\xm}^{\xp}\frac{q(\km,u^n_{j-1})-q(k_{j-\frac32},u^n_{j-1})}{\Dx} \left(\varphi_{j-\frac12}-\varphi(x,t)\right) \ dx dt.}_{\langle II_{\Dx}(\eta,q),\varphi\rangle}
\end{align*}
Since $q'$ and $\varphi$ are both bounded, and $k$ is of bounded variation, we see that
\begin{align*}
\abs{\langle II_{\Dx}(\eta,q),\varphi\rangle} \leq C(T)\,\norm{\varphi}_{L^{\infty}(\Pi_T)}\, \abs{k}_{BV}.
\end{align*}
Hence we conclude that
\begin{align}
II_{\Dx}(\eta,q) \in \mathcal{M}_{\loc}(\Pi_T). \label{eq:7}
\end{align}
Now the term $I_{\Dx}(\eta,q)$ can be estimated, using Cauchy-Schwartz inequality repeatedly,  as follows
\begin{align*}
&\abs{\langle I_{\Dx}(\eta,q),\varphi\rangle} \leq C \sum_j \sum_{n=0}^{N-1} \int_{t^n}^{t^{n+1}}    \int_{\xm}^{\xp}\abs{\frac{u^n_j-u^n_{j-1}}{\Dx}} \, \int_{\xm}^x \abs{\varphi_x(\theta)} d \theta\, dx dt\\
&\quad \leq (\Dx)^\frac12 \, C\sum_j \sum_{n=0}^{N-1} \int_{t^n}^{t^{n+1}}   \int_{\xm}^{\xp}\abs{\frac{u^n_j-u^n_{j-1}}{\Dx}} \, \left(\int_{\xm}^x \abs{\varphi_x(\theta)} ^2d \theta\right)^{1/2}\, dx \,dt\\
&\qquad \leq (\Dx)^\frac12 \, C\sum_j \sum_{n=0}^{N-1} \int_{t^n}^{t^{n+1}}   \abs{u^n_j-u^n_{j-1}}\left(\int_{\xm}^{\xp} \abs{\varphi_x} ^2dx\right)^{1/2} \,dt\\
&\qquad \quad \leq (\Dx)^\frac12 \, C\left(  \Dt\sum_{n=0}^{N-1} \sum_j\int_{t^n}^{t^{n+1}}\abs{u^n_j-u^n_{j-1}}^2\, dt\right)^{\frac12} \, \norm{\varphi}_{H^1(\Pi_T)}.
\end{align*}
Therefore, in view of the a priori bound \eqref{eq:est_final1}, we reach to the conclusion that   
\begin{align}
\label{eq:8}
\seq{I_{\Dx}(\eta,q)}_{\Dx> 0} \, \text{ is compact in} \, \,H^{-1}_{\loc}(\Pi_T).
\end{align}
To sum up,  making use of \eqref{eq:7} and \eqref{eq:8}, along with the Lemma~\ref{lem:Murat}, we conclude that 
\begin{align}
\label{eq:9}
\seq{\mathcal{E}^{1,2,1}_{\Dx}(\eta,q)}_{\Dx> 0} \, \text{ is compact in} \, \,H^{-1}_{\loc}(\Pi_T).
\end{align}
A similar type of argument, with the use of Cauchy-Schwartz inequality, yields 
\begin{align*}
\abs{ \langle\mathcal{E}^{1,2,2}_{\Dx}(\eta,q),\varphi\rangle } \leq  C (\Dt)^{1/2} \left( \Dt \Dx^2 \sum_{n=0}^{N}\sum_j (\Dpt u_j^n)^2 \right)
^{1/2} \norm{\varphi}_{H^1(\Pi_T)}.
\end{align*}
Making use of the a priori bound \eqref{eq:est_final2}, we conclude that 
\begin{align}
\label{eq:10}
\seq{ \mathcal{E}^{1,2,2}_{\Dx}(\eta,q)}_{\Dx> 0} \, \text{ is compact in} \, \,H^{-1}_{\loc}(\Pi_T).
\end{align}
Moreover, note that the bound
\begin{align*}
\abs{ \langle\mathcal{E}^{1,2,3}_{\Dx}(\eta,q),\varphi\rangle } \leq  C \norm{\varphi}_{L^{\infty}(\R)},
\end{align*}
ensures that
\begin{align}
\label{eq:11}
\seq{ \mathcal{E}^{1,2,3}_{\Dx}(\eta,q)}_{\Dx>0} \in \mathcal{M}_{\loc}(\Pi_T).
\end{align}
Using \eqref{eq:9},\eqref{eq:10} and \eqref{eq:11} along the lines of Lemma~\ref{lem:Murat}, proves that $\seq{ \mathcal{E}^{1,2}_{\Dx}(\eta,q)}_{\Dx>0}$ is compact.
\end{proof}

Next we estimate the remaining term $\mathcal{E}^{1,1}_{\Dx}(\eta,q)$. We first note that Taylor's series expansion yields
\begin{equation}
\label{eta-expansion}
 \eta_u(\km,u^n_j)\Dpt u^n_j=\Dpt  \eta(\km,u^n_j) - \frac\Dt2 \eta_{uu}(\km,\theta^{n+\frac12}_j) (D^{t}_{+}u^n_j)^2,
\end{equation}
for some $\theta^{n+\frac12}_j$ between $u^n_j$ and $u^{n+1}_j$.
At this point, we shall make use of the fully-discrete scheme \eqref{eq:scheme} and \eqref{eta-expansion} 
to decompose the term $\mathcal{E}^{1,1}_{\Dx}(\eta,q)$ as follows:
\begin{align*}
&\langle\mathcal{E}_{\Dx}^{1,1}(\eta,q),\varphi \rangle=  \sum_j \sum_{n=0}^{N-1} \int_{t^n}^{t^{n+1}} \int_{\xm}^{\xp}   \left( \Dpt\eta(\km,u^n_j) +\Dm q(\km,u^n_j)\right) \varphi  \ dx dt\\
&\qquad \qquad \qquad \quad := \langle\mathcal{E}_{\Dx}^{1,1,1}(\eta,q),\varphi\rangle\, +\, \langle\mathcal{E}_{\Dx}^{1,1,2}(\eta,q),\varphi\rangle\, +\,\langle\mathcal{E}_{\Dx}^{1,1,3}(\eta,q),\varphi\rangle\, + \, \langle\mathcal{E}_{\Dx}^{1,1,4}(\eta,q),\varphi\rangle,
\end{align*}
where
\begin{align*}
& \langle\mathcal{E}_{\Dx}^{1,1,1}(\eta,q),\varphi\rangle= \, \sum_{j} \sum_{n=0}^{N-1} \int_{t^n}^{t^{n+1}} \int_{I_j} \left[ -\eta_u(\km,u^n_j) \Dm(\kp\fhp) \, +\, \Dm q(\km, u^n_j)\right] \varphi  \, dx dt,\\
& \langle\mathcal{E}_{\Dx}^{1,1,2}(\eta,q),\varphi\rangle= \,\beta\, \sum_j \sum_{n=0}^{N-1}\int_{t^n}^{t^{n+1}}\int_{I_j} \eta_u(\km,u^n_j) \Dx \, (\Dp \Dm u^n_j) \, \,\varphi  \, dx dt,\\
& \langle\mathcal{E}_{\Dx}^{1,1,3}(\eta,q),\varphi\rangle = \,\gamma\, \sum_j\sum_{n=0}^{N-1} \int_{t^n}^{t^{n+1}}\int_{I_j} \eta_u(\km,u^n_j) \mu(\Dx) \, (\Dpt\Dp \Dm u^n_j) \, \,\varphi  \, dx dt,\\
& \langle\mathcal{E}_{\Dx}^{1,1,4}(\eta,q),\varphi\rangle= \, \sum_j\sum_{n=0}^{N-1} \int_{t^n}^{t^{n+1}}\int_{I_j} \frac\Dt2 \eta_{uu}(\km,\theta^{n+\frac12}_j) (D^{t}_{+}u^n_j)^2\,\varphi  \, dx dt.
\end{align*}
Our aim is to estimate each of the above terms suitably. In order to do so, we introduce the notation $\Phi_j :=\int_{\xm}^{\xp} \varphi(x,t)\, dx$.

\begin{lemma}
\label{lem:compact3}
\begin{align}
\label{eq:12}
\seq{\mathcal{E}^{1,1,1}_{\Dx}(\eta,q)}_{\Dx>0} \, \text{ is compact in} \, \,H^{-1}_{\loc}(\Pi_T).
\end{align}
\end{lemma}
\begin{proof}
We rewrite $\mathcal{E}_\Dx^{1,1,1}(\eta,q)$ as
\begin{align*}
\langle\mathcal{E}_{\Dx}^{1,1,1}(\eta,q),\varphi\rangle&= \sum_{j} \sum_{n=0}^{N-1}  \int_{t^n}^{t^{n+1}}  \left[ -\eta_u(\km, u^n_j)\Dm(\kp \fhp)+ \Dm q(\km,u^n_j)\right] \, \Phi_j \, dt\\
&= \sum_{n=0}^{N-1} \int_{t^n}^{t^{n+1}} \sum_j \left. \Bigg[ -\eta_u(\km, u^n_j) \left(\km\Dm\fhp\, + \, \fhp \Dm \kp\right)\, \right. \\
&\hspace{6cm} \left. + \, \frac{q(\km,u^n_j)-q(k_{j-\frac32},u^n_{j-1})}{\Dx}\right]\, \Phi_j \,dt\\
&= \sum_{n=0}^{N-1} \int_{t^n}^{t^{n+1}} \sum_j   \left. \Bigg[ -\eta_u(\km, u^n_j) \,  \km \Dm \fhp \, \right. \\
&\hspace{6cm} \left. + \, \frac{q(\km,u^n_j)-q(k_{j-\frac12},u^n_{j-1})}{\Dx}\right] \ \Phi_j \ dt \\ 
&\quad + \sum_{n=0}^{N-1} \int_{t^n}^{t^{n+1}}\sum_j \left. \Bigg[ -\eta_u(\km, u^n_j)  \fhp \Dm\kp \, \right. \\
&\hspace{6cm} \left.+ \, \frac{q(\km, u^n_{j-1})- q(k_{j-\frac32}, u^n_{j-1})}{\Dx}\right] \Phi_j \, dt \\
&:= \langle\mathcal{E}_\Dx^{1,1,1,1}(\eta,q),\varphi\rangle +  \langle\mathcal{E}_\Dx^{1,1,1,2}(\eta,q), \varphi\rangle.
\end{align*}
Since $\eta'$ and $q'$ are bounded, making use the total variation bound for $k$, we conclude that
\begin{align*}
\abs{\langle\mathcal{E}_\Dx^{1,1,1,2}(\eta, q), \varphi\rangle} \leq C \, \abs{k}_{BV} \,\norm{\varphi}_{L^{\infty}(\Pi_T)}.
\end{align*}
Hence this implies
\begin{align}
\label{eq:13}
\seq{ \mathcal{E}_\Dx^{1,1,1,2}(\eta, q)}_{\Dx> 0} \in \mathcal{M}_{\loc}(\Pi_T).
\end{align}
Next we move on to estimate $\mathcal{E}_\Dx^{1,1,1,1}(\eta,q)$.
At this point we shall make use of specific structure of $\eta$ and $q$ given by \eqref{eq:entropies}.

{\bf Estimate for $(\eta, q):=(\eta_1,q_1)$:} First, we recall that
\begin{align*}
 \eta_1(u)=u-c, \quad q_1(k,u)=f(k,u)-f(k,c)=k(x)(f(u)-f(c)).
 \end{align*}
 Thus, we find that
 \begin{align*}
-\frac{\partial}{\partial u}\eta_1(\km, u^n_j) \, & \km \Dm \fhp \, + \, \frac{q_1(\km,u^n_j)-q_1(k_{j-\frac12},u^n_{j-1})}{\Dx}\\
&=\, - \, \km \, \frac{\fhp-\fhm}{\Dx} \, + \, \km \, \frac{f(u^n_j)-f(u^n_{j-1})}{\Dx}\\
&=\frac{\km}{\Dx} \, \left[ \left(f(u^n_j)-\fhp\right) \, -\, \left(f(u^n_{j-1})-\fhm\right)\right]
= \km\Dm\left(f(u^n_j)-\fhp\right).
 \end{align*}
We insert this in the expression of $\langle\mathcal{E}_\Dx^{1,1,1,1}(\eta,q),\varphi\rangle$.
Then, using summation by parts we obtain
\begin{align*}
\langle\mathcal{E}_\Dx^{1,1,1,1}(\eta_1,q_1),\varphi\rangle &=  \sum_{n=0}^{N-1} \int_{t^n}^{t^{n+1}} \sum_j \km\Dm\left(f(u^n_j)-\fhp\right) \, \Phi_j\, dt\\
&\qquad= - \sum_{n=0}^{N-1} \int_{t^n}^{t^{n+1}} \sum_j  \left(f(u^n_j)-\fhp\right)\, \Dp\left(\km \Phi_j\right) \, dt\\
&\qquad= -  \sum_{n=0}^{N-1} \int_{t^n}^{t^{n+1}} \sum_j  \left(f(u^n_j)-\fhp\right)\, \left[ \kp \, \Dp \Phi_j \, + \, \Phi_j \, \Dp\km\right] \, dt \\
&\qquad := \langle I_{\Dx}(\eta_1,q_1) ,\varphi\rangle \, +  \langle II_{\Dx}(\eta_1,q_1) ,\varphi\rangle.
\end{align*}
To proceed further, we first estimate the term $I_{\Dx}(\eta_1,q_1)$ as follows
\begin{align*}
\abs{\langle I_{\Dx}(\eta_1,q_1) ,\varphi\rangle} &=\abs{- \sum_{n=0}^{N-1} \int_{t^n}^{t^{n+1}} \sum_j  \left(f(u^n_j)-\fhp\right)\,\kp \, \Dp \Phi_j \, dt}\\
&\leq  \norm{k}_{\infty} \norm{f^\prime}_{\infty}\, \sum_{n=0}^{N-1} \int_{t^n}^{t^{n+1}} \sum_j \abs{u^n_j-u^n_{j+1}} \, \abs {\Dp \Phi_j} \, dt\\
&=\norm{k}_{\infty}\norm{f^\prime}_{\infty} \, \sum_{n=0}^{N-1} \int_{t^n}^{t^{n+1}} \sum_j \abs{u^n_j-u^n_{j-1}} \, \abs {\Dm \Phi_j} \, dt.
\end{align*}
Next, we estimate $\abs{\Dm \Phi_j}$ as follows
\begin{align*}
\Phi_j -\Phi_{j-1}&= \int_{\xm}^{\xp}\varphi \, dx \, - \, \int_{x_{j-\frac32}}^{\xm}\varphi \, dx
= \int_{\xm}^{\xp}\left( \varphi-\varphi_{j-\frac12}\right) \, dx \, - \, \int_{x_{j-\frac32}}^{\xm}\left( \varphi-\varphi_{j-\frac12}\right) \, dx\\
&= \int_{\xm}^{\xp} \int_{\xm}^x \varphi_x \, d\theta\,  dx \, - \, \int_{x_{j-\frac32}}^{\xm} \int_{\xm}^x \varphi_x \, d\theta\,  dx.
\end{align*}
Therefore,
\begin{align*}
\abs{\Phi_j -\Phi_{j-1}}& \leq \int_{\xm}^{\xp} \int_{\xm}^x \abs{\varphi_x }\, d\theta\,  dx \, +
 \, \int_{x_{j-\frac32}}^{\xm} \int_{\xm}^x \abs{\varphi_x} \, d\theta\,  dx\\
 &\leq 2 \, \Dx \int_{\xm}^{\xp} \abs{\varphi_x} \, dx \ \leq \,2\, (\Dx)^{\frac32} \left(\int_{\xm}^{\xp}\abs{\varphi_x}^2\  dx\right)^{1/2}.
\end{align*}
Thus, we conclude that
\begin{align}
\label{estimate-a}
\abs{\Dm \Phi_j} \leq 2 \, (\Dx)^{1/2}\, \left(\int_{\xm}^{\xp}\abs{\varphi_x}^2\  dx\right)^{1/2}
\end{align}
Therefore, applying Cauchy-Schwartz inequality and using the a priori estimate \eqref{eq:est_final1}, we obtain
\begin{align*}
\abs{\langle I_{\Dx}(\eta_1,q_1) ,\varphi\rangle} & \leq  2\, \norm{k}_{\infty}\norm{f^\prime}_{\infty}\, (\Dx)^{1/2} \, \left( \sum_{n=0}^N \Dt \Dx\norm{\Dm u^n}^2 \right)^{1/2} \, \norm{\varphi}_{H^1(\Pi_T)} \\
&\leq C(\Dx)^{1/2} \, \norm{\varphi}_{H^1(\Pi_T)}.
\end{align*}
Hence
\begin{align}
\label{eq:14}
\seq{ I_{\Dx}(\eta_1,q_1)}_{\Dx> 0} \, \text{ is compact in} \, \,H^{-1}_{\loc}(\Pi_T).
\end{align}
On the other hand, making use of the total variation bound for $k$, we have
\begin{align*}
\abs{\langle II_{\Dx} (\eta_1,q_1),\varphi\rangle}  \leq C \, \norm{\varphi}_{L^{\infty}(\Pi_T)} \,\abs{k}_{BV}.
\end{align*}
Therefore
\begin{align}
\label{eq:15}
\seq{ II_{\Dx}(\eta_1,q_1)}_{\Dx> 0} \in \mathcal{M}_{\loc}(\Pi_T).
\end{align}
Using \eqref{eq:14}, \eqref{eq:15} in tandem with Lemma~\ref{lem:Murat}, we get the compactness of  $\seq{ \mathcal{E}^{1,1,1,1}_{\Dx}(\eta,q)}_{\Dx>0}$ in $H^{-1}_{\loc}(\Pi_T)$, for $(\eta, q):=(\eta_1, q_1)$.

{\bf Estimate of $\mathcal{E}_{\Dx}^{1,1,1,1}$ for $(\eta, q):=(\eta_2, q_2)$:}  We recall that
 \begin{align*}
 \eta_2(k,u)=k\left(f(u)-f(c)\right), \quad q_2(k,u)=\int_c^u(kf'(\xi))^2d\xi,
 \end{align*}
hence,
\begin{align*}
\frac{\partial}{\partial u}\eta_2(k,u)=\, k \, f'(u), \quad  \frac{\partial}{\partial u}q_2(k,u)=k^2 f'(u)^2.
\end{align*}
This implies that
\begin{align*}
- \frac{\partial}{\partial u}\eta_2(\km, u^n_j)  \km & \Dm\fhp \, + \, \frac{1}{\Dx}\left[q_2(\km,u^n_j)-q_2(\km,u^n_{j-1})\right]\\
&= -\km^2\, f'(u^n_j) \, \Dm\fhp \, + \,   \frac{1}{\Dx}\left[q_2(\km,u^n_j)-q_2(\km,u^n_{j-1})\right].
\end{align*}
Again, by Taylor's Theorem
\begin{align*}
q_2(\km,&u^n_j)-q_2(\km,u^n_{j-1})\\
&= (u^n_j-u^n_{j-1})\, \frac{\partial}{\partial u} q_2(\km,u^n_j) \, +\frac12 \frac{\partial^2}{\partial u^2} q_2(\km,\theta^n_{j-\frac12}) \, (u^n_j-u^n_{j-1})^2 \\
&=\km^2 f'(u^n_j)^2\, (u^n_j-u^n_{j-1})\, +\frac12 \frac{\partial^2}{\partial u^2} q_2(\km,\theta^n_{j-\frac12}) \, (u^n_j-u^n_{j-1})^2 \\
&=\km^2 f'(u^n_j)\, (f(u^n_j)-f(u^n_{j-1}))\,  +\frac12 \frac{\partial^2}{\partial u^2} q_2(\km,\theta^n_{j-\frac12}) \, (u^n_j-u^n_{j-1})^2,
\end{align*}
where $\theta^n_{j-\frac12}$ lies between $u^n_j$ and $u^n_{j-1}$. Thus
\begin{align*}
- \frac{\partial}{\partial u}\eta_2(\km, u^n_j) \km & \Dm\fhp \, + \, \frac{1}{\Dx}\left[q_2(\km,u^n_j)-q_2(\km,u^n_{j-1})\right]\\
&=\km^2 f'(u^n_j)\, \Dm(f(u^n_j)-\fhp)\, + \frac{1}{2\Dx} \frac{\partial^2}{\partial u^2} q_2(\km,\theta^n_{j-\frac12}) \, (u^n_j-u^n_{j-1})^2.
\end{align*}
Therefore
\begin{align*}
\langle\mathcal{E}_{\Dx}^{1,1,1,1}(\eta_2,q_2),\varphi\rangle&= \underbrace{\sum_{n=0}^{N-1} \int_{t^n}^{t^{n+1}}\sum_j \left(\km^2 f'(u^n_j)\, \Dm(f(u^n_j)-\fhp) \right)\Phi_j \, dt}_{\langle \mathcal{Q}^1_{\Dx}(\eta_2,q_2),\varphi\rangle}\\
&\quad\qquad \qquad +\underbrace{ \sum_{n=0}^{N-1} \int_{t^n}^{t^{n+1}}\sum_j  \frac{1}{2\Dx} \frac{\partial^2}{\partial u^2} q_2(\km,\theta^n_{j-\frac12}) \, (u^n_j-u^n_{j-1})^2 \ \Phi_j\, dt}_{\langle \mathcal{Q}^2_{\Dx}(\eta_2,q_2),\varphi\rangle}.
\end{align*}
Making use of the a priori estimate\eqref{eq:est_final1}, we conclude that 
\begin{align*}
{\langle \mathcal{Q}^2_{\Dx}(\eta_2,q_2),\varphi\rangle}:=  \sum_{n=0}^{N-1} \int_{t^n}^{t^{n+1}}\sum_j  \frac{1}{2\Dx} \frac{\partial^2}{\partial u^2} q_2(\km,\theta^n_{j-\frac12}) \, (u^n_j-u^n_{j-1})^2 \ \Phi_j\, dt \leq C \ \norm{\varphi}_{\infty}.
\end{align*}
Hence
\begin{align}
\label{eq:16}
\seq{ \mathcal{Q}^2_{\Dx}(\eta_2,q_2)}_{\Dx> 0} \in \mathcal{M}_{\loc}(\Pi_T).
\end{align}
On the other hand, using summation by parts, we can estimate the other term as follows
\begin{align*}
{\langle \mathcal{Q}^1_{\Dx}(\eta_2,q_2),\varphi\rangle}:=\sum_{n=0}^{N-1}& \int_{t^n}^{t^{n+1}}\sum_j  \left(\km^2 f'(u^n_j)\, \Dm(f(u^n_j)-\fhp) \right)\Phi_j \, dt\\
& =-\sum_{n=0}^{N-1} \int_{t^n}^{t^{n+1}}\sum_j \Dp\left(\km^2 f'(u^n_j)\ \Phi_j \right)\, (f(u^n_j)-\fhp) \, dt\\
&\quad=-\underbrace{\sum_{n=0}^{N-1} \int_{t^n}^{t^{n+1}}\sum_j(f(u^n_j)-\fhp) \, (\Dp \km^2) \, f'(u^n_{j+1})\ \Phi_{j+1}\, dt}_{\langle\mathcal{X}_{\Dx}(\eta_2,q_2),\varphi\rangle}\\
&\qquad -\underbrace{\sum_{n=0}^{N-1} \int_{t^n}^{t^{n+1}}\sum_j(f(u^n_j)-\fhp) \, \km^2 \ (\Dp f'(u^n_{j}) ) \, \Phi_{j+1}\, dt}_{\langle\mathcal{Y}_{\Dx}(\eta_2,q_2),\varphi\rangle}\\
&\qquad \quad-\underbrace{\sum_{n=0}^{N-1} \int_{t^n}^{t^{n+1}}\sum_j(f(u^n_j)-\fhp) \, \km^2 \  f'(u^n_{j}) )\,\Dm( \Phi_{j})\, dt}_{\langle\mathcal{Z}_{\Dx}(\eta_2,q_2),\varphi\rangle}.
\end{align*}
Using essentially the same type of arguments as before, we can deal with the above terms and prove the following estimates:
\begin{align*}
& \abs{{\langle\mathcal{X}_{\Dx}(\eta_2,q_2),\varphi\rangle}}\leq C \, \norm{\varphi}_{\infty}\, \abs{k}_{BV}\, \norm{k}_{\infty},\\
& \abs{{\langle\mathcal{Y}_{\Dx}(\eta_2,q_2),\varphi\rangle}}\leq C \, \norm{\varphi}_{\infty}\,  \norm{k}_{\infty},\\
& \abs{{\langle\mathcal{Z}_{\Dx}(\eta_2,q_2),\varphi\rangle}}\leq C \, (\Dx)^{1/2}\,\norm{\varphi}_{H^1}\, \norm{k}_{\infty},
\end{align*}
with the help of \eqref{estimate-a}, and the a priori estimates \eqref{eq:est_final1}, and \eqref{eq:est_final2}. This implies that
\begin{align}
\label{eq:17}
\seq{ \mathcal{X}_{\Dx}(\eta_2,q_2)}_{\Dx> 0}, \seq{ \mathcal{Y}_{\Dx}(\eta_2,q_2)}_{\Dx> 0} \in \mathcal{M}_{\loc}(\Pi_T),
\end{align}
and 
\begin{align}
\label{eq:18}
\seq{ \mathcal{Z}_{\Dx}(\eta_2,q_2)}_{\Dx> 0} \, \text{ is compact in} \, \,H^{-1}_{\loc}(\Pi_T).
\end{align}
We use Lemma~\ref{lem:Murat} along with \eqref{eq:16}, \eqref{eq:17} and \eqref{eq:18} to conclude the compactness of  $\seq{ \mathcal{E}^{1,1,1,1}_{\Dx}(\eta,q)}_{\Dx>0}$ in $H^{-1}_{\loc}(\Pi_T)$, for $(\eta, q):=(\eta_2, q_2)$ as well.

Finally, the compactness of  $\seq{ \mathcal{E}^{1,1,1,1}_{\Dx}(\eta,q)}_{\Dx>0}$ along with \eqref{eq:13} ensures the compactness of  $\seq{ \mathcal{E}^{1,1,1}_{\Dx}(\eta,q)}_{\Dx>0}$ in $H^{-1}_{\loc}(\Pi_T)$.
\end{proof}

Next, we tackle the term $\mathcal{E}^{1,1,2}_{\Dx}(\eta,q)$.
\begin{lemma}\label{lem:compact4}
\begin{align}
\label{eq:19}
\seq{\mathcal{E}^{1,1,2}_{\Dx}(\eta,q)}_{\Dx>0} \, \text{ is compact in} \, \,H^{-1}_{\loc}(\Pi_T).
\end{align}
\end{lemma}
\begin{proof}
We first recall the expression for $\langle\mathcal{E}_{\Dx}^{1,1,2}(\eta,q),\varphi\rangle$
\begin{align*}
& \langle\mathcal{E}_{\Dx}^{1,1,2}(\eta,q),\varphi\rangle= \, \beta\, \sum_j \sum_{n=0}^{N-1}\int_{t^n}^{t^{n+1}}\int_{I_j} \eta_u(\km,u^n_j) \Dx \, (\Dp \Dm u^n_j) \, \,\varphi  \, dx dt.
\end{align*}
Using summation-by-parts formula, we can write
\begin {align*}
\langle\mathcal{E}_{\Dx}^{1,1,2}(\eta,q),\varphi\rangle&=- \underbrace{\beta\,\sum_{n=0}^{N-1} \int_{t^n}^{t^{n+1}}\sum_j\Dx\, \Dm u^n_j\  \Dm\left(\frac{\partial}{\partial u}\eta(\km, u^n_j)\ \right) \Phi_{j-1}\, dt}_{\langle\mathcal{E}_\Dx^{1,1,2,1}(\eta,q),\varphi\rangle}\\
&\hspace{3.8cm} -\underbrace{\beta\,\sum_{n=0}^{N-1} \int_{t^n}^{t^{n+1}}\sum_j \Dx\, \Dm u^n_j\ \frac{\partial}{\partial u}\eta(\km, u^n_j) \ \Dm \Phi_j\ dt}_{\langle\mathcal{E}_\Dx^{1,1,2,2}(\eta,q),\varphi\rangle}.
\end{align*}
Now we write $\mathcal{E}_{\Dx}^{1,1,2,1}(\eta,q)$ as 
\begin{align*}
&\langle\mathcal{E}_{\Dx}^{1,1,2,1}(\eta,q),\varphi\rangle= - \underbrace{\beta\,\sum_{n=0}^{N-1} \int_{t^n}^{t^{n+1}} \Dx \sum_j  \Dm u^n_j \frac{\eta_u(\km, u^n_j)- \eta_u(\km, u^n_{j-1})}{\Dx} \ \Phi_{j-1} \ dt}_{\langle \mathcal{A}_{\Dx}(\eta,q), \varphi \rangle}\\
&\hspace{3.5cm}-\underbrace{\beta\,\sum_{n=0}^{N-1} \int_{t^n}^{t^{n+1}}\sum_j \Dx\, \Dm u^n_j \frac{\eta_u(\km, u^n_{j-1})- \eta_u(k_{j-\frac32}, u^n_{j-1})}{\Dx} \ \Phi_{j-1} \ dt}_{\langle \mathcal{B}_{\Dx}(\eta,q), \varphi \rangle}.
\end{align*}
We estimate the terms $\mathcal{A}_{\Dx}(\eta,q)$ and $\mathcal{B}_{\Dx}(\eta,q)$, using a priori estimate \eqref{eq:est_final1}, as
\begin{align*}
&\abs{\langle \mathcal{B}_{\Dx}(\eta,q), \varphi \rangle} \leq \beta\,\norm{\varphi}_{\infty}\sum_{n=0}^{N-1} \int_{t^n}^{t^{n+1}}\sum_j  \Dx \, \abs{\Dm u^n_j}\, \abs{\km-k_{j-\frac32}}\, dt\\
& \qquad \qquad \leq \beta\,\norm{\varphi}_{\infty} \  T^{1/2}\ \left(\sum_{n=0}^{N-1} \int_{t^n}^{t^{n+1}} \sum_j (\Dx)^2 \ (\Dm u^n_j)^2\ dt \right)^{1/2} \left(\sum_j\abs{\km-k_{j-\frac32}}^2\right)^{1/2}\\
&\qquad \qquad \leq \ C\, \beta\,\norm{\varphi}_{\infty} \ \abs{k}_{BV}.
\end{align*}
and, similarly
\begin{align*}
\abs{\langle \mathcal{A}_{\Dx}(\eta,q), \varphi \rangle} \leq \beta\,\norm{\varphi}_{\infty} \ \Dx\sum_{n=0}^{N-1} \int_{t^n}^{t^{n+1}}\norm{\Dm u^n}^2 \, dt \leq \beta\,C\ \norm{\varphi}_{\infty}.
\end{align*}
Hence
\begin{align}
\label{eq:20}
\seq{\mathcal{A}_{\Dx}(\eta,q)}_{\Dx> 0}, \seq{\mathcal{B}_{\Dx}(\eta,q)}_{\Dx> 0}   \in \mathcal{M}_{\loc}(\Pi_T) \quad \implies \quad \seq{\mathcal{E}_{\Dx}^{1,1,2,1}(\eta,q)}_{\Dx> 0}   \in \mathcal{M}_{\loc}(\Pi_T).
\end{align}
Next consider $\mathcal{E}_{\Dx}^{1,1,2,2}(\eta,q)$. Using Cauchy-schwartz inequality along with the a priori estimate \eqref{eq:est_final1} and the estimate  \eqref{estimate-a}, we obtain
\begin{align*}
\langle\mathcal{E}_{\Dx}^{1,1,2,2}(\eta,q), \varphi\rangle&=-\beta\,\sum_{n=0}^{N-1} \int_{t^n}^{t^{n+1}}\sum_j \Dx\, \Dm u^n_j\ \frac{\partial}{\partial u}\eta(\km, u^n_j) \ \Dm \Phi_j\ dt\\
&\leq \beta\, \ C \sum_{n=0}^{N-1} \int_{t^n}^{t^{n+1}}\sum_j\Dx \abs{\Dm u^n_j}\ \abs{\Dm \Phi_j} \ dt\\
&\leq \beta\, C \ (\Dx)^{1/2}\ \left( \Dx\sum_{n=0}^{N-1} \int_{t^n}^{t^{n+1}}\norm{\Dm u^n}^2 \, dt\right)^{1/2} \ \norm{\varphi}_{H^1(\Pi_T)}\\
&\leq \beta\,C\,(\Dx)^{1/2}\, \norm{\varphi}_{H^1(\Pi_T)}.
\end{align*}
Therefore,
\begin{align}
\label{eq:21}
\seq{\mathcal{E}_{\Dx}^{1,1,2,2}(\eta,q)}_{\Dx> 0} \, \text{ is compact in} \, \,H^{-1}_{\loc}(\Pi_T).
\end{align}
Thus, \eqref{eq:20}, and \eqref{eq:21} ensure the compactness of $\seq{\mathcal{E}_{\Dx}^{1,1,2}(\eta,q)}_{\Dx> 0}$.
\end{proof}

Next, we focus on $\mathcal{E}^{1,1,3}_{\Dx}(\eta,q)$.
\begin{lemma}\label{lem:compact5}
\begin{align}
\label{eq:22}
\seq{\mathcal{E}^{1,1,3}_{\Dx}(\eta,q)}_{\Dx>0} \, \text{ is compact in} \, \,H^{-1}_{\loc}(\Pi_T).
\end{align}
\end{lemma}
\begin{proof}
Recall that
\begin{align*}
& \langle\mathcal{E}_{\Dx}^{1,1,3}(\eta,q),\varphi\rangle = \,\gamma\, \sum_j\sum_{n=0}^{N-1} \int_{t^n}^{t^{n+1}}\int_{I_j} \eta_u(\km,u^n_j) \mu(\Dx) \, (\Dpt\Dp \Dm u^n_j) \, \,\varphi  \, dx dt.
\end{align*}
Making use of the summation-by-parts formula and discrete Libnitz rule, we rewrite
\begin {align*}
\langle\mathcal{E}_{\Dx}^{1,1,3}(\eta,q),\varphi\rangle &=- \underbrace{\gamma\,\sum_{n=0}^{N-1} \int_{t^n}^{t^{n+1}}\sum_j
\mu(\Dx)\, \Dpt\Dm u^n_j\  \Dm\left(\frac{\partial}{\partial u}\eta(\km, u^n_j)\ \right) \Phi_{j-1}\, dt}_{\langle \mathcal{C}_{\Dx}(\eta,q), \varphi \rangle}\\
&\qquad  \qquad \qquad \quad -\underbrace{\gamma\,\sum_{n=0}^{N-1} \int_{t^n}^{t^{n+1}}\sum_j \mu(\Dx)\, \Dpt\Dm (u^n_j)\ \frac{\partial}{\partial u}\eta(\km, u^n_j) \ \Dm \Phi_j\ dt}_{\langle \mathcal{D}_{\Dx}(\eta,q), \varphi \rangle}.
\end{align*}
Again, making use of the a priori estimate \eqref{eq:est_final3} reveals that the term
$\mathcal{D}_{\Dx}(\eta,q)$ can be estimated as 
\begin{align*}
\abs{\langle \mathcal{D}_{\Dx}(\eta,q), \varphi \rangle} & \leq \gamma\,C \ (\Dx)^{\frac12} \ \left(\mu(\Dx) \Dx \Dt\sum_{n=0}^{N-1}\norm{\Dpt\Dm u^n}^2  \right)^{1/2} \norm{\varphi}_{H^1(\Pi_T)} 
 \leq \gamma\,C(\Dx)^{\frac12} \norm{\varphi}_{H^1(\Pi_T)},
\end{align*}
so that
\begin{align}
\label{eq:23}
\seq{\mathcal{D}_{\Dx}(\eta,q)}_{\Dx> 0} \, \text{ is compact in} \, \,H^{-1}_{\loc}(\Pi_T).
\end{align}
On the other hand, to estimate $\mathcal{C}_{\Dx}(\eta,q)$ term, we first split that term as
\begin{align*}
&\langle \mathcal{C}_{\Dx}(\eta,q), \varphi \rangle\\
&\qquad \qquad = - \underbrace{\gamma\,\sum_{n,j} \int_{t^n}^{t^{n+1}}\mu(\Dx)\, \Dpt\Dm u^n_j \frac{\eta_u(\km, u^n_j)- \eta_u(\km, u^n_{j-1})}{\Dx} \ \Phi_{j-1} \ dt}_{\langle \mathcal{G}_{\Dx}(\eta,q), \varphi \rangle}\\
&\qquad \qquad \qquad \qquad \qquad  -\underbrace{\gamma\,\sum_{n,j} \int_{t^n}^{t^{n+1}} \mu(\Dx)\, \Dpt\Dm u^n_j \frac{\eta_u(\km, u^n_{j-1})- \eta_u(k_{j-\frac32}, u^n_{j-1})}{\Dx} \ \Phi_{j-1} \ dt}_{\langle \mathcal{F}_{\Dx}(\eta,q), \varphi \rangle}.
\end{align*}
A similar type of argument, as used before, can be used to deal with the terms $\mathcal{G}_{\Dx}$ and $\mathcal{F}_{\Dx}$,
with the help of the a priori estimate \eqref{eq:est_final3}, returns
\begin{align*}
&\abs{\langle \mathcal{G}_{\Dx}(\eta,q), \varphi \rangle} \\
& \leq 2\,\gamma\,  \norm{\varphi}_{L^{\infty}(\Pi_T)} \left(\mu(\Dx)\sum_{n=0}^{N}\Dt \Dx \norm{\Dpt \Dm u^n}^2\right)^{1/2} 
\left(\sum_{n=0}^{N}\Dt \Dx \norm{\Dm u^n}^2\right)^{1/2}  
 \leq \gamma\,C \,  \norm{\varphi}_{L^{\infty}(\Pi_T)},
\end{align*}
and
\begin{align*}
\abs{\langle \mathcal{F}_{\Dx}(\eta,q), \varphi \rangle} & \leq \gamma\,\text{BV}(k)\, \norm{\varphi}_{L^{\infty}(\Pi_T)}\ \left(\mu(\Dx)\Dx \Dt\sum_{n=0}^{N-1}\norm{\Dpt\Dm u^n}^2  \right)^{1/2}
\leq \gamma\,C \,  \norm{\varphi}_{L^{\infty}(\Pi_T)}.
\end{align*}
This implies that
\begin{align}
\label{eq:24}
\seq{\mathcal{G}_{\Dx}(\eta,q)}_{\Dx> 0}, \seq{\mathcal{F}_{\Dx}(\eta,q)}_{\Dx> 0} \in \mathcal{M}_{\loc}(\Pi_T).
\end{align}
Therefore, in view of the Lemma~\ref{lem:Murat} along with \eqref{eq:23}, and \eqref{eq:24}, we reach at the conclusion
\begin{align}
\seq{\mathcal{E}_{\Dx}^{1,1,3}(\eta,q)}_{\Dx> 0} \, \text{ is compact in} \, \,H^{-1}_{\loc}(\Pi_T).
\end{align}
\end{proof}

Making use of Lemma~\ref{lem:compact3},\ref{lem:compact4} and \ref{lem:compact5}, we can finally prove the following.

\begin{lemma}\label{lem:compact6}
\begin{align}
\label{eq:25}
\seq{\mathcal{E}^{1,1}_{\Dx}(\eta,q)}_{\Dx>0} \, \text{ is compact in} \, \,H^{-1}_{\loc}(\Pi_T).
\end{align}
\end{lemma}
\begin{proof}
We have already shown that $\seq{\mathcal{E}^{1,1,1}_{\Dx}(\eta,q)}_{\Dx>0}$, $\seq{\mathcal{E}^{1,1,2}_{\Dx}(\eta,q)}_{\Dx>0}$ and $\seq{\mathcal{E}^{1,1,3}_{\Dx}(\eta,q)}_{\Dx>0}$ are compact $H^{-1}(\Pi_T)$. Thus, we only need to find an estimate for $\mathcal{E}^{1,1,4}_{\Dx}(\eta,q)$. using the a priori bound \eqref{eq:est_final1}, we have
\begin{align*}
\abs{ \langle\mathcal{E}_{\Dx}^{1,1,4}(\eta,q),\varphi\rangle}\leq C \, \norm{\varphi}_{\infty}\left(\sum_{n=0}^N\Dt \Dx\norm{\Dpt u^n}^2\right)
\leq C \, \norm{\varphi}_{\infty},
\end{align*}
so that
\begin{align}
\label{eq:26}
\seq{\mathcal{E}_{\Dx}^{1,1,4}(\eta,q)}_{\Dx> 0} \in \mathcal{M}_{\loc}(\Pi_T).
\end{align}
An application of Lemma~\ref{lem:Murat} allows us to conclude \eqref{eq:25}, thus proving the lemma.
\end{proof}

Combining Lemma~\ref{lem:compact2} and \ref{lem:compact6} ensures that
\begin{align}
\label{eq:27}
\seq{\mathcal{E}_{\Dx}^{1}(\eta,q)}_{\Dx> 0}  \, \text{ is compact in} \, \,H^{-1}_{\loc}(\Pi_T).
\end{align}
Thus, Lemma~\ref{lem:compact1} and \eqref{eq:27} justify the assumptions we had made in the proof of Lemma~\ref{lem:compact}.

Now we are in a position to state the ``convergence theorem'' which guarantees the convergence of
approximate solutions ${\lbrace u_{\Dx} \rbrace}_{\Dx>0}$, generated by the scheme \eqref{eq:scheme}, to 
a weak solution of \eqref{eq:discont}. The following theorem can also be viewed as a modified version of the classical Lax-Wendroff theorem (for more details, consult the monograph by LeVeque \cite{lev}).
\begin{theorem}
\label{thm:theorem1}
Let $u_{\Dx}$ be a sequence of approximations generated via the scheme \eqref{eq:scheme} with $\mu(\Dx) = \bigO{\Dx^2}$.
Then there exists a function $u \in L^{\infty}([0,T];L^1_{\loc}(\R))$ and a subsequence of $\lbrace \Dx\rbrace$ (not relabeled) such that $u_{\Dx} \mapsto u$ as $\Dx \downarrow 0$. Moreover, the function $u$
is a weak solution to \eqref{eq:discont}.
\end{theorem}

\begin{proof}
The strong convergence of $u_{\Dx}$ to a function $u$ immediately follows from the 
Theorem~\ref{thm:compcomp} and the
Lemma~\ref{lem:compact}. It remains to show that $u$ is a weak solution. 
Let $\psi \in C_0^{\infty}(\R \times [0,T))$ be any test function and denote $\psi_j^n = \psi(x_j, t^n)$.
Multiplying the scheme \eqref{eq:scheme} by $\Dx \Dt \psi_j^n$, and subsequently suming over all $j$ and $n$ yields
\begin{align*}
&\Dx \Dt \sum_{j} \sum_{n} \psi_j^n \,D^t_{+} u^n_j  + \Dx \Dt \sum_{j} \sum_{n} \psi_j^n \,\Dm \left(  k_{j+\frac12} \hat{f}_{j+\frac12}\right) \\
&\qquad \qquad \qquad \qquad = \beta\, \Dx \Dt \sum_{j} \sum_{n}  \Dx \,\psi_j^n\,  \Dp\Dm u^n_j  +\gamma\, \Dx \Dt \sum_{j} \sum_{n} \mu(\Dx)\, \psi_j^n \,D^t_{+} \Dp \Dm u^n_j.  
\end{align*}
By standard arguments, it is clear that
\begin{align*}
\Dx \Dt \sum_{j} \sum_{n} \psi_j^n \,D^t_{+} u^n_j  
&= - \Dx \Dt \sum_{j} \sum_{n} D^t_{-} \psi_j^n \, u^n_j - \Dx  \sum_{j}  \psi_j^0 \, u^0_j \\
& \quad \mapsto -\int_{\R} \int_0^T u \psi_t \,dx\,dt - \int_{\R} \psi(x,0) u_0(x) \,dx, \, \text{as} \, \Dx \downarrow 0.
\end{align*}
Next, using the a priori bound \eqref{eq:est_final1}, we conclude
\begin{align*}
\beta\,\Dx \Dt \sum_{j} \sum_{n}  \Dx \,\psi_j^n\,  &\Dp\Dm u^n_j  =  -\beta\,\Dx \Dt \sum_{j} \sum_{n}  \Dx \,\Dm \psi_j^n\,  \Dm u^n_j  \\
& \le\beta\, \Dx \Dt \sum_{n} \left( \Dx \sum_{j} (\Dm u^n_j)^2\right)^{1/2} \,  \left(\Dx \sum_{j} (\Dm \psi^n_j)^2 \right)^{1/2} \\
& \le\beta\, \left( \Dx^2 \Dt \sum_{n}\sum_{j} (\Dm u^n_j)^2\right)^{1/2} \,  \left(\Dx^2 \Dt \sum_{n} \sum_{j} (\Dm \psi^n_j)^2 \right)^{1/2} \\
& \le \beta\,C\,(\Dx)^{1/2} \, \norm{\psi}_{L^2((0,T);H^1(\R))} \mapsto 0, \,\text{as}\, \Dx \downarrow 0.
\end{align*}
Repeated use of Cauchy-Schwartz inequality along with the help of a priori bound \eqref{eq:est_final3} returns 
\begin{align*}
\gamma\,\Dx \Dt \sum_{j} \sum_{n} & \mu(\Dx)\,  \psi_j^n \,D^t_{+} \Dp \Dm u^n_j  \\
&\le    \gamma\,C\, \frac{\mu(\Dx)^{1/2}}{\Dx^{1/2}}    \left(\sum_{j} \sum_{n} \mu(\Dx) \Dx^2 \Dt \abs{D^t_{+} \Dm u^n_j }^2 \right)^{1/2}     \left(\sum_{j} \sum_{n}  \Dx \Dt \abs{ \Dm \psi^n_j }^2 \right)^{1/2}\\
& \le \gamma\,C\, \frac{\mu(\Dx)^{1/2}}{\Dx^{1/2}} \, \norm{\psi}_{L^2((0,T);H^1(\R))} \mapsto 0, \,\text{as}\, \Dx \downarrow 0.
\end{align*}
Finally, a simple use of summation-by-parts formula implies
\begin{align*}
\Dx \Dt \sum_{j} \sum_{n} \psi_j^n & \, \Dm \left(  k_{j+\frac12} \hat{f}^n_{j+\frac12}\right) = - \Dx \Dt \sum_{j} \sum_{n} \Dp \psi_j^n \, k_{j+\frac12} \hat{f}^n_{j+\frac12} \\
& = - \underbrace{\Dt \sum_{j,n} \int_{I_j}\Dp \psi_j^n \, k(x) \, f(u_j) \,dx}_{\mathcal{E}^1_{\Dx}} - \underbrace{ \Dt \sum_{j,n}\int_{I_j}   \Dp \psi_j^n \, \left(k_{j+\frac12} -k(x)\right)\hat{f}^n_{j+\frac12}}_{\mathcal{E}^2_{\Dx}}\,dx  \\
& \qquad \qquad \qquad \qquad \qquad-\underbrace{ \Dt   \sum_{j} \sum_{n} \int_{I_j}\Dp \psi_j^n \, k(x) \left(\hat{f}^n_{j+\frac12} -f(u_j) \right)\,dx}_{\mathcal{E}^3_{\Dx}},
\end{align*}
where, again, a standard argument reveals that
\begin{align*}
\mathcal{E}^1_{\Dx} \mapsto -\int_{\R} \int_0^T k(x)\,f(u)\,\psi_x \,dx\,dt, \, \text{as}\, \Dx \downarrow 0.
\end{align*}
Now observe that
\begin{align*}
\sum_{j} \int_{I_j} \abs{k(x) -k_{j+\frac12}} \,dx  \le \int_{\R} \abs{k(x) -k_{\Dx}(x)}\,dx + \frac{\Dx}{2} \int_{\R} \frac {\abs{k(x+\Dx/2) -k(x)}}{\Dx/2} \,dx 
\end{align*}
and consequently
\begin{align*} 
\sum_{j} \int_{I_j} \abs{k(x) -k_{j+\frac12}} \,dx \mapsto 0 \, \text{as}\,\, \Dx \downarrow 0.
\end{align*}
Keeping this in mind, we find that
\begin{align*}
\mathcal{E}^2_{\Dx} &= \Dt   \sum_{j,n}\int_{I_j} \Dp \psi_j^n \, \left(k_{j+\frac12} -k(x)\right)\hat{f}^n_{j+\frac12}\,dx \\
& \le M  \norm{\psi}_{L^1([0,T];L^{\infty}(\R))} \sum_{j} \int_{I_j}\abs{k(x) -k_{j+\frac12}} \,dx\mapsto 0 \, \text{as}\,\, \Dx \downarrow 0.
\end{align*}
Finally, making use of the a priori bound \eqref{eq:est_final1}, the last term can be estimated as follows:
\begin{align*}
\mathcal{E}^3_{\Dx} &=\Dt   \sum_{j} \sum_{n} \int_{I_j}\Dp \psi_j^n \, k(x) \left(\hat{f}^n_{j+\frac12} -f(u_j) \right)\,dx \\
& \le M  (\Dx)^{1/2} \norm{\psi}_{L^2([0,T];H^1(\R))} \left(\Dx^2 \Dt \sum_{j} \sum_{n} (\Dm u^n_j)^2\right)^{1/2} \mapsto 0 \, \text{as}\,\, \Dx \downarrow 0.
\end{align*}
To sum up, we have proved that $u$ is a weak solution of \eqref{eq:discont}, i.e.,
\begin{align*}
\int_{\R} \int_0^T \left(  \psi_t u + \psi_x k(x) f(u) \right) \,dx\,dt + \int_{\R} \psi(x,0) u_0(x) \,dx =0,\, \text{for all}\,\, \psi \in C_0^{\infty} (\R \times [0,T)).
\end{align*}

\end{proof}

\section{Entropy Solution}
\label{sec:entropy}
As we have already mentioned in Remark~\ref{rem:EO},
we use the Engquist-Osher scheme to make the analysis more
concrete, but our methods can easily be adapted to general monotone schemes.
Drawing preliminary motivation from 
Karlsen et al. \cite{kenneth1} and Towers \cite{towers1}, we first show that the limit 
solution satisfies the Kru\v{z}kov entropy inequalities locally, 
away from the jumps in $k$. Then we proceed to show that the limit solution satisfies Kru\v{z}kov type entropy
inequalities when the test function has support which intersects one or more 
jumps in $k$, and finally we show that this implies Kru\v{z}kov entropy solution.

To proceed further, we need some additional regularity assumptions on the discontinuous coefficient $k(x)$.
We assume that $k(x)$ is piecewise Lipschitz continuous in $\R$ with
finitely many jumps (in $k$ and $k'$), located at $\xi_1, \xi_2, \cdots, \xi_M$. More specifically, we
assume that there are finitely many Lipschitz continuous curves $\omega_1, \omega_2, \cdots, \omega_M$ such that $\xi_i \in \omega_i$, the union of which we
denote by
\begin{align*}
\bigcup_{i=1}^M \omega_i =\Omega.
\end{align*}
For the sake of simplicity, we also assume that none of the curves intersect.
The curves $\omega_1, \omega_2, \cdots, \omega_M$ partition $\R \setminus \Omega$
into a finite union of open sets:
\begin{align*}
\R \setminus \Omega = \mathcal{R}_0 \cup \mathcal{R}_1 \cup \cdots \mathcal{R}_M,
\end{align*}
with the curve $\omega_m$ separating the sets $\mathcal{R}_{m-1}$ and $\mathcal{R}_m$.
We will assume that
\begin{align*}
k \in \mathrm{Lip}(\overline{\mathcal{R}}_m), \quad m=1,2,\cdots,M.
\end{align*}
With this assumption, $k$ has well defined limits from the right and left along each of the curves
$\omega_m$, $1\le m \le M$, and we denote these limits by $k_m^{\pm}:= k(\xi_m^{\pm})$ respectively.

To this end, we first establish the following entropy inequality for smooth entropy 
function $\eta$, to avoid complications arising from the discontinuity $\eta'(u) = \sgn{u-c}$ in Kru\v{z}kov entropy.

\begin{lemma}
\label{thm:entropy1}
Let $(\eta, Q)$ be a convex entropy pair with $\eta$ being a $C^2$-function and $\eta'(u) = Q'(u) \,f'(u)$.
For the test function $\psi \ge 0$ with compact
support in $t>0$, $x \in \R \setminus \Omega$, and every $c \in \R$, the following 
entropy inequality holds
\begin{align*}
\int_{\R} \int_0^T \Big(\eta(u) \,\psi_t + k \,Q(u)\, \psi_x \Big)\,dx\,dt + \int_{\R} & \eta(u_0) \psi(x,0)\,dx \\
& \quad- \int_{\R} \int_0^T  k'(x)\,\Big(\eta'(u) \,f(u) -  Q(u) \Big) \psi\,\,dx\,dt \ge 0.
\end{align*}
\end{lemma}

\begin{proof}

We begin by rewriting the scheme \eqref{eq:scheme} as
\begin{align*}
u^{n+1}_j = w^n_j + \beta \Dx \Dt \Dp \Dm u^n_j + \gamma\, \Dt \,\mu(\Dx) D_{t}^{+} \Dp \Dm u^n_j,
\end{align*}
with
\begin{align*}
w^n_j  = u^n_j - \lambda \ \Delta_+ \left(  k_{j-\frac12} \hat{f}^n_{j-\frac12}\right),
\end{align*}
where $\Delta_+$ denote the undivided forward difference operator, and $ \hat{f}^n_{j+\frac12}$ is the EO flux corresponding to the flux function $f(u)$. Let $H_{j+\frac12}$ be a entropy flux function, consistent with $Q(u)$. 
Then, in light of Kenneth et al. \cite[Lemma 4.1]{kenneth1}, it is evident that for any $c \in \R$
\begin{align*}
\abs{w^n_j -c} \le \abs{u^n_j -c} - \lambda \left(  k_{j+\frac12} H_{j+\frac12} - k_{j-\frac12} H_{j-\frac12} \right)
- \lambda \,\sgn{w^n_j -c} f(c)\, \Delta_+\km.
\end{align*}
To obtain a similar type inequality for a smooth entropy-entropy flux pair, we follow a classical approximation argument.
For a rigorous proof, we refer to the paper by Karlsen et al. \cite[Lemma 5.7]{Triang_Cocliteetal}. In what follows, we have the following inequality:
\begin{align*}
\eta(w^n_j) \le \eta(u^n_j) - \lambda \left(  k_{j+\frac12} H_{j+\frac12} - k_{j-\frac12} H_{j-\frac12} \right)
+\lambda \left(Q(w^n_j)-\eta'(w^n_j)f(w^n_j) \right) \Delta_+\km
\end{align*}
After rearranging terms, a discrete entropy inequality for the scheme results
\begin{align*}
\eta^{n+1}_j\le \eta^n_j + \left(\eta^{n+1}_j -\eta(w^n_j)\right)  - \lambda \Delta_{+} \left( k_{j-\frac12} H_{j-\frac12} \right) + \lambda \, \xi^n_j \Delta_{+} k_{j-\frac12}.
\end{align*}
where $\xi^n_j=Q(w^n_j)-\eta'(w^n_j)f(w^n_j)$ and $\eta^n_j=\eta(u^n_j)$.
Next, multiplying the above inequality by $\Dx \Dt \psi_j^n $ with $\psi_j^n=\psi(x_j,t_n)$, where $\psi$ smooth, non-negative test function $\psi$
with compact support in $(\R \setminus \Omega) \times [0,T)$ and using summation by parts yields
\begin{align*}
\Dx \Dt \sum_{j} \sum_{n}   \psi^n_j\,D^t_{+}\eta^n_j  - \Dx \Dt & \sum_{j}  \sum_{n}  k_{j-\frac12} H_{j-\frac12} \Dm \psi^n_j \\
& - \Dx \Dt \sum_{j} \sum_{n} \psi_j^n \left(\xi^n_j \Dp k_{j-\frac12} + \left(\eta^{n+1}_j -\eta(w^n_j)\right)/\Dt  \right) \le 0.
\end{align*}
As before, a standard argument reveals that
\begin{align*}
\Dx \Dt \sum_{j} \sum_{n} \psi_j^n \,& D^t_{+} \eta(u^n_j)  
= - \Dx \Dt \sum_{j} \sum_{n} D^t_{-} \psi_j^n \, \eta(u^n_j) - \Dx  \sum_{j}  \psi_j^0 \, \eta(u^0_j) \\
& \qquad \qquad \quad \mapsto -\int_{\R} \int_0^T \eta(u) \psi_t \,dx\,dt - \int_{\R} \psi(x,0) \eta(u_0(x)) \,dx, \, \text{as} \, \Dx \downarrow 0.
\end{align*}
Next, we rewrite 
\begin{align*}
\Dx \Dt \sum_{j} \sum_{n}  k_{j+\frac12} H_{j+\frac12} \Dm \psi^n_j &= \underbrace{\Dx \Dt \sum_{j} \sum_{n}  k_{j+\frac12} \left(H_{j+\frac12} - Q(u^n_j) \right)\Dm \psi^n_j }_{\mathcal{E}^1_{\Dx}}\\
& \qquad \qquad \qquad \qquad +\underbrace{\Dx \Dt \sum_{j} \sum_{n}  k_{j+\frac12} Q(u^n_j) \Dm \psi^n_j, }_{\mathcal{E}^2_{\Dx}}
\end{align*}
where it is straightforward to conclude that
\begin{align*}
\mathcal{E}^2_{\Dx} \mapsto \int_{\R} \int_{0}^{T} k(x) \,Q(u)\, \psi_x \,dx\,dt, \, \text{as} \, \Dx \downarrow 0,
\end{align*}
and the a priori estimate \eqref{eq:est_final1} gives that
\begin{align*}
\abs{\mathcal{E}^1_{\Dx}}\le C (\Dx)^{1/2} \left(\Dx^2\Dt \sum_{n,j}\abs{\Dm u^n_j}^2 \right)^{1/2}
\norm{\psi}_{L^2([0,T]; H^1(\R))} \mapsto 0 \, \text{as} \, \Dx \downarrow 0.
\end{align*}
To estimate the term involving $\Dp \km$, we split the term as follows
\begin{align*}
\Dx \Dt   \sum_n \sum_j    & \psi^n_j    \xi^n_j        \Dp\km 
= \Dx \Dt    \sum_n \sum_j    \left( Q(w^n_j)-\eta'(w^n_j)f(w^n_j)\right)    \psi^n_j     \Dp \km\\
& =\underbrace{\Dx \Dt   \sum_n \sum_j  \left( Q(u^n_j)-\eta'(u^n_j)f(u^n_j)\right)    \psi^n_j     \Dp \km}_{\mathcal{E}^a_{\Dx}}\\
& \hspace{2cm} +\underbrace{ \Dx \Dt   \sum_n \sum_j  \left( Q(w^n_j)-Q(u^n_j)\right)    \psi^n_j     \Dp \km}_{\mathcal{E}^b_{\Dx}}\\
&\hspace{3cm} +\underbrace{\Dx \Dt    \sum_n \sum_j    \left( \eta'(w^n_j)f(w^n_j)-\eta'(u^n_j)f(u^n_j)\right)    \psi^n_j     \Dp \km}_{\mathcal{E}^c_{\Dx}}.
\end{align*}
Again, a straightforward argument shows that 
\begin{align*}
\mathcal{E}^a_{\Dx} \mapsto \int_{\R} \int_{0}^{T} \left( Q(u) - \eta'(u) f(u) \right) k'(x) \psi dx dt, \textrm{ as }
\Dx  \downarrow 0.
\end{align*}
For the other terms, we intend to show that 
\begin{align*}
\mathcal{E}^b_{\Dx},  \mathcal{E}^c_{\Dx} \mapsto 0 \textrm{ as } \Dx \downarrow 0.
\end{align*}
To see this, first note that
\begin{equation}
\label{eq:diff}
\begin{aligned}
\abs{Q(w^n_j)-Q(u^n_j)} \leq C \abs{w^n_j - u^n_j}
&\leq C \lambda \abs{\kp \fhp - \km \fhm}\\
& \leq C \lambda \left( \abs{\kp-\km} + \abs{u^n_{j+1}-u^n_j} + \abs{u^n_{j-1}-u^n_j}\right)
\end{aligned}
\end{equation}
Since $k(x)$ is Lipschitz continuous within the support of $\psi$, it is clear tha
\begin{align*}
\Dx \Dt   \sum_n \sum_j    \psi^n_j  \abs{\kp-\km} \abs{\Dp \km} 
\le \norm{\psi}_{\infty} T  \Dx   \norm{k'}_{\infty}
\abs{k}_{BV} \mapsto 0 \text{ as } \Dx \downarrow 0,
\end{align*}
and
\begin{align*}
&\Dx \Dt   \sum_n \sum_j    \psi^n_j  \abs{\Dp \km} \abs{u^n_{j+1}-u^n_j}\\
 &\hspace{3cm} \leq (\Dx)^{1/2} \left( \Dx^2 \Dt \sum_{n,j}
  \abs{\Dm u^n_j}^2\right)^{1/2}  \norm{k'}_{\infty}  \norm{\psi}_{L^2([0,T]; L^2(\R))}  \mapsto 0 \text{ as } \Dx \downarrow 0.
\end{align*}
This proves that $\mathcal{E}^b_{\Dx} \mapsto 0$ as $\Dx \mapsto 0$. A similar calculations show that 
 $\mathcal{E}^c_{\Dx} \mapsto 0$ as $\Dx \mapsto 0$.
 
Finally, the last term can be estimated via Taylor series as follows
\begin{align*}
&\Dx \Dt \sum_{n} \sum_{j} \psi^n_j \left(\eta^{n+1}_j -\eta(w^n_j)\right)/\Dt \\
& \qquad = \beta\,\Dx \Dt \sum_{j} \sum_{n}  \psi^n_j \eta'(u^n_j) \Dx \Dp \Dm u^n_j 
+ \gamma\,\Dx \Dt \sum_{j} \sum_{n}  \psi^n_j \eta'(u^n_j) \mu(\Dx) D_{t}^{+}  \Dp \Dm u^n_j  \\
& \qquad-  \Dx \sum_{j,n}   \lambda  \,\eta''(\theta_2) \, \psi^n_j \,\Delta_+\left(\km \fhm\right)
\,(u^{n+1}_j - w^n_j) 
 + \frac12 \Dx \sum_{j,n}  \psi^n_j \,  \eta''(\theta_1) \, \left(  u_j^{n+1} -w^n_j \right)^2 \\
& \qquad := \mathcal{Q}^1_{\Dx} + \mathcal{Q}^2_{\Dx} + \mathcal{Q}^3_{\Dx} + \mathcal{Q}^4_{\Dx}.
\end{align*}
For the first term $\mathcal{Q}^1_{\Dx}$, using the discrete chain rule, we proceed as follows
\begin{align*}
\mathcal{Q}^1_{\Dx}:=  \beta\,\Dx \Dt \sum_{j} \sum_{n}  \psi^n_j \eta'(u^n_j) \,  \Dx & \Dp \Dm u^n_j = \underbrace{\beta\,\Dx \Dt \sum_{j} \sum_{n} \Dx\, \psi^n_j \, \Dp \left( \eta'(u^n_j)\, \Dm u^n_j \right) }_{\mathcal{E}^3_{\Dx}}\\
& \qquad \qquad  \qquad \qquad - \underbrace{\beta\,\Dx \Dt \sum_{j} \sum_{n}  \Dx\,\psi^n_j \, \Dp u^n_j \,  \eta''(\theta^n_j)\, \Dp u^n_j}_{\mathcal{E}^4_{\Dx}},
\end{align*}
where using the non-negativity of the test function $\psi$, we conclude that
\begin{align*}
\mathcal{E}^4_{\Dx} \ge 0.
\end{align*}
Moreover, using the a priori bound \eqref{eq:est_final1}, we find
\begin{align*}
\abs{\mathcal{E}^3_{\Dx}} \le\beta\, C (\Dx)^{1/2} \left(\Dx^2 \Dt \sum_{j} \sum_{n} (\Dm u^n_j)^2 \right)^{1/2}\, \norm{\psi}_{L^2([0,T]; H^1(\R))} \mapsto 0 \, \text{as} \, \Dx \downarrow 0.
\end{align*}
For the next term $\mathcal{Q}^2_{\Dx}$, we proceed as follows:
\begin{align*}
\mathcal{Q}^2_{\Dx}:= \gamma\,\Dx \Dt \sum_{j} \sum_{n}  \psi^n_j \eta'(u^n_j) \, \mu(\Dx) & D_{t}^{+} \Dp \Dm u^n_j = \underbrace{\gamma\,\Dx \Dt \sum_{j,n} \mu(\Dx)\, \psi^n_j \, \Dp \left( \eta'(u^n_j)\, D_{t}^{+} \Dm u^n_j \right)}_{\mathcal{E}^5_{\Dx}} \\
& \qquad \qquad -\underbrace{ \gamma\,\Dx \Dt \sum_{j} \sum_{n}  \mu(\Dx) \,\psi^n_j \, D_{t}^{+} \Dp u^n_j \,  \eta''(\theta^n_j)\, \Dp u^n_j}_{\mathcal{E}^6_{\Dx}}.
\end{align*}
Making use of the a priori bound \eqref{eq:est_final3}, we conclude
\begin{align*}
\abs{\mathcal{E}^5_{\Dx}} \le \gamma\,C     \frac{\mu(\Dx)^{1/2}}{\Dx^{1/2}}\left(\mu(\Dx)\Dx^2 \Dt \sum_{j} \sum_{n} (D^t_{+}\Dm u^n_j)^2 \right)^{1/2}\, \norm{\psi}_{L^2([0,T]; H^1(\R))} \mapsto 0 \, \text{as} \, \Dx \downarrow 0,
\end{align*}
and
\begin{align*}
\abs{\mathcal{E}^6_{\Dx}} &\le\gamma\, C        \frac{\mu(\Dx)^{1/2}}{\Dx} \left( \mu(\Dx)\Dx^2 \Dt \sum_{j} \sum_{n} (D^t_{+}\Dm u^n_j)^2 \right)^{1/2}\, \left(\Dx^2 \Dt \sum_{j} \sum_{n} (\Dm u^n_j)^2 \right)^{1/2}\, \norm{\psi}_{\infty} \\
& \mapsto 0 \, \text{as} \, \Dx \downarrow 0,
\end{align*}
Next, we turn our focus on the term $\mathcal{Q}^3_{\Dx}$. In fact, we write
\begin{align*}
\mathcal{Q}^3_{\Dx}:= \Dx &\sum_{j} \sum_{n}  \lambda  \,\eta''(\theta_2) \, \psi^n_j \,\Delta_+\left(\km\fhm\right) \,(u^{n+1}_j - w^n_j) \\
& \qquad=\underbrace{\beta\,\Dx      \sum_{j} \sum_{n}       \lambda \,      \eta''(\theta_2) \,      \psi^n_j \,      
\Delta_+\left(\km\fhm\right)\,     \Dx \Dt     \Dp \Dm u^n_j }_{\mathcal{E}^7_{\Dx}}\\
& \qquad \qquad \quad +\underbrace{ \gamma\,\Dx \sum_{j} \sum_{n}  \lambda  \,\eta''(\theta_2) \, \psi^n_j \,\Delta_+\left(\km\fhm\right)\,        \Dt \mu(\Dx) \, 
D_{t}^{+} \Dp \Dm u^n_j}_{\mathcal{E}^8_{\Dx}}.
\end{align*}
Before we proceed, we recall that in accordance to Remark~\ref{rem:mu} $\lambda = \zeta \mu(\Dx)/(\Dx)^2$, or in other words $\Dt = \zeta \mu(\Dx)/\Dx$.  We also need to use the identity $\Dx \Dp \Dm u^n_j = \Dp u^n_j - \Dm u^n_j$. Thus, using the discrete chain rule for the term $\Delta_+\left(\km\fhm\right)$ and  apply Cauchy-Schwartz inequality repeatedly, we obtain
\begin{align*}
\abs{\mathcal{E}^7_{\Dx}}
 & \le \beta\, C \frac{\mu(\Dx)}{\Dx^2} \norm{\psi}_{\infty}   \left[ \abs{k}_{BV}      
                     \left(  \Dx^2 \Dt     \sum_{j} \sum_{n}    (\Dm u^n_j)^2    \right)^{1/2}        \right. \\
&      \hspace{6cm}  \left. + \norm{k}_{\infty} \Dx^2 \Dt     \sum_{j} \sum_{n}    (\Dm u^n_j)^2 \right] \mapsto 0 \, \text{as} \, \Dx \downarrow 0.
\end{align*}
Similarly,
\begin{align*}
\abs{\mathcal{E}^8_{\Dx}}
 & \le \gamma\,C \frac{\mu(\Dx)}{\Dx^2} \norm{\psi}_{\infty}   \left[\abs{k}_{BV}     
                     \left(  \Dx^2 \mu(\Dx) \Dt     \sum_{j} \sum_{n}    (\Dpt\Dm u^n_j)^2    \right)^{1/2}        \right. \\
&      \hspace{5cm}  \left. + \norm{k}_{\infty} \Dx^2 \mu(\Dx) \Dt     \sum_{j} \sum_{n}    (\Dpt\Dm u^n_j)^2 \right] \mapsto 0 \, \text{as} \, \Dx \downarrow 0.
\end{align*}
Finally, we are left with the term $\mathcal{Q}^4_{\Dx}$. We see that 
\begin{align*}
\mathcal{Q}^4_{\Dx}:=  \frac12 \Dx \sum_{j} \sum_{n}   \psi^n_j \,  \eta''(\theta_1) \,  & \left(  u_j^{n+1} -w^n_j \right)^2 \\
& \le \underbrace{\beta^2\Dx \sum_{j} \sum_{n}  \psi^n_j \,  \eta''(\theta_1) \,
                                \left(  \Dx \Dt \Dp \Dm u^n_j \right)^2 }_{\mathcal{E}^{9}_{\Dx}}\\
& \hspace{2.5cm} + \underbrace{ \gamma^2\,\Dx \sum_{j} \sum_{n}  \psi^n_j \,  \eta''(\theta_1) \, \left(  \mu(\Dx) \Dt D_{t}^{+}\Dp \Dm u^n_j \right)^2}_{\mathcal{E}^{10}_{\Dx}}.
\end{align*}
We start with the first term $\mathcal{E}^9_{\Dx}$. Making use of the a priori estimate \eqref{eq:est_final1}, we conclude
\begin{align*}
\abs{\mathcal{E}^{9}_{\Dx}}  \le C \frac{\mu(\Dx)}{\Dx^2}     \left(\Dx^2 \Dt \sum_{j} \sum_{n} (\Dm u^n_j)^2 \right)  \,\norm{\psi}_{\infty} \mapsto 0 \, \text{as} \, \Dx \downarrow 0.
\end{align*}
Similarly, making use of the a priori estimate \eqref{eq:est_final3}, we argue that
\begin{align*}
\abs{\mathcal{E}^{10}_{\Dx}}  \le C \frac{\mu(\Dx)^2}{\Dx^4}     \left(\Dx^2 \mu(\Dx)\Dt \sum_{j} \sum_{n} (\Dpt\Dm u^n_j)^2 \right)  \,\norm{\psi}_{\infty} \mapsto 0 \, \text{as} \, \Dx \downarrow 0.
\end{align*}

\end{proof}

Now in view of the above result, we are ready to prove the following lemma:
\begin{lemma}
\label{thm:entropy11}
Let $u(x,t)$ be a weak solution constructed as the limit of the approximation $u_{\Dx}$ generated
by the scheme \eqref{eq:scheme} with $ \mu(\Dx)= \scalebox{1.1}{$\scriptstyle\mathcal{O}$}({\Dx^2})$.
For the test function $\psi \ge 0$ with compact
support in $t>0$, $x \in \R \setminus \Omega$, and every $c \in \R$, the following 
entropy inequality holds
\begin{align*}
\int_{\R} \int_0^T & \left( \abs{u-c} \psi_t  + \sgn{ u-c}  k(x) \left(f(u) -f(c) \right) \psi_x \right) \,dx\,dt \\
& \qquad \qquad \qquad \quad + \int_{\R} \abs{u_0 -c} \psi(x,0)\,dx -  \int_{\R \setminus \Omega} \int_0^T \sgn{u-c}k'(x) \,f(c)\,\psi \,dx\,dt  \ge 0.
\end{align*}
\end{lemma}

\begin{proof}  
A simple manifestation of the above Lemma~\ref{thm:entropy1} for the specific entropy $\eta(u) =\abs{u-c}$ and consequently the entropy flux function $Q(u) =\sgn{u-c} (f(u) -f(c))$ essentially completes the proof.
\end{proof}

\begin{lemma}
\label{thm:entropy12}
Let $u(x,t)$ be a weak solution constructed as the limit of the approximation $u_{\Dx}$ generated
by the scheme \eqref{eq:scheme} with $ \mu(\Dx)= \scalebox{1.1}{$\scriptstyle\mathcal{O}$}({\Dx^2})$. Let $0 \le \psi \in \mathcal{D}(\R \times [0,T]) $. Then the following entropy inequality is satisfied for all $c \in \R$
\begin{align*}
\int_{\R} \int_0^T & \left( \abs{u-c} \psi_t + \sgn{ u-c}  k(x) \left(f(u) -f(c) \right) \psi_x \right) \,dx\,dt 
 + \int_{\R} \abs{u_0 -c} \psi(x,0)\,dx \\
 & \qquad \qquad + \abs{f( c)} \int_{\R \setminus \Omega} \int_0^T \abs{k'(x)} \psi \,dx\,dt 
 +  \sum_{m=1}^{M} \int_0^T \abs{ f(c) (k_m^{+} - k_m^{-})} \psi(\xi_m, t) \,dt \ge 0.
\end{align*}
\end{lemma}

\begin{proof}
A straightforward adaptation of \cite[Lemma 4.2]{kenneth1} along with the help of Lemma~\ref{thm:entropy1} concludes the proof.
\end{proof}

To proceed further, we combine above two lemma's. In what follows, we have the following important theorem:

\begin{theorem}
\label{thm:theorem2}
Let $u(x,t)$ be a weak solution constructed as the limit of the approximation $u_{\Dx}$ generated
by the scheme \eqref{eq:scheme} with $ \mu(\Dx)= \scalebox{1.1}{$\scriptstyle\mathcal{O}$}({\Dx^2})$. 
Let $0 \le \psi \in \mathcal{D}(\R \times [0,T]) $. Then the following
entropy inequality is satisfied for all $c \in \R$
\begin{align*}
\int_{\R} \int_0^T & \left( \abs{u-c} \psi_t + \sgn{ u-c}  k(x) \left(f(u) -f(c) \right) \psi_x \right) \,dx\,dt 
 + \int_{\R} \abs{u_0 -c} \psi(x,0)\,dx \\
 & \qquad \quad +  \int_{\R \setminus \Omega} \int_0^T \sgn{u-c} k'(x) \,f(c)\,\psi \,dx\,dt 
+  \sum_{m=1}^{M} \int_0^T \abs{ f(c) (k_m^{+} - k_m^{-})} \psi(\xi_m, t) \,dt \ge 0.
\end{align*}
\end{theorem}

\begin{proof}
A verbatim copy of the proof of \cite[Lemma 4.4]{kenneth1} ensures the proof.

\end{proof}

\subsection{Uniqueness of Entropy Solutions}
Following \cite{kenneth1}, we mention that entropy solutions are unique under a crossing condition. 
It is well known that one has to impose the crossing condition only because the entropy inequality \eqref{eq:ent} alone is not
sufficient to guarantee uniqueness when the crossing condition is violated. However, we mention that
in the multiplicative case $f(k, u) = k(x)\, f(u)$ there is no flux crossing, hence we don't assume any such crossing condition.

One more technical issue is the existence of traces along the discontinuity curves $\omega_m$, for
$m=1,2,\cdots,M$. Without going into the details of it, we simply mention one instance where we
automatically have the existence of strong traces. We remark that, due to the genuinely nonlinearity
assumption, existence of traces is guaranteed when $k(x)$ is constant on each region $\mathcal{R}_m$, for $m=1,2,\cdots,M$.
This is a consequence of a general result by 
Vasseuer \cite{vasseuer}. For more general case, we encourage the readers to
consult \cite{kenneth1} for general assumptions ensuring the existence of traces.

To this end, we collect all the above mentioned result in the following theorem.

\begin{theorem}\label{thm:main}
Assume that the assumptions ~\ref{def:w1}, \ref{def:w2}, \ref{def:w3}, \ref{def:w4}, and \ref{def:b1} hold.
Moreover, suppose that the addition regularity assumptions on the discontinuous coefficient $k$ holds,
and the traces along the discontinuity curves exist. Then a unique entropy solution to the Cauchy
problem \eqref{eq:discont} exists and the entire computed sequence of approximations ${\lbrace u_{\Dx} \rbrace}_{\Dx>0}$, generated by the scheme \eqref{eq:scheme}
with $ \mu(\Dx)= \scalebox{1.1}{$\scriptstyle\mathcal{O}$}({\Dx^2})$, converges to the unique entropy solution of \eqref{eq:discont}.
\end{theorem}

%##############################################################################

\section{Diffusive-Dispersive Approximations}
\label{ap:dif_dis}

As we already pointed out in the introduction, 
all the aforementioned techniques can be utilized  
to analyze the equation given by the diffusive dispersive approximations 
of scalar conservation laws with a discontinuous flux of the type
\begin{equation}
  \label{eq:system_a}
  \begin{cases}
     u^{\eps}_t + f (k(x),u^{\eps} )_x 
     = \eps \beta \,u^{\eps}_{xx} + \mu(\eps) \gamma \,u^{\eps}_{xxx}, &\ \ x \in  \R \times (0,T),\\
    u^{\eps}(x,0)=u_0(x), &\ \ x \in \R,
  \end{cases}
\end{equation}
when $\eps >0$ tends to zero with $0<\mu(\eps) \mapsto 0$ as $\eps \mapsto 0$.
Here $T>0$ is fixed, $\beta, \gamma >0$ are fixed parameters, $u^{\eps}: \R \times
[0,T) \mapsto \R$ is the unknown scalar map, 
$u_0$ the initial data, $k: \R \mapsto \R$ is a spatially varying coefficient, and the flux function $f : \R^2 \mapsto \R$
is a sufficiently smooth scalar function.

We propose the following fully-discrete (in space and time) finite difference scheme 
approximating the limiting solutions generated by the equation \eqref{eq:system_a}
\begin{align}
D^t_{+} u^n_j + \Dm \hp^n &= \beta \Dx \Dp\Dm u^n_j  + \gamma \mu(\Dx) \Dp \Dm^2 u^n_j, \quad j \in \Z, \, n \in \N_0, \label{eq:scheme_a}\\
u^0_j &=\frac{1}{\Dx}\int_{\xm}^{\xp}u_0(\theta) d\theta,\quad j \in \Z, \label{eq:scheme_initial_a}
\end{align}
where $\beta, \gamma >0$ are fixed parameters, and $\mu(\Dx) \mapsto 0$ as $\Dx \mapsto 0$. 
As before, we will either use $\mu(\Dx)=\mathcal{O}(\Dx^2)$ 
or $ \mu(\Dx)= \scalebox{1.1}{$\scriptstyle\mathcal{O}$}({\Dx^2})$ depending on the
quest for the convergence of approximate solution $u_{\Dx}$ towards a weak solution or the entropy solution, respectively. Note that in this case, in contrast to the a priori estimates in section~\ref{sec:energy}, we have 
the following estimates

\begin{lemma}
\label{lemma2_a}
Let $u_{\Dx}$ be a sequence of approximations generated by the scheme 
\eqref{eq:scheme_a}-\eqref{eq:scheme_initial_a}. 
Moreover, assume that the initial data $u_0$ lies in $L^2(\R)$. Then  
the following estimate holds
\begin{equation}
\label{eq:std_a}
\begin{aligned}
\frac{1}{2} D^t_{+} \norm{u^n}^2 +\delta \frac{\gamma \Dx \mu(\Dx)}{2}\norm{\Dm^2 u^n}^2 +\delta \Dx \norm{\Dm u^n}^2 \le C,
\end{aligned}
\end{equation}
provided $\Dt$ and $\Dx$ satisfies the following CFL condition 
\begin{align}
\label{eq:cfl_a}
\max{ \Bigg\{ 2 \lambda \left(\beta^2 \frac{\Dx^2}{\gamma \mu(\Dx)} + 2\frac{\gamma \mu(\Dx)}{\Dx^2}\right), \,\, 8\lambda\norm{k}^2\norm{f^\prime}^2 \Bigg\} } \le \min{(1 -\delta,\beta-\delta)}, \quad \delta \in (0,\min{(1,\beta)}),
\end{align}
where $\lambda= \Dt/\Dx$ and the constant $C>0$ is independent of $\Dx$. 

In particular, the estimate \eqref{eq:std_a} guarantees following space-time estimates:
\begin{subequations}
\begin{align}
\forall n \in \N, \quad \Dx \sum_{j} (u^n_j)^2 & \le C,  \label{eq:est_final_a}\\
 \Dx^2\Dt \sum_{j} \sum_{n} (\Dm u^n_j)^2  &\le C, \label{eq:est_final1_a}\\
 \Dt \Dx^2 \mu(\Dx) \sum_{j} \sum_{n} (\Dm^2 u^n_j)^2 &\le C.\label{eq:est_final2_a}\\
  \Dx^2 \Dt \sum_{j} \sum_{n} (D^t_{+} u^n_j)^2 &\le C, \label{eq:est_final3_a}
\end{align}
\end{subequations}
\end{lemma}

\begin{proof}
To start with, we multiply the scheme \eqref{eq:scheme_a} by $\Dx \,u^n_j$ 
and subsequently sum over $j \in \Z$. Then, using summation-by-parts formula and the identity \eqref{eq:iden}, we obtain
\begin{align*}
\frac12 D^t_{+} \sum_j \Dx (u^n_j)^2  &- \frac{\Dt}{2} \sum_j \Dx \left( D^t_{+} u^n_j \right)^2\,  + \, \underbrace{\Dx \sum_ju^n_j \Dm\hp^n}_{\mathcal{I}_{\Dx}(f)} \\
& \qquad \qquad \qquad \qquad  = \, -\beta\Dx \sum_j \Dx \abs{\Dm u_j}^2 \, -  \gamma \mu(\Dx) \Dx \sum_j \Dm u^n_j  \Dm (\Dm u^n_j).
\end{align*}
Using the identity 
\begin{equation}
\label{eq:iden2}
\begin{aligned}
u^n_j \Dm u^n_j = \frac{1}{2} \Dm (u^n_j)^2 + \frac{\Dx}{2} (\Dm u^n_j)^2
\end{aligned}
\end{equation}
which is very similar to \eqref{eq:iden}, we get
\begin{align*}
\gamma \mu(\Dx) \Dx \sum_j \Dm u^n_j  \Dm (\Dm u^n_j) &= \underbrace{\gamma \frac{\Dx \mu(\Dx)}{2} \sum_j \Dm (\Dm u^n_j)^2}_{=0} 
+ \frac{\gamma \Dx^2 \mu(\Dx)}{2} \sum_j \left(  \Dm^2 u^n_j \right)^2.
\end{align*}
Furthermore, it has already been shown in the proof of Lemma \ref{lemma2} that
\[
- \mathcal{I}_{\Dx}(f) \leq C
\] 
where $C$ is independent of $\Dx$. Thus, we have the estimate
\begin{equation}
\label{eq:rel1}
\frac{1}{2} D^t_{+} \norm{u^n}^2 \le  \frac{\Dt}{2} \norm{D^t_{+} u^n}^2 - \beta \Dx  \norm{\Dm u^n}^2 -\frac{\gamma \Dx \mu(\Dx)}{2}\norm{\Dm^2 u^n}^2+ C
\end{equation}
Noting that $\Dx \Dp \Dm^2 = \Dm(\Dp - \Dm)$, we use the scheme \eqref{eq:scheme_a} to get the relation
\begin{equation}
\label{eq:rel2}
\begin{aligned}
\frac{\Dt}{2} \norm{D^t_{+} u^n}^2 \le \Dt \norm{\Dm \hp^n}^2 &+ \beta^2 \Dt \Dx^2 \norm{\Dp \Dm u^n}^2 \\
 \qquad \qquad &+ \frac{\gamma^2 \Dt \mu(\Dx)^2}{\Dx^2} \norm{\Dm \Dp u^n}^2 +  \frac{\gamma^2 \Dt \mu(\Dx)^2}{\Dx^2} \norm{\Dm^2 u^n}^2
\end{aligned}
\end{equation}
Now
\begin{equation*}
\begin{aligned}
\Dt \norm{\Dm \hp^n}^2 &= \lambda \sum_j \left(\kp \hat{f}^n_{j+\frac12} - \km \hat{f}^n_{j-\frac12}\right)^2\\
&=\lambda \sum_j \left(\kp \hat{f}^n_{j+\frac12} - \km\hat{f}^n_{j+\frac12} + \km\hat{f}^n_{j+\frac12} - \km  \hat{f}^n_{j-\frac12}\right)^2\\
&\le2 \lambda\sum_j \left(\kp - \km \right)^2 (\hat{f}^n_{j+\frac12})^2 + (\km)^2\left(\hat{f}^n_{j+\frac12} - \hat{f}^n_{j-\frac12}\right)^2\\
\end{aligned}
\end{equation*}
Since EO flux is Lipschitz continuous with Lipschitz constant $\norm{f'}_{\infty}$ (c.f. \eqref{eq:Lip}), we conclude that 
\begin{align*}
\displaystyle{\sup_{j} \sup_{n} \abs{\hat{f}^n_{j+\frac12}}}\leq \norm{f}_{\infty}, \,\, \text{and} \,\,
\abs{\hat{f}^n_{j+\frac12}-\hat{f}^n_{j-\frac12}} \leq \norm{f'}_{\infty} \left( \abs{u^n_j-u^n_{j-1}} + \abs{u^n_{j+1}-u^n_j} \right).
\end{align*}
Thus, we get
\begin{equation}
\label{eq:rel3}
\begin{aligned}
\Dt \norm{\Dm \hp^n}^2 &\le2 \lambda \norm{k}\norm{f}^2 \abs{k}_{BV} +  4\Dt \norm{k}^2\norm{f^\prime}^2 \left(\norm{\Dm u^n}^2 + \norm{\Dp u^n}^2 \right)\\
\end{aligned}
\end{equation}
Using relations \eqref{eq:rel1},\eqref{eq:rel2} and \eqref{eq:rel3} in tandem with the fact that $\norm{\Dm(.)^n} = \norm{\Dp(.)^n}$, we get the estimate
\begin{equation}
\label{eq:rel4}
\begin{aligned}
\frac{1}{2} D^t_{+} \norm{u^n}^2 +& \left( 1  - 2 \frac{\beta^2 \lambda \Dx^2}{\gamma \mu(\Dx)} - 4 \frac{\gamma\lambda  \mu(\Dx)}{\Dx^2}\right) \frac{\gamma \Dx \mu(\Dx)}{2}\norm{\Dm^2 u^n}^2 \\
& \qquad \qquad \qquad \qquad \qquad + \left(\beta - 8 \lambda \norm{k}^2 \norm{f^\prime}^2 \right) \Dx \norm{\Dm u^n}^2 \le C
\end{aligned}
\end{equation}
Choosing $\lambda$ in accordance to \eqref{eq:cfl_a} ensures that for some $\delta \in (0,\min{(\beta,1)})$
\begin{align*}
1  - 2 \frac{\beta^2\lambda \Dx^2}{\gamma \mu(\Dx)} - 4 \frac{\gamma \lambda \mu(\Dx)}{\Dx^2} > \delta, \,\,\, \text{and} \,\, \beta - 8 \lambda \norm{k}^2 \norm{f^\prime}^2 > \delta,
\end{align*}
thus leading to the estimate \eqref{eq:std_a}.

In order to prove estimates \eqref{eq:est_final_a}, \eqref{eq:est_final1_a} and  \eqref{eq:est_final2_a}, we multiply the inequality \eqref{eq:std_a} by $\Dt$ and subsequently sum over all $n=0,1,\cdots, N-1$ to reach 
\begin{equation*}
\begin{aligned}
\frac{1}{2} \norm{u^N}^2 & + \delta \frac{\gamma \Dx \Dt \mu(\Dx)}{2} \sum_{n}\norm{\Dm^2 u^n}^2  + \delta \Dx\Dt \sum_{n} \norm{\Dm u^n}^2  \le \frac{1}{2} \norm{u_0}^2 + C,
\end{aligned}
\end{equation*}
This essentially finishes the proof of the a priori bounds \eqref{eq:est_final_a}, \eqref{eq:est_final1_a} and  \eqref{eq:est_final2_a} . Using these estimates and the relation \eqref{eq:rel2}, a direct computation shows that \eqref{eq:est_final3_a} holds. This completes the proof.
\end{proof}

%%%%%%%%%%%%%%%%%%%%%%%%%%%%%%%%%%%%%%%%%%%%%%%%%%%%%%%%%%%%%%%%%%%%%%%

Note that, making use of these a priori estimates, 
an appropriate ``convergence theorem'' can be formulated, which guarantees the convergence of
approximate solutions ${\lbrace u_{\Dx} \rbrace}_{\Dx>0}$, 
generated by the scheme \eqref{eq:scheme_a}-\eqref{eq:scheme_initial_a}, to 
a weak solution of \eqref{eq:discont}.

\begin{theorem}
\label{thm:theorem1_a}
Let $u_{\Dx}$ be a sequence of approximations generated via the scheme \eqref{eq:scheme_a}-\eqref{eq:scheme_initial_a} with $\mu(\Dx) = \bigO{\Dx^2}$.
Then there exists a function $u \in L^{\infty}([0,T];L^1_{\loc}(\R))$ and a sequence $\lbrace \Dx_j \rbrace$ of
$\lbrace \Dx \rbrace$ such that $u_{\Dx_j} \mapsto u$ as $\Dx_j \downarrow 0$. Moreover, the function $u$
is a weak solution to \eqref{eq:discont}.
\end{theorem}

\begin{proof}
The proof of this theorem is very much similar to the proof of the 
Theorem~\ref{thm:theorem1}, except the analysis of the terms involving $\Dp \Dm^2 u^n_j$. 
However, these term can be treated like the term $\Dpt\Dp \Dm u^n_j$ 
in section~\ref{sec:convergence}, but we need to use the a priori bound \eqref{eq:est_final2_a} instead of \eqref{eq:est_final3}. For brevity of exposition, we omit the details of the proof.
\end{proof}
%%%%%%%%%%%%%%%%%%%%%%%%%%%%%%%%%%%%%%%%%%%%%%%%%%%%%%%%%%%%%%%%%%%%%%%
A similar result, in view of the analysis in section~\ref{sec:entropy} 
and above priori estimates, can be obtained regarding 
the convergence of ${\lbrace u_{\Dx} \rbrace}_{\Dx>0}$, 
generated by the scheme \eqref{eq:scheme_a}-\eqref{eq:scheme_initial_a}, 
to the unique entropy solution of \eqref{eq:discont}. 
\begin{theorem}
\label{thm:theorem2_a}
Let $u(x,t)$ be a weak solution constructed as the limit of the approximation $u_{\Dx}$ generated
by the scheme \eqref{eq:scheme_a}-\eqref{eq:scheme_initial_a} with 
$ \mu(\Dx)= \scalebox{1.1}{$\scriptstyle\mathcal{O}$}({\Dx^2})$. 
Let $0 \le \psi \in \mathcal{D}(\R \times [0,T]) $. Then the following
entropy inequality is satisfied for all $c \in \R$
\begin{align*}
\int_{\R} \int_0^T & \left( \abs{u-c} \psi_t + \sgn{ u-c}  k(x) \left(f(u) -f(c) \right) \psi_x \right) \,dx\,dt 
 + \int_{\R} \abs{u_0 -c} \psi(x,0)\,dx \\
 & \qquad +  \int_{\R \setminus \Omega} \int_0^T \sgn{u-c} k'(x) \,f(c)\,\psi \,dx\,dt 
+  \sum_{m=1}^{M} \int_0^T \abs{ f(c) (k_m^{+} - k_m^{-})} \psi(\xi_m, t) \,dt \ge 0.
\end{align*}
\end{theorem}
\begin{proof}
This can be achieved using similar arguments used in the proof of Theorem~\ref{thm:theorem2}.
\end{proof}

%############################################################

\section{Numerical Experiments}
\label{sec:numerical}
We present a few numerical results to substantiate the results we have shown in the previous sections. We consider the capillarity problem approximated by the scheme \eqref{eq:scheme}, as well as the diffusive-dispersive problem approximated by \eqref{eq:scheme_a}. Note that while the former is approximated by an implicit type scheme, the latter is an explicit-in-time scheme. 

\subsection{Capillarity approximation}
We consider the flow of two phases in a heterogeneous porous medium, in the limit of vanishing dynamic capillary pressure. The model equation is given by
\begin{align}
\label{eq:cap}
 u_t + f (k(x),u )_x = \eps \beta \,(g(k(x),u) u_x)_x + \mu(\eps) \gamma \,(h(k(x),u) u_{xt})_x, &\ \ x \in  \R \times (0,T),
\end{align}
where $g,h: \R^2 \rightarrow \R$ are assumed to be smooth functions such that
\begin{align*}
\alpha  \le g(.,.),h(.,.)
\end{align*}
for some constant $\alpha > 0$. This model has also been considered and numerically analysed in \cite{sid}, \cite{Kissling} and \cite{E. vanDuijn}. Note that, we have theoretically shown 
the convergence for the special case when $g \equiv h \equiv 1$. 
However, we mention that this cosmetic changes in the
equation has no effect on the central idea of the paper and a straightforwardly incremental 
modification of our analysis can be adopted to analyze the equation \eqref{eq:cap}. We have decide to work
with this equation because of the availability of results for such equations which help
us to compare our results.

As done in \cite{sid}, we choose the various quantititis in the model as follows:
\begin{equation*}
\begin{aligned}
f(k(x),u) &= \frac{z^w( 1 - k(x)z^o)}{z^w + z^o}\\
g(k(x),u) & = k(x)\,g_1(u) = k(x)  \frac{z^w z^o}{z^w + z^o} P^\prime(u)\\
h(k(x),u) & = k(x)\,h_1(u) = k(x) \frac{ z^w z^o}{z^w + z^o} \\
P(u)       &= \left(u^{-\frac{4}{3}} - 1 \right)^\frac{1}{4}
\end{aligned}
\end{equation*}
where $z^w = u^2$ and $z^o = (1 - u)^2$. From the physical point of view, $u(x,t)$ and $(1 - u(x,t))$ represent the water and oil saturations respectively, while $k(x)$ corresponds to the rock  permeability. The corresponding modified numerical approximation is chosen to be
\begin{equation*}
\begin{aligned}
D^t_{+} u^n_j + \Dm \hp^n &= \beta \Dx \Dp \left( \km  \frac{\left(g_1(u^n_j) + g_1(u^n_{j-1})\right)}{2} \Dm u^n_j\right)  \\
& \qquad \qquad \qquad \qquad+ \gamma \mu(\Dx)  \Dp \left( \km  \frac{\left(h_1(u^n_j) + h_1(u^n_{j-1})\right)}{2} D^t_{+}\Dm u^n_j\right) 
\end{aligned}
\end{equation*}
As stated earlier, we can replace the EO numerical flux with any other monotone flux, with a similar covergence analysis following through. For simplicity, we choose the Lax Friedrichs flux for the present model. The spatial domain is $[0,2]$ with an initial solution profile 
\begin{align*}
u_0(x) = \begin{cases}
              0.8 \qquad \text{for } x\le0.25,\\
              0.2 \qquad \text{for } x>0.25
              \end{cases}
\end{align*} 
Furthermore, we choose $\beta = 6$, $\gamma=36$ and CFL $0.3$, with the final time of simulation being $T = 0.6$. 

\subsubsection{Continuous flux} We first consider the scenario when $k \equiv 1$. This corresponds to the flow in a homogeneous rock structure. When the dispersion coefficient is chosen as $\mu(\Dx) = \Dx^2$, the numerical solution converges to a weak solution of \eqref{eq:cap}, as shown in Figure \ref{fig:cap_ref}. The weak solution consists of a leading classical shock wave and a trailing non-classical shock wave, with an intermediate state in between. However, when the dispersion coefficient is chosen to $\mu(\Dx) = \Dx^3$, the solution converges to the entropy solution, as shown in Figure \ref{fig:cap_k_const}. This is in accordance with Theorem \ref{thm:main}.

\begin{figure}
\begin{center}
\includegraphics[width=0.60\textwidth]{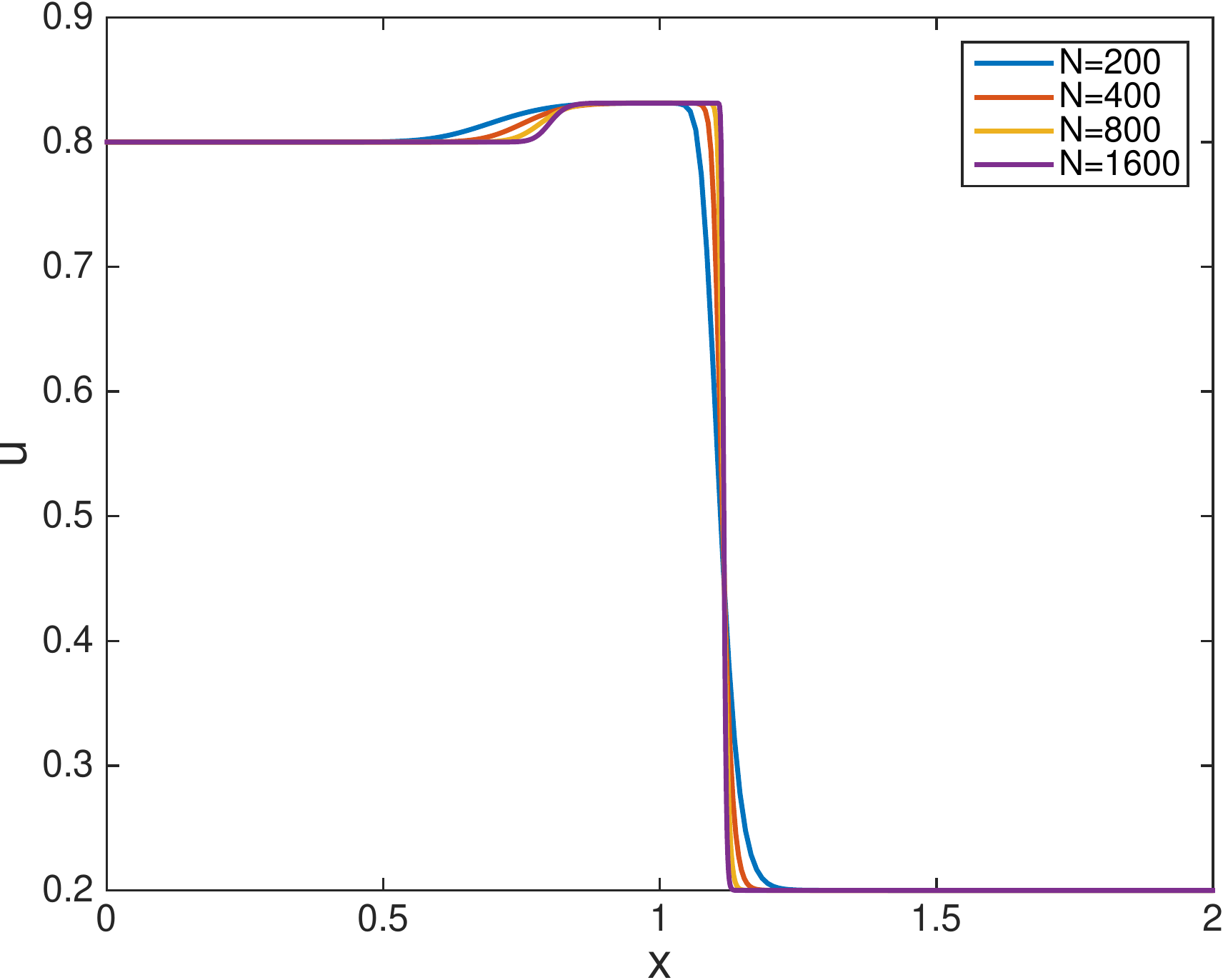}
\caption{Two phase flow through a homogeous medium ($k(x) \equiv 1$) at time $T=0.6$ with $\mu(\Dx) = \Dx^2$. Mesh refinement study indicates convergence to a non-classical weak solution of the underlying conservation law.}
\label{fig:cap_ref}
\end{center}
\end{figure}

\begin{figure}
\begin{center}
\includegraphics[width=0.60\textwidth]{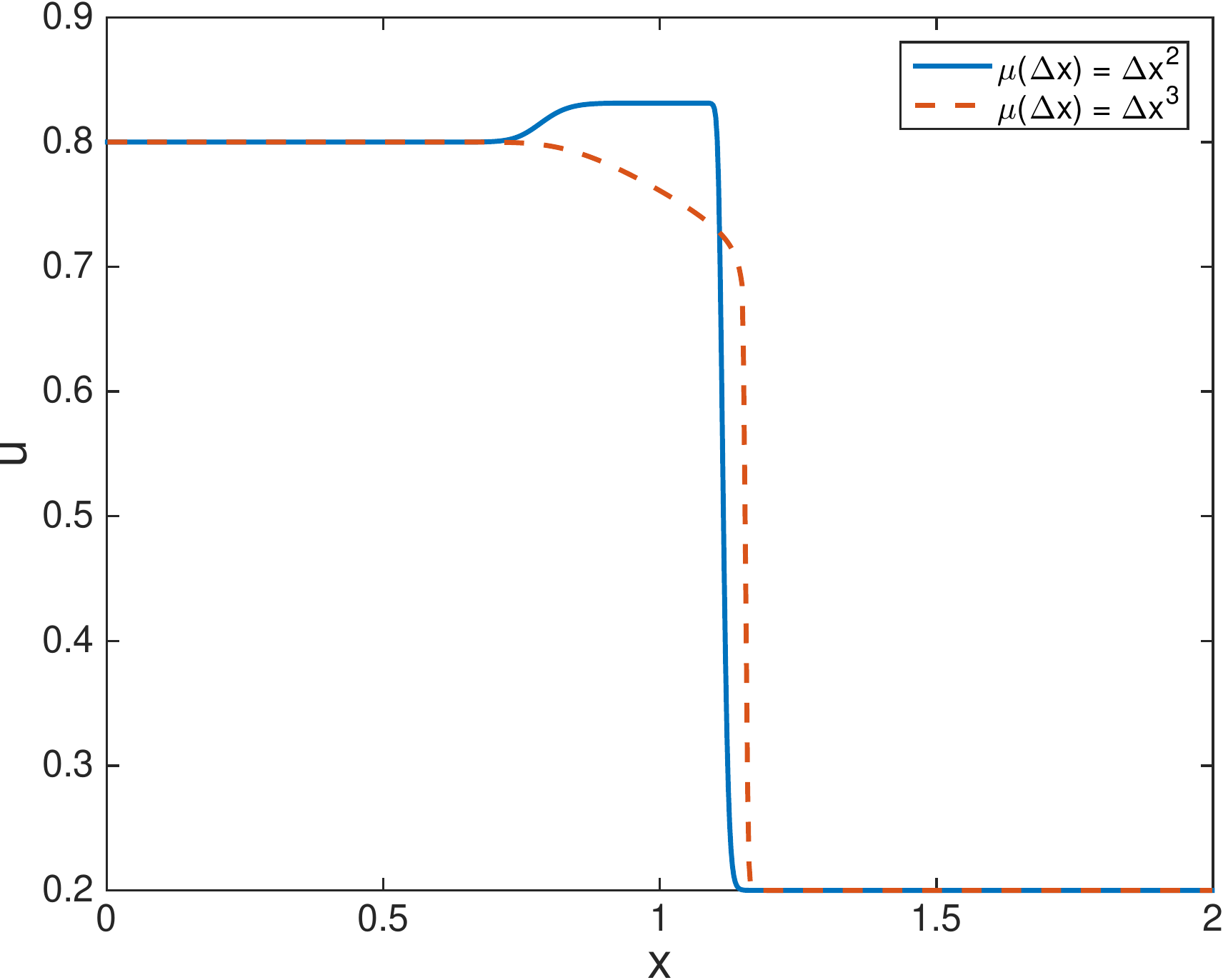}
\caption{Two phase flow through a homogeous medium ($k(x) \equiv 1$) at time $T=0.6$. A non-classical weak solution is obtained for $\mu(\Dx) = \Dx^2$, while the unqiue entropy solution is obtained for $\mu(\Dx) = \Dx^3$.}
\label{fig:cap_k_const}
\end{center}
\end{figure}

\subsubsection{Discontinuous flux} We next consider the scenario depicting the flow through a heterogeneous medium. We chose the rock permeability as
\begin{align*}
k(x) =  \begin{cases}
              1.1 \qquad \text{for } x\le0.6\\
              1.4 \qquad \text{for } x>0.6
              \end{cases}
\end{align*} 
which corresponds to two rock types with a sharp interface at $x=0.6$. As before we first consider the solution by setting $\mu(\Dx) = \Dx^2$, which is shown in Figure \ref{fig:cap_k_var}. The numerical approximation is a weak solution consisting of a leading classical shock wave and a trailing non-classical shock wave separated by an intermediate state, and discontinuity at $x=0.6$ corresponding to rock structure. Once again,the non-classical shock disappears when the dispersion coefficient is chosen as $\mu(\Dx) = \Dx^{3}$, with the solution approximating the entropy solution.

\begin{figure}
\begin{center}
\includegraphics[width=0.60\textwidth]{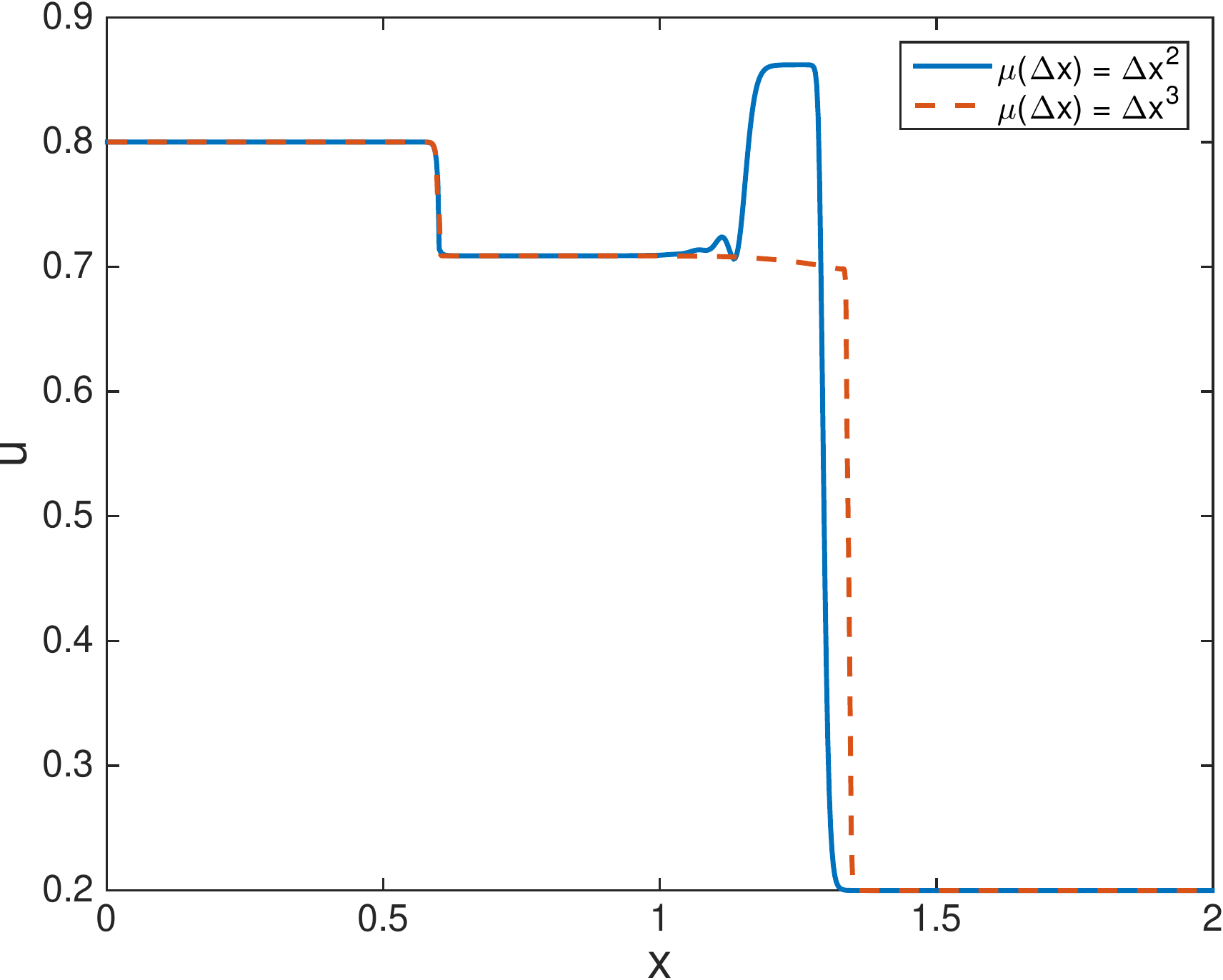}
\caption{Two phase flow through a hetergeneous medium at time $T=0.6$. A non-classical weak solution is obtained for $\mu(\Dx) = \Dx^2$, while the unqiue entropy solution is obtained for $\mu(\Dx) = \Dx^3$.}
\label{fig:cap_k_var}
\end{center}
\end{figure}

\subsection{Diffusive-dispersive model}
As we have already mentioned, the convergence analysis for the diffusive-dispersive equation \eqref{eq:system_a}
is very similar to the other. We work with the flux function
\begin{equation*}
f(k(x),u) = k(x)(u^3 - u)
\end{equation*}
and the numerical scheme given by \eqref{eq:scheme_initial_a}. The continuous flux version of this problem has been studied numerically in \cite{chalons} as well. The spatial domain is $[-0.5,0.5]$ with the final simulation time being $T=0.01$. The initial profile of the solution for this problem is taken to be
\begin{align*}
u_0(x) = \begin{cases}
              4 \qquad \text{for } x\le0,\\
              -2 \qquad \text{for } x>0
              \end{cases}
\end{align*} 
with the model parameters chosen as $\beta = 5$ and $\gamma = 20$. For this problem, the EO numerical flux is used.

\subsubsection{Continuous flux} We first work with a continuous flux by choosing $k(x) \equiv 1$. For $\mu(\Dx) = \Dx^2$, the numerical solution converges to a weak solution (see Figure \ref{fig:dd_ref}) consisting of trailing classical shock wave and a leading non-classical shock wave separated by an intermediate state. Figure \ref{fig:dd_k_cont} shows that for the choice $\mu(\Dx) = \Dx^{2.5}$, we get an approximation to the unique entropy solution, which corresponds to a single shock wave satisfying Ole\u{i}nik's entropy condition \cite{Oleinik}. 

\begin{figure}
\begin{center}
\includegraphics[width=0.60\textwidth]{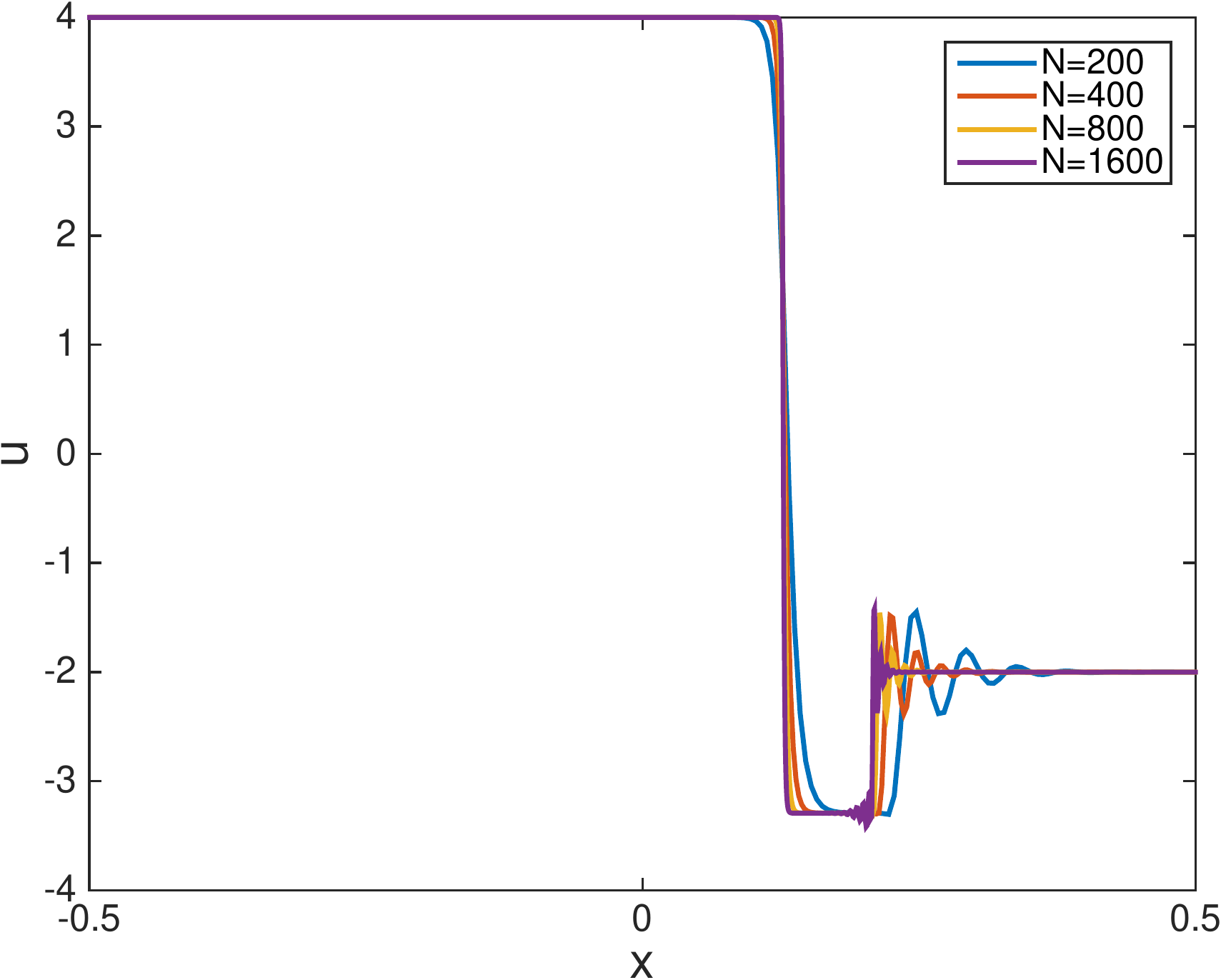}
\caption{Diffusive-dispersive model with $k(x) \equiv 1$, at time $T=0.01$ and $\mu(\Dx) = \Dx^2$. Mesh refinement study indicates convergence to a non-classical weak solution of the underlying conservation law.}
\label{fig:dd_ref}
\end{center}
\end{figure}

\begin{figure}
\begin{center}
\includegraphics[width=0.60\textwidth]{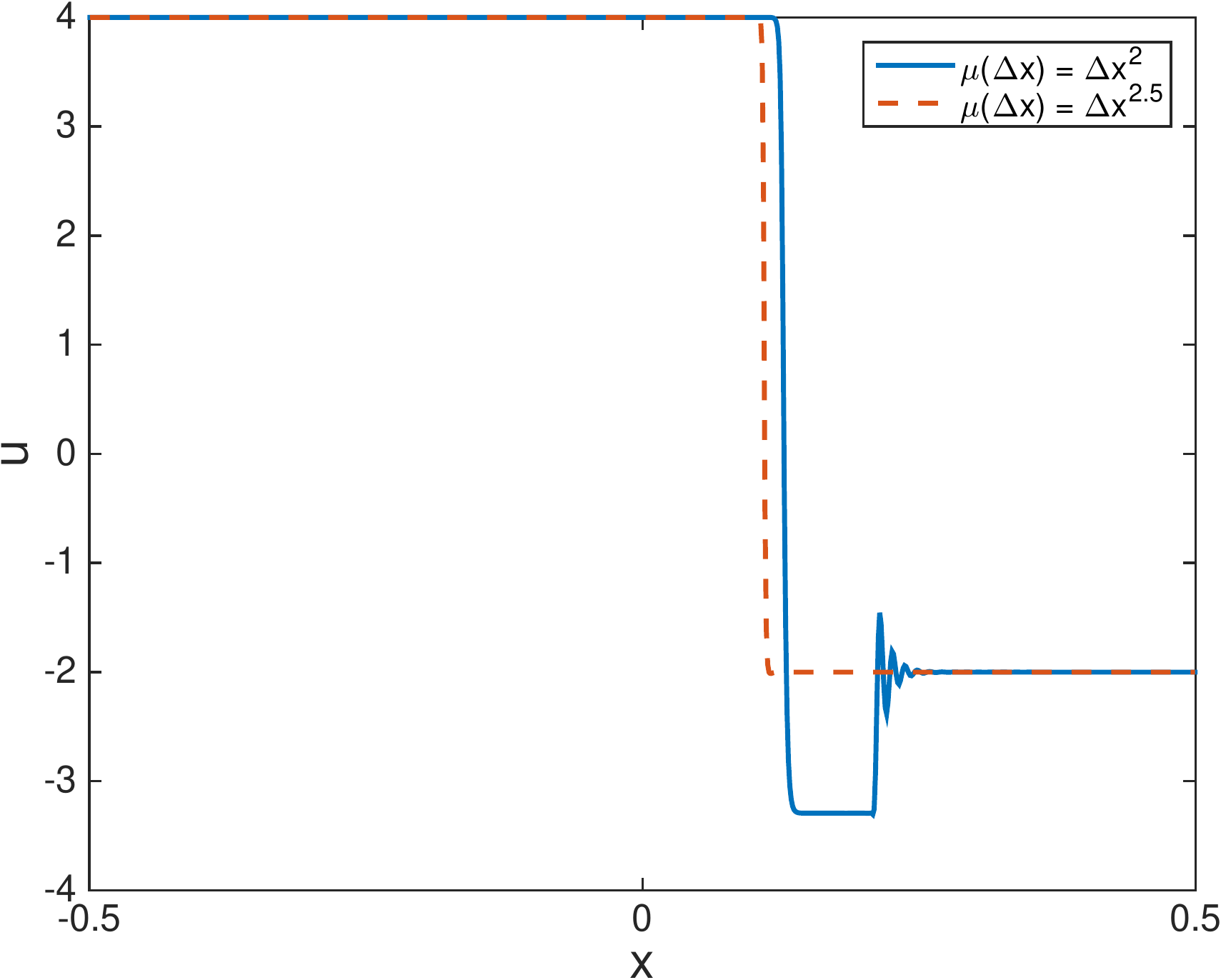}
\caption{Diffusive-dispersive model with $k(x) \equiv 1$, at time $T=0.01$. A non-classical weak solution is obtained for $\mu(\Dx) = \Dx^2$, while the unqiue entropy solution is obtained for $\mu(\Dx) = \Dx^3$.}
\label{fig:dd_k_cont}
\end{center}
\end{figure}

\subsubsection{Discontinuous flux} We work with a discontinuous flux characterised by
\begin{align*}
k(x) =  \begin{cases}
              1.1 \qquad \text{for } x\le0.1\\
              0.9 \qquad \text{for } x>0.1
              \end{cases}
\end{align*} 
The numerical approximation shown in Figure \ref{fig:dd_k_var} with $\mu(\Dx) = \Dx^2$, is a weak solution consisting of a trailing classical shock wave and a leading non-classical shock wave separated by an intermediate state. In addition, there is a standing discontinuity at $x=0.1$ corresponding to discontinuiuty in the flux. The non-classical shock disappears when the dispersion coefficient is chosen as $\mu(\Dx) = \Dx^{2.5}$, with the solution approximating the entropy solution.
\begin{figure}
\begin{center}
\includegraphics[width=0.60\textwidth]{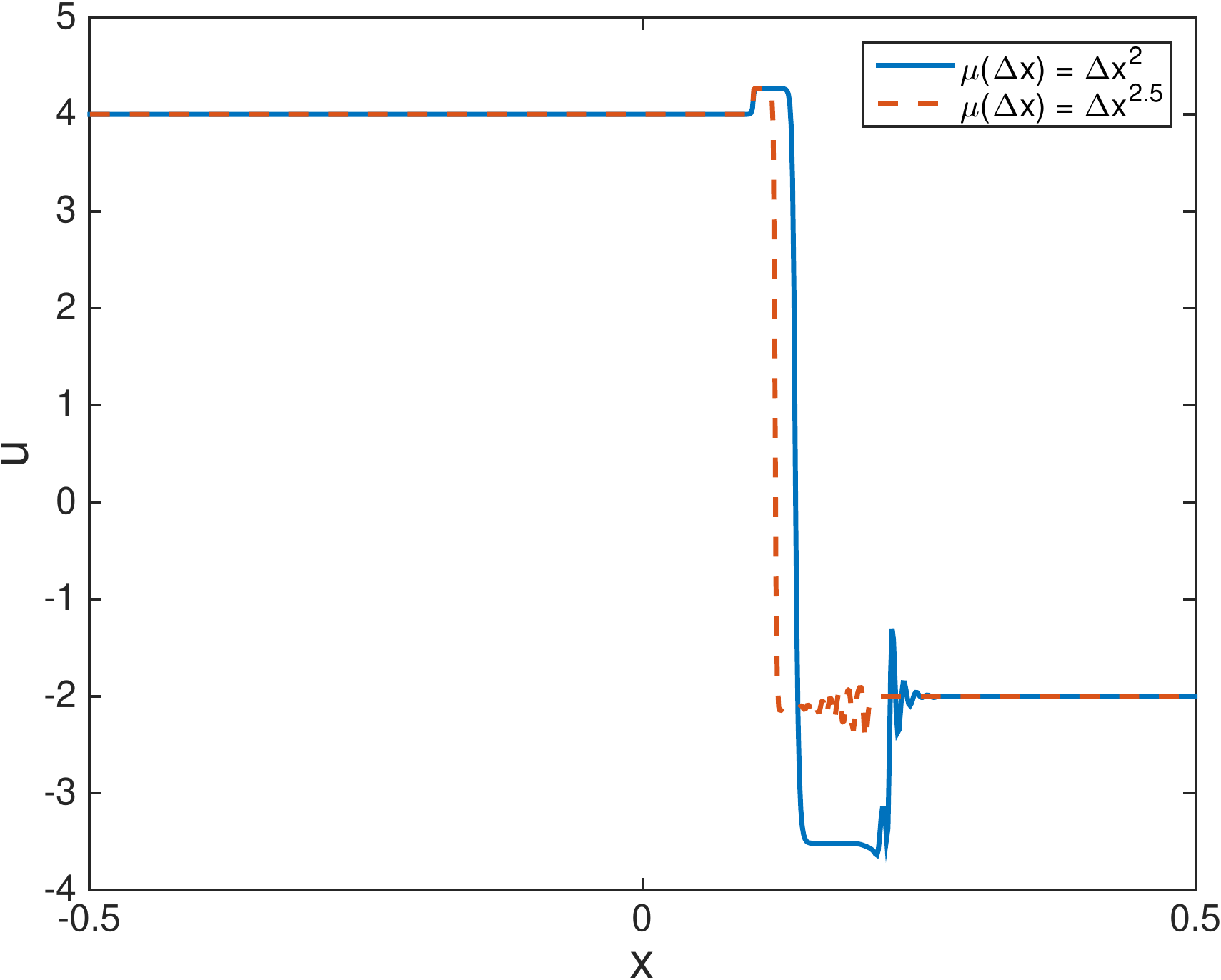}
\caption{Diffusive-dispersive model with a discontinuous flux, at time $T=0.01$. A non-classical weak solution is obtained for $\mu(\Dx) = \Dx^2$, while the unqiue entropy solution is obtained for $\mu(\Dx) = \Dx^3$.}
\label{fig:dd_k_var}
\end{center}
\end{figure}
%############################################################


\begin{thebibliography}{10}

\bibitem{aziz} K. Aziz, and A. Settari,
 \newblock Fundamentals of petroleum reservoir simulation.
  \newblock {\em Applied Science Publishers, London}, 1979.
  
\bibitem{chalons} C. Chalons, and P. G. Lefloch, 
\newblock A fully-discrete scheme for diffusive-dispersive conservation laws.
\newblock {\em Numerische Math.}, 89, 493-509, 2001.

\bibitem{Chen} G. -Q. Chen, \newblock Compactness methods and
  nonlinear hyperbolic conservation laws.  \newblock {\em In some
    current topics on nonlinear conservation laws},
  pp. 33-75. Amer. Math. Soc., Providence, RI, 2000.

  
\bibitem{sid} G. M. Coclite, L. di Ruvo, J. Ernest, and S. Mishra,
 \newblock Convergence of vanishing capillarity approximations 
 for scalar conservation laws with discontinuous fluxes.
  \newblock {\em Netw. Heterog. Media.}, 8(4), 969-984 (2013).


\bibitem{diperna} R. J. DiPerna, 
\newblock Convergence of approximate solutions to conservation laws.
  \newblock {\em Arch. Ration. Mech. Anal.}, 82, 27-70, 1983.
  
\bibitem{ernest} J. Ernest, P. G. Lefloch, and S. Mishra, 
\newblock Schemes with well-controlled dissipation (WCD). I: Non-classical shock waves,
\newblock {\em Siam. J. Numer. Math}, to appear, 2015.


\bibitem{hass} S. M. Hassanizadeh, and W. G. Gray, 
\newblock Mechanics and thermodynamics of multiphase flow in porous media including interphase
boundaries.
\newblock {\em Adv. Water Resour.}, 13, 169-186 (1990).

\bibitem{hayes} B. T. Hayes, and P. G. Lefloch,  
\newblock Nonclassical shock waves and kinetic relations. Strictly hyperbolic systems,
\newblock {\em Preprint no. 357, CMAP, Ecole Polytechnique,
Palaiseau, France, Nov. 1996}.

\bibitem{lefloch1} B. T. Hayes, and P. G. Lefloch, 
\newblock Nonclassical shocks and kinetic relations: strictly hyperbolic systems.
\newblock {\em Siam. J. Math. Anal.} 31(5), 941-991, 2000.


\bibitem{hou} T. Y. Hou, and P. G. Lefloch,  
\newblock Why nonconservative schemes converge to
wrong solutions. Error analysis,
\newblock {\em Math. of Comput.}, 62, 497-530, 1994.


\bibitem{hwang} S. Hwang, and A. E. Tzavaras, 
\newblock Kinetic decomposition of approximate solutions to conservations laws: 
Application to relaxation and diffusion-dispersion approximations.
\newblock {\em Commun. Partial Differ. Equ.}, 27(5-6), 1229-1254, 2002.

\bibitem{jacobs} D. Jacobs, W. R. McKinney, and M. Shearer, 
\newblock Traveling wave solutions of the modified
Korteweg-deVries Burgers equation.
\newblock {\em J. Differential Equations.}, 116, 448-467, 1995.

\bibitem{towers} K. H. Karlsen, and J. D. Towers,
\newblock Convergence of the Lax-Friedrichs scheme and stability for conservation 
laws with a discontinous space-time dependent flux,
\newblock {\em Chinese Ann. Math. Ser. B.}, 25(3), 287-318, 2004.

\bibitem{kenneth1} K. H. Karlsen, N. H. Risebro, and J. D. Towers,
\newblock $L^1$ stability for entropy solutions of nonlinear degenerate parabolic convection-diffusion equations with discontinuous coefficients.
\newblock {\em Preprint series. Pure mathematics}, http://urn. nb. no/URN: NBN: no-8076 (2003)

\bibitem{Triang_Cocliteetal} K. H. Karlsen, S. Mishra, and N. H. Risebro, 
  \newblock Convergence of finite volume schemes for triangular systems of conservation laws.
  \newblock {\em Numerische Mathematik.}, 111(4), 559-589, 2008.
  
\bibitem{Kissling} F. Kissling, and C. Rohde, 
\newblock The computation of nonclassical shock waves with a heterogeneous multi-scale method.
\newblock {\em Netw. Heterog. Media.}, 5(3), 661-674, 2010.

  
\bibitem{kondo} C. I. Kondo, and P. G. Lefloch, 
\newblock Zero diffusion-dispersion limits for scalar conservation laws.
\newblock {\em Siam. J. Math. Anal.}, 33(6), 1320-1329, 2002.

\bibitem{kruzkov} S. N. Kru\v{z}kov, \newblock First order
  quasilinear equations with several independent variables.  \newblock
  {\em Mat. Sb. (N.S).}, 81 (123), 1970, pp. 228-255.


\bibitem{lefloch2} P. G. Lefloch, 
\newblock Hyperbolic Systems of Conservation Laws: The Theory of Classical and 
Nonclassical Shock Waves.
\newblock {\em Lectures in Mathematics.}, (Birkh\"{a}user, Basel, 2002).

\bibitem{sid1} P. G. Lefloch, and S. Mishra,
\newblock Numerical methods with controlled dissipation for small-scale dependent shocks,
\newblock {\em Acta Numerica.}, 1-72, 2014.

\bibitem{lev} R. J. Leveque,  
\newblock Numerical Methods for Conservation Laws.
\newblock {\em Birkhauser Verlag, Boston}, 1992.

\bibitem{Lu} Y. Lu, \newblock Hyperbolic conservation laws and the
compensated compactness method.  \newblock {\em Vol 128 of Chapman
and Hall/CRC Monographs and surveys in Pure and Applied
Mathematics}, Chapman and Hall/CRC, Boca Raton, FL, 2003.

\bibitem{Murat} F. Murat, 
\newblock Compacite par compensation.
  \newblock {\em Ann. Scuola Norm. Sup. Pisa Cl. Sci (4)}, 5(3), 489-507, 1978.

\bibitem{Oleinik}
O.~A. Ole{\u\i}nik,
\newblock Convergence of certain difference schemes.
\newblock {\em Soviet Math. Dokl.}, 2:313--316, 1961.

\bibitem{perthame} B. Perthame, and P. E. Souganidis, 
\newblock A limiting case for velocity averaging.
\newblock {\em Ann. Sci. E.N.S.}, 31(4), 591-598, 1998.


\bibitem{schonbek} M. E. Schonbek, 
\newblock Convergence of solution to nonlinear dispersive equations.
\newblock {\em Commun. Partial Differ. Equ.}, 7(8), 959-1000, 1982.

\bibitem{Tartar} L. Tartar, \newblock Compensated compactness and
applications to partial differential equations.  \newblock {\em
Research notes in mathematics, nonlinear analysis, and
mechanics.}, Heriot-Symposium, vol. 4, 1979, pp. 136-212.

\bibitem{towers1} J. D. Towers,
\newblock Convergence of a finite difference scheme for conservation
laws with a discontinous flux,
\newblock {\em Siam. J. Numer. Anal.}, 38(2), 681-698, 2000.

\bibitem{E. vanDuijn} E. VanDuijn, L. A. Peletier, and S. Pop, 
\newblock A new class of entropy solutions of the Buckley-Leverett equation.
  \newblock {\em Siam. J. Math. Anal.}, 39(2), 507-536, 2007.


\bibitem{vasseuer} A. Vasseuer, 
\newblock Strong traces for solutions of multidimensional scalar conservation laws.
  \newblock {\em Arch. Ration. Mech. Anal.}, 160(3), 181-193, 2001.

\bibitem{Volpert}
A.~I. Vol'pert,
\newblock Generalized solutions of degenerate second-order quasilinear
  parabolic and elliptic equations.
\newblock {\em Adv. Differential Equations}, 5(10-12):1493--1518, 2000.

\bibitem{wu} C. C. Wu, 
\newblock New theory of MHD shock waves, in Viscous Profiles and Numerical Methods for
Shock Waves.
\newblock {\em M. Shearer ed., SIAM, Philadelphia.}, PA, 209–236, 1991.




\end{thebibliography}
\end{document}